\documentclass[10pt]{scrartcl}

\usepackage{tikz}
\usetikzlibrary{matrix,arrows,decorations.pathmorphing,shapes.geometric}

\usepackage{etex}
\usepackage{subcaption}
\usepackage[utf8]{inputenc}
\usepackage{a4wide}
\usepackage{graphicx}
\usepackage{wrapfig}
\usepackage[margin=3cm]{geometry}
\usepackage{dingbat}
\usepackage{amsmath,accents}
\usepackage{amsthm}
\usepackage{amssymb}

\usepackage{mathrsfs, mathtools}
\usepackage{overpic}

\usepackage[utf8]{inputenc}

\usepackage{tikz}
\usetikzlibrary{matrix,arrows,decorations.pathmorphing,shapes.geometric}
\usepackage{tikz-cd}
\usepackage{needspace}
\newcommand{\spaceplease}{\needspace{5\baselineskip}}
\usetikzlibrary{decorations.markings}

\usetikzlibrary{backgrounds}

\usetikzlibrary{decorations.markings}
\usetikzlibrary{backgrounds}

\usepackage{etex}

\tikzstyle{tikzfig}=[baseline=-0.25em,scale=0.5]

\pgfkeys{/tikz/tikzit fill/.initial=0}
\pgfkeys{/tikz/tikzit draw/.initial=0}
\pgfkeys{/tikz/tikzit shape/.initial=0}
\pgfkeys{/tikz/tikzit category/.initial=0}

\pgfdeclarelayer{edgelayer}
\pgfdeclarelayer{nodelayer}
\pgfsetlayers{background,edgelayer,nodelayer,main}

\tikzstyle{none}=[inner sep=0mm]

\newcommand{\tikzfig}[1]{%
	{\tikzstyle{every picture}=[tikzfig]
		\IfFileExists{#1.tikz}
		{\input{#1.tikz}}
		{%
			\IfFileExists{./#1.tikz}
			{\input{./#1.tikz}}
			{\tikz[baseline=-0.5em]{\node[draw=red,font=\color{red},fill=red!10!white] {\textit{#1}};}}%
	}}%
}

\tikzstyle{every loop}=[]

\usepackage{tikzit}

\tikzstyle{nc}=[fill=white, draw=black, shape=circle, minimum size=0.40cm]
\tikzstyle{nr}=[fill=white, draw=black, shape=rectangle, minimum width=0.60cm, minimum height=0.60cm]
\tikzstyle{nrg}=[fill={rgb,255: red,199; green,199; blue,199}, draw=black, shape=rectangle]
\tikzstyle{sr}=[fill=white, draw=black, shape=rectangle, minimum width=0.60cm, minimum height=0.40cm, inner sep=0pt]
\tikzstyle{lsr}=[fill=white, draw=black, shape=rectangle, minimum width=0.60cm, minimum height=0.40cm, inner sep=0pt, rotate=30]
\tikzstyle{rsr}=[fill=white, draw=black, shape=rectangle, minimum width=0.60cm, minimum height=0.40cm, inner sep=0pt, rotate=-30]
\tikzstyle{rrrsr}=[fill=white, draw=black, shape=rectangle, minimum width=0.60cm, minimum height=0.40cm, inner sep=0pt, rotate=-120]
\tikzstyle{ivsr}=[fill=none, draw=none, shape=rectangle, minimum width=0.60cm, minimum height=0.40cm, inner sep=0pt]
\tikzstyle{ivlsr}=[fill=none, draw=none, shape=rectangle, minimum width=0.60cm, minimum height=0.40cm, inner sep=0pt, rotate=30]
\tikzstyle{ivrsr}=[fill=none, draw=none, shape=rectangle, minimum width=0.60cm, minimum height=0.40cm, inner sep=0pt, rotate=-30]
\tikzstyle{ivrrrsr}=[fill=none, draw=none, shape=rectangle, minimum width=0.60cm, minimum height=0.40cm, inner sep=0pt, rotate=-120]
\tikzstyle{mr}=[fill=white, draw=black, shape=rectangle, minimum width=1.00cm, minimum height=0.50cm, inner sep=0pt]
\tikzstyle{product}=[fill={rgb,255: red,255; green,5; blue,80}, draw=black, shape=circle, inner sep=0pt, minimum size=0.15cm]
\tikzstyle{unit}=[fill={rgb,255: red,255; green,166; blue,217}, draw=black, shape=circle, inner sep=0pt, minimum size=0.15cm]
\tikzstyle{coproduct}=[fill={rgb,255: red,2; green,145; blue,255}, draw=black, shape=circle, inner sep=0pt, minimum size=0.15cm]
\tikzstyle{counit}=[fill={rgb,255: red,165; green,219; blue,255}, draw=black, shape=circle, inner sep=0pt, minimum size=0.15cm]
\tikzstyle{antipode}=[fill=white, draw=black, shape=circle, inner sep=0pt, minimum size=0.15cm]
\tikzstyle{region}=[fill={rgb,255: red,191; green,191; blue,191}, draw={rgb,255: red,191; green,191; blue,191}, shape=circle, minimum size=0.80cm]
\tikzstyle{boundarydisc}=[fill={rgb,255: red,128; green,128; blue,128}, draw=black, shape=circle, minimum size=0.80cm]
\tikzstyle{identity}=[fill=black, draw=black, shape=circle, inner sep=0pt, minimum size=0.15cm]
\tikzstyle{S}=[fill=white, draw=black, shape=rectangle, minimum width=0.15cm, minimum height=0.15cm]
\tikzstyle{S-}=[fill={rgb,255: red,199; green,199; blue,199}, draw=black, shape=rectangle, minimum width=0.15cm, minimum height=0.15cm]
\tikzstyle{dot}=[fill=white, draw=black, shape=circle, inner sep=0pt, minimum size=0.07cm]
\tikzstyle{bdot}=[fill=black, draw=black, shape=circle, inner sep=0pt, minimum size=0.07cm]
\tikzstyle{m}=[fill={rgb,255: red,255; green,5; blue,80}, draw={rgb,255: red,255; green,5; blue,80}, shape=circle, inner sep=0pt, minimum size=0.14cm, tikzit shape=circle]
\tikzstyle{mod}=[fill={rgb,255: red,172; green,229; blue,232}, draw={rgb,255: red,14; green,115; blue,177}, shape=circle]
\tikzstyle{smod}=[fill={rgb,255: red,172; green,229; blue,232}, draw={rgb,255: red,14; green,115; blue,177}, shape=circle, inner sep=0pt, minimum size=0.15cm]
\tikzstyle{smod1}=[fill={rgb,255: red,150; green,255; blue,138}, draw={rgb,255: red,5; green,130; blue,3}, shape=circle, inner sep=0pt, minimum size=0.15cm]
\tikzstyle{smod2}=[fill={rgb,255: red,255; green,202; blue,239}, draw={rgb,255: red,126; green,0; blue,124}, shape=circle, inner sep=0pt, minimum size=0.15cm]
\tikzstyle{smod3}=[fill={rgb,255: red,255; green,213; blue,164}, draw={rgb,255: red,121; green,57; blue,2}, shape=circle, inner sep=0pt, minimum size=0.15cm]
\tikzstyle{mid-nc}=[fill=white, draw=black, shape=circle, inner sep=0pt, minimum size=0.25cm]
\tikzstyle{l70-none}=[fill=none, draw=none, shape=circle, rotate=70]

\tikzstyle{di}=[draw=black, ->]
\tikzstyle{da}=[-, dashed]
\tikzstyle{da-di}=[dashed, ->]
\tikzstyle{th}=[-, thick]
\tikzstyle{th-di}=[->, thick]
\tikzstyle{th-da}=[-, dashed, thick]
\tikzstyle{gr}=[-, draw={rgb,255: red,0; green,186; blue,143}, thick]
\tikzstyle{gr-di}=[draw={rgb,255: red,0; green,186; blue,143}, ->, thick]
\tikzstyle{gr-da}=[-, dashed, draw={rgb,255: red,0; green,186; blue,143}, thick]
\tikzstyle{pu}=[-, draw={rgb,255: red,124; green,95; blue,239}, thick]
\tikzstyle{pu-da}=[-, dashed, draw={rgb,255: red,124; green,95; blue,239}, thick]
\tikzstyle{re}=[-, draw={rgb,255: red,255; green,5; blue,80}, thick]
\tikzstyle{re-di}=[->, draw={rgb,255: red,255; green,5; blue,80}, thick]
\tikzstyle{re-da}=[-, draw={rgb,255: red,255; green,5; blue,80}, dashed, thick]
\tikzstyle{bl}=[-, draw={rgb,255: red,14; green,115; blue,177}, thick]
\tikzstyle{bl-di}=[draw={rgb,255: red,14; green,115; blue,177}, ->, thick]
\tikzstyle{bl-da}=[-, draw={rgb,255: red,14; green,115; blue,177}, dashed, thick]
\tikzstyle{lgr}=[-, draw={rgb,255: red,188; green,220; blue,156}, thick]
\tikzstyle{lgr-da}=[-, draw={rgb,255: red,188; green,220; blue,156}, thick, dashed]
\tikzstyle{lbl}=[-, draw={rgb,255: red,168; green,207; blue,245}, thick]
\tikzstyle{lbl-da}=[-, draw={rgb,255: red,168; green,207; blue,245}, thick, dashed]
\tikzstyle{pi}=[-, draw={rgb,255: red,208; green,160; blue,160}, thick]
\tikzstyle{pi-da}=[-, draw={rgb,255: red,208; green,160; blue,160}, thick, dashed]
\tikzstyle{sh}=[-, draw={rgb,255: red,171; green,171; blue,171}]
\tikzstyle{sh-da}=[-, draw={rgb,255: red,171; green,171; blue,171}, dashed]
\tikzstyle{fi-gr}=[-, fill={rgb,255: red,194; green,228; blue,162}, draw={rgb,255: red,194; green,228; blue,162}]
\tikzstyle{fi-bl}=[-, fill={rgb,255: red,175; green,215; blue,255}, draw={rgb,255: red,175; green,215; blue,255}]
\tikzstyle{fi-pi}=[-, fill={rgb,255: red,246; green,188; blue,188}, draw={rgb,255: red,246; green,188; blue,188}]
\tikzstyle{fi-pu}=[-, fill={rgb,255: red,230; green,203; blue,246}, draw={rgb,255: red,230; green,203; blue,246}]
\tikzstyle{fi-ye}=[-, fill={rgb,255: red,255; green,255; blue,155}, draw={rgb,255: red,255; green,255; blue,155}]
\tikzstyle{fi-or}=[-, fill={rgb,255: red,255; green,197; blue,51}, draw={rgb,255: red,255; green,197; blue,51}]
\tikzstyle{fi-dg}=[-, fill={rgb,255: red,152; green,170; blue,139}, draw={rgb,255: red,152; green,170; blue,139}]
\tikzstyle{fi-db}=[-, fill={rgb,255: red,111; green,147; blue,183}, draw={rgb,255: red,111; green,147; blue,183}]
\tikzstyle{fi-sh}=[-, fill={rgb,255: red,222; green,222; blue,222}, draw={rgb,255: red,222; green,222; blue,222}]
\tikzstyle{fi-dsh}=[-, fill={rgb,255: red,204; green,204; blue,204}, draw={rgb,255: red,204; green,204; blue,204}]
\tikzstyle{vt}=[-, very thick]
\tikzstyle{vt-di}=[->, very thick]
\tikzstyle{vt-da}=[-, dashed, very thick]
\tikzstyle{vt-gr}=[-, draw={rgb,255: red,0; green,186; blue,143}, very thick]
\tikzstyle{vt-gr-di}=[draw={rgb,255: red,0; green,186; blue,143}, ->, very thick]
\tikzstyle{vt-gr-da}=[-, dashed, draw={rgb,255: red,0; green,186; blue,143}, very thick]
\tikzstyle{vt-pu}=[-, draw={rgb,255: red,124; green,95; blue,239}, very thick]
\tikzstyle{vt-pu-da}=[-, dashed, draw={rgb,255: red,124; green,95; blue,239}, very thick]
\tikzstyle{vt-re}=[-, draw={rgb,255: red,255; green,5; blue,80}, very thick]
\tikzstyle{vt-re-di}=[->, draw={rgb,255: red,255; green,5; blue,80}, very thick]
\tikzstyle{vt-re-da}=[-, draw={rgb,255: red,255; green,5; blue,80}, dashed, very thick]
\tikzstyle{vt-bl}=[-, draw={rgb,255: red,14; green,115; blue,177}, very thick]
\tikzstyle{vt-bl-di}=[draw={rgb,255: red,14; green,115; blue,177}, ->, very thick]
\tikzstyle{vt-bl-da}=[-, draw={rgb,255: red,14; green,115; blue,177}, dashed, very thick]
\tikzstyle{vt-bl-da-di}=[draw={rgb,255: red,14; green,115; blue,177}, ->, very thick, dashed]
\tikzstyle{vt-lgr}=[-, draw={rgb,255: red,188; green,220; blue,156}, very thick]
\tikzstyle{vt-lgr-di}=[draw={rgb,255: red,188; green,220; blue,156}, ->, very thick]
\tikzstyle{vt-lgr-da}=[-, draw={rgb,255: red,188; green,220; blue,156}, very thick, dashed]
\tikzstyle{vt-lbl}=[-, draw={rgb,255: red,168; green,207; blue,245}, very thick]
\tikzstyle{vt-lbl-di}=[->, draw={rgb,255: red,168; green,207; blue,245}, very thick]
\tikzstyle{vt-lbl-da}=[-, draw={rgb,255: red,168; green,207; blue,245}, very thick, dashed]
\tikzstyle{vt-lpp}=[-, draw={rgb,255: red,223; green,199; blue,240}, very thick]
\tikzstyle{vt-pi}=[-, draw={rgb,255: red,236; green,151; blue,151}, very thick]
\tikzstyle{vt-pi-da}=[-, draw={rgb,255: red,236; green,151; blue,151}, very thick, dashed]
\tikzstyle{vt-ye}=[-, draw={rgb,255: red,241; green,241; blue,115}, very thick]
\tikzstyle{vt-or}=[-, draw={rgb,255: red,245; green,150; blue,32}, very thick]
\tikzstyle{vt-sh}=[-, fill=none, draw={rgb,255: red,171; green,171; blue,171}, very thick]
\tikzstyle{vt-sh-di}=[draw={rgb,255: red,171; green,171; blue,171}, very thick, ->]
\tikzstyle{vt-sh-da}=[-, fill=none, draw={rgb,255: red,171; green,171; blue,171}, very thick, dashed]
\tikzstyle{fi-bx}=[-, fill={rgb,255: red,172; green,229; blue,232}, draw={rgb,255: red,14; green,115; blue,177}, thick]
\tikzstyle{ut-gr}=[-, draw={rgb,255: red,0; green,186; blue,143}, line width=3pt]
\tikzstyle{ut-gr-di}=[draw={rgb,255: red,0; green,186; blue,143}, ->, line width=3pt]
\tikzstyle{ut-ig}=[-, draw={rgb,255: red,255; green,5; blue,80}, line width=3pt]
\tikzstyle{ut-ig-di}=[draw={rgb,255: red,255; green,5; blue,80}, line width=3pt, ->]
\tikzstyle{fi-lbl}=[-, fill={rgb,255: red,213; green,237; blue,255}, draw={rgb,255: red,213; green,237; blue,255}]
\tikzstyle{fi-vlbl}=[-, fill={rgb,255: red,229; green,251; blue,255}, draw={rgb,255: red,229; green,251; blue,255}]
\tikzstyle{op-sh}=[-, draw={rgb,255: red,113; green,113; blue,113}, fill={rgb,255: red,113; green,113; blue,113}, opacity=0.5]

\usetikzlibrary{decorations.pathreplacing,decorations.markings}
\tikzset{
	on each segment/.style={
		decorate,
		decoration={
			show path construction,
			moveto code={},
			lineto code={
				\path [#1]
				(\tikzinputsegmentfirst) -- (\tikzinputsegmentlast);
			},
			curveto code={
				\path [#1] (\tikzinputsegmentfirst)
				.. controls
				(\tikzinputsegmentsupporta) and (\tikzinputsegmentsupportb)
				..
				(\tikzinputsegmentlast);
			},
			closepath code={
				\path [#1]
				(\tikzinputsegmentfirst) -- (\tikzinputsegmentlast);
			},
		},
	},
	mid arrow/.style={postaction={decorate,decoration={
				markings,
				mark=at position .5 with {\arrow[#1]{stealth}}
	}}},
}
\tikzset{%
	link/.style    = { white, double = blue, line width =2.0pt,
		double distance = 1.1pt },
	link2/.style    = { white, double = black, line width = 1.8pt,
		double distance = 0.4pt },
	channel/.style = { white, double = blue, line width = 0.8pt,
		double distance = 0.6pt },
}

\usepackage[all,cmtip]{xy}

\makeatletter
\DeclareRobustCommand{\em}{%
	\@nomath\em \if b\expandafter\@car\f@series\@nil
	\normalfont \else \slshape \fi}
\makeatother

\usepackage[hidelinks]{hyperref}

\numberwithin{equation}{section}
\usepackage{mathtools}
\mathtoolsset{showonlyrefs}
\numberwithin{equation}{section}

\newtheoremstyle{style1}
{13pt}
{13pt}
{}
{}
{\normalfont\bfseries}
{.}
{.5em}
{}

\theoremstyle{style1}

\newtheorem{definition}{Definition}[section]
\newtheorem{example}[definition]{Example}
\newtheorem{remark}[definition]{Remark}

\makeatletter
\newtheorem*{repd@theorem}{\repd@title}
\newcommand{\newrepdtheorem}[2]{%
	\newenvironment{repd#1}[1]{%
		\def\repd@title{#2 \ref{##1}}%
		\begin{repd@theorem}}%
		{\end{repd@theorem}}}
\makeatother

\newrepdtheorem{definition}{Definition}

\newcommand{\catf}[1]{{\mathsf{#1}}}

\newtheoremstyle{style2}
{13pt}
{13pt}
{\slshape}
{}
{\normalfont\bfseries}
{.}
{.5em}
{}

\theoremstyle{style2}

\makeatletter
\newtheorem*{rep@theorem}{\rep@title}
\newcommand{\newreptheorem}[2]{%
	\newenvironment{rep#1}[1]{%
		\def\rep@title{#2 \ref{##1}}%
		\begin{rep@theorem}}%
		{\end{rep@theorem}}}
\makeatother

\newreptheorem{theorem}{Theorem}
\newreptheorem{corollary}{Corollary}

\newtheorem{lemma}[definition]{Lemma}
\newtheorem{theorem}[definition]{Theorem}
\newtheorem{proposition}[definition]{Proposition}
\newtheorem{corollary}[definition]{Corollary}

\usepackage{tikz}
\usetikzlibrary{matrix,arrows,decorations.pathmorphing,shapes.geometric}
\usepackage{tikz-cd}

\usepackage{enumitem}

\usepackage{needspace}

\newcommand{\Ca}{\mathcal{C}}

\newcommand{\Map}{\catf{Map}}
\newcommand{\Diff}{\catf{Diff}}

\newcommand{\ra}[1]{\xrightarrow{\ #1 \ }}

\newcommand{\sn}{\catf{sn}}
\newcommand{\SN}{\catf{SN}}
\newcommand{\snc}{\catf{sn}_\cat{C}}
\newcommand{\SNC}{\catf{SN}_\cat{C}}

\newcommand{\colim}{\operatorname{colim}}

\newcommand{\colimsub}[1]{\underset{#1}{\operatorname{colim}}\,}

\usepackage{pifont}

\newcommand{\cat}[1]{\mathcal{#1}}

\newcommand{\End}{\operatorname{End}}

\newcommand{\Hom}{\operatorname{Hom}}

\newcommand{\id}{\operatorname{id}}

\newcommand{\ev}{\operatorname{ev}}

\newcommand{\Cat}{\catf{Cat}}

\let\to\undefined
\newcommand{\to}{\longrightarrow}
\let\mapsto\undefined
\newcommand{\mapsto}{\longmapsto}

\newcommand{\Rexf}{\catf{Rex}^{\mathsf{f}}}
\newcommand{\Lexf}{\catf{Lex}^\mathsf{f}}

\newcommand{\Bimod}{\catf{Bimod}_k}
\newcommand{\Bimodf}{\catf{Bimod}^\catf{f}_k}
\newcommand{\Vect}{\catf{Vect}_k}
\newcommand{\vect}{\catf{vect}_k}
\newcommand{\vectwok}{\catf{vect}}

\newcommand{\Graphs}{\catf{Graphs}}

\newcommand{\opp}{\text{opp}}

\let\colon\undefined\newcommand{\colon}{:}

\newcommand{\Proj}{\catf{Proj}\,}
\DeclareMathSymbol{\Phiit}{\mathalpha}{letters}{"08} 
\DeclareMathSymbol{\Psiit}{\mathalpha}{letters}{"09}
\DeclareMathSymbol{\Sigmait}{\mathalpha}{letters}{"06}
\DeclareMathSymbol{\Xiit}{\mathalpha}{letters}{"04}
\DeclareMathSymbol{\Piit}{\mathalpha}{letters}{"05}\let\Pi\undefined\newcommand{\Pi}{\Piit}
\DeclareMathSymbol{\Gammait}{\mathalpha}{letters}{"00}
\DeclareMathSymbol{\Omegait}{\mathalpha}{letters}{"0A}\let\Omega\undefined\newcommand{\Omega}{\Omegait}
\DeclareMathSymbol{\Upsilonit}{\mathalpha}{letters}{"07}
\DeclareMathSymbol{\Thetait}{\mathalpha}{letters}{"02}
\DeclareMathSymbol{\Lambdait}{\mathalpha}{letters}{"03}\let\Lambda\undefined\newcommand{\Lambda}{\Lambdait}

\let\Phi\undefined\newcommand{\Phi}{\Phiit}
\let\Sigma\undefined\newcommand{\Sigma}{\Sigmait}
\let\Psi\undefined\newcommand{\Psi}{\Psiit}
\let\Gamma\undefined\newcommand{\Gamma}{\Gammait}

\usepackage{multicol}
\newenvironment{pnum}{\begin{enumerate}[label=(\roman*)]}{\end{enumerate}}

\setkomafont{section}{\rmfamily\large}
\setkomafont{subsection}{\rmfamily}
\setkomafont{subsubsection}{\rmfamily}
\setkomafont{subparagraph}{\rmfamily} 

\makeatletter
\renewcommand\section{\@startsection {section}{1}{\z@}%
	{-3.5ex \@plus -1ex \@minus -.2ex}%
	{2.3ex \@plus.2ex}%
	{\normalfont\scshape\centering}}
\makeatother

\usepackage{titletoc}

\dottedcontents{section}[1em]{\rmfamily}{2em}{0.2cm}
\dottedcontents{subsection}[0em]{}{3.3em}{1pc}

\usepackage{titlesec}
\titleformat{\subsection}[runin]
{\normalfont\bfseries}
{\thesubsection}
{0.5em}
{}
[.]

\definecolor{Blue}  {rgb} {0.282352,0.239215,0.803921}
\definecolor{Green} {rgb} {0.133333,0.545098,0.133333}
\definecolor{Red}   {rgb} {0.803921,0.000000,0.000000}
\definecolor{Violet}{rgb} {0.580392,0.000000,0.827450}

\usepackage{multicol}

\usepackage{titletoc}
\dottedcontents{section}[1.5em]{\rmfamily}{1.5em}{0.2cm}
\dottedcontents{subsection}[0em]{}{3.3em}{1pc}

\begin{document}
	\begin{flushright}
	{\sffamily	[ZMP-HH/23-23] \\ 
		Hamburger Beiträge zur Mathematik Nr.~955  }
		\end{flushright}
	\vspace*{0.5cm}
	\begin{center}	\textbf{\large{The Lyubashenko Modular Functor for Drinfeld Centers \\[0.5ex]
				via Non-Semisimple String-Nets}}\\	\vspace{1cm}	{\large Lukas Müller $^{a}$, Christoph Schweigert $^{b}$, 
			Lukas Woike $^{c}$ and Yang Yang $^{d}$ }\\ 	\vspace{5mm}

		\begin{minipage}{0.48\textwidth}\centering
			{\slshape $^a$ Perimeter Institute \\ 31 Caroline Street North \\  CA-N2L 2Y5 Waterloo \\ Canada }	
			\\[7pt]{\slshape $^b$ Fachbereich Mathematik \\ Universität Hamburg \\ Bereich Algebra und Zahlentheorie \\ Bundesstraße 55 \\ D-20146 Hamburg \\ Germany	} 	
		\end{minipage}\hfill
		\begin{minipage}{0.48\textwidth}
			\centering 		{\slshape $^c$ Université Bourgogne Europe\\ CNRS\\ IMB UMR 5584\\ F-21000 Dijon\\ France }	\\[7pt]	{\slshape $^d$ Technische Universität München \\
				Department of Mathematics \\
				Boltzmannstraße 3 \\
				D-85748 Garching bei München \\ 
				Germany }
		\end{minipage}
	\end{center}	\vspace{0.3cm}	
	\begin{abstract}\noindent 
			 The Levin-Wen string-nets of a spherical fusion category $\mathcal{C}$ describe, by results of Kirillov and Bartlett, the representations of mapping class groups of closed surfaces obtained from the Turaev-Viro construction applied to $\mathcal{C}$. We provide a far-reaching generalization of this statement to arbitrary pivotal finite tensor categories, including non-semisimple or non-spherical ones: We show that the finitely cocompleted string-net modular functor built from the projective objects of a pivotal finite tensor category is equivalent to Lyubashenko's modular functor built from the Drinfeld center $Z(\mathcal{C})$. 
	\end{abstract}

	\tableofcontents
	
	\spaceplease
	\section{Introduction and summary}
	During the last few decades, an enormous amount of work has been
	devoted to the construction of representations of mapping class
	groups from
	certain monoidal and braided monoidal
	categories that appear in quantum algebra,
	for instance
	as categories of
	representations of quantum groups or vertex operator algebras.
	One of the classical constructions is the Turaev-Viro
	construction~\cite{turaevviro}
	that was developed further by Barrett-Westbury~\cite{barrettwestbury}.
	The input datum is a \emph{spherical fusion category}, a certain type of 
	monoidal category with rigid duals subject to finiteness and semisimplicity
	conditions.
	The construction yields for each surface $\Sigma$
	a finite-dimensional vector space $Z^\text{TV}_\cat{C}(\Sigma)$
	(for us, a surface is always
	compact and oriented; if $\Sigma$ has a boundary, additional boundary
	labels need to be specified), called the \emph{state space}. It comes with
	a representation of the mapping class group
	of $\Sigma$ on $Z^\text{TV}_\cat{C}(\Sigma)$.
	The assignment is compatible with gluing; it produces what is
	called a \emph{modular functor} in the sense of~\cite{Segal,ms89,turaev,tillmann,baki}.
	In fact, due to semisimplicity,
	one even obtains a three-dimensional topological
	field theory~\cite{turaevviro,barrettwestbury} that can be extended
	to codimension two~\cite{balsamkirillov,balsam2,balsam,tuvi,TV}.
	This topological field theory conjecturally extends down to the
	point: In \cite[Corollary~3.4.8]{dsps}
	a \emph{framed} three-dimensional fully extended topological field theory
	is built from a fusion category over a field of characteristic zero
	through the use of the cobordism hypothesis~\cite{baezdolan,lurietft}.
	According to~\cite[Conjecture~3.5.10]{dsps}, this topological field
	theory will descend to an oriented theory and, as a once-extended topological field theory, should coincide with the
	Turaev-Viro topological field theory.
	
	A different classical construction of mapping class
	group representations is the Reshetikhin-Turaev
	construction~\cite{rt1,rt2} that takes as an input a \emph{modular
		fusion category} $\cat{A}$,
	a certain type of finitely semisimple braided fusion category.
	The construction relies on surgery.
	Once again the construction yields
	a three-dimensional topological field theory $Z^\text{RT}_\cat{A}$
	that extends to codimension two~\cite{turaev,BDSPV15,TV}.

	The mapping class group representations produced in this way are
	sometimes referred to as \emph{quantum representations of mapping class
		groups}.
	These representations
	are highly non-trivial; they are, in certain examples, what is known as `asymptotically faithful'~\cite{andersenfaithful,fww}.
	However, they are generally difficult to describe explicitly.
	This creates a need for
	comparison results between different constructions because this
	allows to combine the knowledge that is 
	available separately on both constructions.
	The arguably most famous comparison result is the one between the
	Turaev-Viro and the Reshetikhin-Turaev construction: For a spherical
	fusion category $\cat{C}$
	over an algebraically closed field,
	we have the equivalence of once-extended three-dimensional topological field theories
	\begin{align}
	Z^\text{TV}_\cat{C} \simeq Z^\text{RT}_{Z(\cat{C})} \ ,  \label{tvvsrt}
	\end{align}
	 and
	hence in particular an equivalence of modular functors. Here $Z(\cat{C})$ is the
	Drinfeld center of $\cat{C}$,
	the
	braided monoidal category of all pairs $(X,\beta)$ of an object
	$X\in\cat{C}$ together with a half braiding $\beta : X \otimes -
	\cong -
	\otimes X$.
	Since $\cat{C}$ is a spherical fusion category, $Z(\cat{C})$ is a modular
	fusion category~\cite{mueger-sph}.
	The result~\eqref{tvvsrt} was established independently by
	Balsam-Kirillov~\cite{balsamkirillov,balsam2,balsam} (in characteristic
	zero) and Turaev-Virelizier~\cite{tuvi}, see also the monograph
	\cite[Section~17.1.2]{TV}.
	
	When it comes to obtaining an efficient description of the  mapping class group representations involved,
	\eqref{tvvsrt} is hardly helpful: Both the Turaev-Viro and the
	Reshetikhin-Turaev construction make heavy use
	of combinatorial tools like triangulations or surgery. While these tools
	are absolutely vital for the  invariants of three-dimensional manifolds that these topological
	field theories produce, it seems unnecessary for the description of the
	two-dimensional part of the theories.
	
	Indeed there is a third, skein-theoretic model that has
	gained popularity in recent years because it affords a direct
	topological construction of the mapping class group representations
	coming from the Turaev-Viro construction, namely the \emph{string-net model} that
	was developed by Levin-Wen~\cite{levinwen} for applications in
	condensed
	matter physics and refined by Kirillov~\cite{kirillovsn}. The idea
	of the
	string-net 	model for a spherical fusion category $\cat{C}$ is extremely
	simple: In order to obtain the string-net space $\SNC(\Sigma)$
	for a surface $\Sigma$ (for simplicity, let us assume it is
	closed for the moment), we draw, roughly, all possible finite oriented
	graphs on the surface, label the edges by objects in $\cat{C}$ and the
	vertices with morphisms (the source and target object are obtained from
	the incident labels).
	One then imposes local relations coming from the evaluation of
	string-nets in disk-shaped regions via the graphical calculus of $\cat{C}$.
	Through this definition, the graphical calculus is globalized from disks to surfaces.
	We recall in Section~\ref{secstringnet} how this intuitive idea is
	technically implemented.

	Kirillov proves in~\cite{kirillovsn} that
	\begin{align}
	Z^\text{TV}_\cat{C}(\Sigma)\cong \SNC(\Sigma)\label{eqntvsn}
	\end{align}
	as vector spaces.
	The equivariance of these isomorphisms under the mapping class
	group $\Map(\Sigma)$ of $\Sigma$
	is a rather subtle point:
	It follows from the construction of the three-dimensional
	string-net topological field theory of
	Bartlett-Goosen~\cite{bartlettgoosen}
	which, however, relies on the still unproven presentation of the
	three-dimensional bordism category through generators and relations
	used in~\cite{BDSPV15}.
	Nonetheless, by \cite{bartlett}
	this presentation can be replaced by a known
	presentation given by Juhász~\cite{juhasz} that circumvents this problem at least at the level of non-extended three-dimensional topological field theories. This gives us~\eqref{eqntvsn} as equivariant isomorphism for closed surfaces, but strictly speaking no equivalence of modular functors
	(the Juhász presentation does not include codimension two).
	Remarkably, a proof of~\eqref{eqntvsn}, even just
	as mapping class group representations for closed surfaces,
	 that does not pass through a presentation of
	the three-dimensional cobordism category in terms of generators and
	relations is
	not known.

	The purpose of this article is to generalize the string-net construction
	 beyond spherical fusion categories, more precisely to \emph{pivotal finite tensor categories} (this will be the easy part), and establish comparison results of the above flavor (this will be the hard part).
	Let us first clarify the terminology:
	 Finite tensor categories~\cite{etingofostrik} are linear abelian monoidal categories with rigid 	duals, a simple unit and some finiteness assumptions on the underlying linear category; a pivotal structure is a monoidal trivialization of the double dual functor.
	The reason why this is actually a vast and relatively involved generalization 
	is that,
	in contrast to fusion categories, finite tensor categories can be \emph{non-semisimple}
	(this will allow for a much richer supply of examples, most notably those coming from logarithmic conformal field theory). 
	While non-semisimplicity seems just
	like a technicality, it is known that it has the most dramatic consequences when it comes to the three mentioned constructions above:  
	The Reshetikhin-Turaev type topological field theory cannot exist as a once-extended three-dimensional topological field theory by the results of~\cite{BDSPV15}, but by a famous result of Lyubashenko~\cite{lyubacmp,lyu,lyulex}
	 a not necessarily semisimple modular category still gives rise to a modular functor. In~\cite{rggpmr,rggpmr2}, see also~\cite{bcgp} for the $\mathfrak{sl}(2)$-case, this modular functor is extended to a partially defined three-dimensional topological field theory, with the modified trace being used to produce interesting invariants of closed three-dimensional manifolds~\cite{cgp}. 
	The Turaev-Viro construction is not available in the non-semisimple case, but
	 a framed modular functor is built in~\cite{fssstatesum} using a state-sum construction.
	 
\subsection*{The main result}	 With this being the status quo, our path towards a generalization of the string-net construction and towards the best possible comparison results is the following:
	First we generalize the string-net construction to a construction that takes as an input a pivotal finite tensor category $\cat{C}$ and outputs a modular functor, i.e.\ a system of mapping class group representations. Neither sphericality in the sense of~\cite{dsps} nor semisimplicity of $\cat{C}$
	are needed.
	In fact, we even obtain an open-closed modular functor, which means that we include marked intervals on the boundaries of the surfaces. Setting up this generalization is not difficult, underlining once again the fact that
	 	the string-net construction is the easiest, most intuitive one of the above-mentioned constructions. On a technical level, it will be sufficient to adapt the generalized string-net construction of~\cite{fsy-sn} so that it can be applied to the tensor ideal of projective objects in a pivotal finite tensor category. 
	 	The idea of restricting to projective objects also underlies the recently introduced admissible skeins~\cite{asm} and ongoing work by Reutter and Walker~\cite{Reu,Walker2}. 
	 	Focusing on the projective objects is what gives non-semisimple modular functors the correct behavior under gluing, as was observed in~\cite{dva,dmf}. 
	 	We establish in Theorem~\ref{thmsnmf} that this non-semisimple string-net construction produces an open-closed modular functor.
	 	The only involved part of this construction is an excision result for string-nets (Theorem~\ref{thmexcision}) that, in this generality, is to the best of our knowledge
	 	not covered by classical gluing results for 
	 	 field theories built through skein-theoretic methods~\cite{Walker,skeinfin}. 
	 	
	 	Having the generalized string-net construction at our disposal,
	 	we can state our comparison result. More precisely,
	 	we will compare the modular functor $\SNC$
	 	built from string-nets with coefficients in a pivotal finite tensor category $\cat{C}$
	 	to the Lyubashenko modular functor for the Drinfeld center $Z(\cat{C})$ of $\cat{C}$; this amounts to a non-semisimple analog of the combination of~\eqref{tvvsrt} and~\eqref{eqntvsn}.

	 		\begin{reptheorem}{thmmain}
	 		For any pivotal finite tensor category $\cat{C}$, the string-net modular functor $\SNC$ associated 
	 	to $\cat{C}$ is equivalent to the Lyubashenko modular functor associated to the Drinfeld center $Z(\cat{C})$. 
	 	\end{reptheorem}

	 	The notion of equivalence of modular functors used here is the one from~\cite[Section~3.2]{brochierwoike}; the equivalence of the modular functor particularly implies that the (a priori projective)
	 	mapping class group representations are isomorphic.
	 	
	 	The reader may have noticed that if $\cat{C}$ is not spherical, $Z(\cat{C})$ does not have to be modular~\cite{shimizuribbon}. This has the consequence that Lyubashenko's original construction actually does not apply to $Z(\cat{C})$. Nonetheless in \cite[Section~8.4]{brochierwoike} the algebraic results of \cite{mwcenter} are used to build a modular functor from $Z(\cat{C})$ in this 
	 	 more general situation in a way that the original Lyubashenko modular functor is recovered in the spherical case. In Theorem~\ref{thmmain} the modular functor for $Z(\cat{C})$ has to be understood in this generalized sense.

	 \subsection*{The proof strategy}	The methods used 
	 	to prove Theorem~\ref{thmmain}
	 	are completely different from the methods used for the above comparison results~\eqref{tvvsrt}
	 	and~\eqref{eqntvsn} in the semisimple case.
	 	The Lyubashenko modular functor was originally built using generators and relations for the mapping class groups.
	 	A similar description for the topologically defined string-nets seems completely beyond reach.
	 	Instead, we make ample use of the factorization homology classification of modular functors~\cite{brochierwoike} that allows us to reduce
	 	 the comparison result between the string-nets for $\cat{C}$ and the modular functor for $Z(\cat{C})$ to a comparison between the ribbon Grothendieck-Verdier categories that both modular functors produce in genus zero~\cite{cyclic}. Nonetheless, no knowledge of factorization homology and almost no knowledge of Grothendieck-Verdier duality is needed for reading this article. 
	 	 The key ingredient for the genus zero comparison is a string-net description of the central monad (Section~\ref{seccentralmonad})
	 	 and an analysis of the Swiss-Cheese algebra obtained by evaluation of an open-closed version of the string-net modular functor in genus zero
	 	 (Section~\ref{secswisscheese}). The latter relies on Idrissi's characterization of categorical Swiss-Cheese algebras in~\cite{najib}. 
	 	 Grothendieck-Verdier duality comes into play to describe the topological duality that orientation reversal
	 	 induces on the category assigned by the string-net modular functor to the circle; it is one of the
	 	  insights of this article that this duality need not coincide with the rigid one, see Remark~\ref{remduality}.
	 	 
	 	 In the remaining subsections of this introduction,
	 	 let us highlight some applications, calculations and special cases of our main result.

	 	 \subsection*{The calculation of string-net spaces}
	 	 The string-net spaces are easy to define,
	 	  and the mapping class group action on them has a clear geometric origin. However, they 
	 	   have the disadvantage that the state spaces are difficult to compute from scratch, and not even their finite-dimensionality is obvious. 
	 	 With the comparison result from Theorem~\ref{thmmain}, we can express the string-net spaces  as follows
	 	 (we denote by $Z(\cat{C})(-,-)$ the morphism spaces in the Drinfeld center):

	 	 	\begin{repcorollary}{corconformalblocks}
	 	 	Let $\cat{C}$ be a pivotal finite tensor category and $\Sigma$ a surface with $n$ boundary components that are  labeled with $X_1 , \dots , X_n \in \cat{C}$.
	 	 	Then we can identify
	 	 	\begin{align}
	 	 		\SNC(\Sigma;X_1,\dots,X_n) \cong Z(\cat{C})  \left(  FX_1 \otimes \dots \otimes FX_n \otimes
	 	 		\left(\int_{X \in Z(\cat{C})} X \otimes X^\vee\right) ^{\otimes g} 
	 	 		, \alpha^{\otimes (g-1)}   \right) ^*    \ , 
	 	 	\end{align}
	 	 	where $F:\cat{C}\to Z(\cat{C})$ is the left adjoint to the forgetful functor from $Z(\cat{C})$ to $\cat{C}$, and $\alpha$ is the distinguished invertible object of $\cat{C}$, seen as object in the Drinfeld center via the pivotal structure and the Radford isomorphism.
	 	 	This isomorphism intertwines  the $\Map(\Sigma)$-action if 	$\SNC(\Sigma;X_1,\dots,X_n)$ is equipped with the geometric action and the right hand side with the (generalized) Lyubashenko action.
	 	 \end{repcorollary}
	 	 
	 	 With Corollary~\ref{corconformalblocks}, we discuss in Example~\ref{exsphH} the case where $\cat{C}$ is given by modules over a spherical Hopf algebra and give lower and upper bounds for the dimensions of string-net spaces. 
	 	 A full calculation of the string-net spaces in a non-spherical situation is given in Example~\ref{exgroupG}.

	 	 We should highlight 
	 	  that it is, in the general case, only the comparison 
	 	  result that allows us to conclude when string-net 
	 	  spaces are even non-zero. 
	 	 Indeed, this can hardly be decided through direct
	 	 computations  using the string-nets. 
	 	 One might think that it might be possible to find certain `generic' non-zero contributions to the string-net spaces (through local calculations or estimates, as they are available for spherical Hopf algebras), but the following result shows that this picture is too na\" ive:

	 		\begin{repcorollary}{corzero}
	 		Fix a closed surface $\Sigma$
	 		 of genus $g\neq 1$.
	 		Then there exists a pivotal finite tensor category $\cat{C}$ such that $\SNC(\Sigma)=0$, even though the modular functor $\SNC$ is still overall non-trivial, i.e.\ non-zero on some other closed surface.
	 	\end{repcorollary}

	 	\subsection*{The semisimple situation} 
	 	In the semisimple spherical case, our main
	 	result specializes to the statement that the string-net construction for a spherical fusion category $\cat{C}$ describes the Reshetikhin-Turaev type mapping class group representations for closed surfaces constructed from the modular fusion category $Z(\cat{C})$. Since the latter are equivalent to the Turaev-Viro type mapping class group representations built from $\cat{C}$, we recover the result
	 	 from~\cite{kirillovsn,bartlett}. It should be pointed out that our proof, also in the semisimple case, is genuinely new as it is two-dimensional. 
	 	 Moreover, the comparison as modular functors, as opposed to a comparison of representations of mapping class groups of closed surfaces, is to the best of our knowledge new.
	 	  One might also slightly shift the perspective:
	 	 As opposed to \cite{kirillovsn,bartlett}, we compare the string-net construction directly with the Reshetikhin-Turaev construction for the Drinfeld center, without passing through the Turaev-Viro construction. We can therefore use \cite{kirillovsn,bartlett} in combination with~\cite{BDSPV15} to give a new proof of $Z^\text{TV}_\cat{C} \simeq Z^\text{RT}_{Z(\cat{C})}$ (Corollary~\ref{cortvrt}). 
	 	 
	 	For a non-spherical fusion category $\cat{C}$, the string-net construction was considered in~\cite{Non-s-string-nets}. We prove that this construction produces the modular functor for the Drinfeld center $Z(\cat{C})$, but equipped with a Grothendieck-Verdier duality that does not coincide with the rigid duality (Corollary~\ref{corsnr}).
	 	
	 	\spaceplease
	 	\subsection*{Anomaly-freeness and extensions}
	 	As an application of our main result, we conclude from the properties of the string-nets:
	 	\begin{repcorollary}{coranomalyopenclosed}
	 		For a pivotal finite tensor category $\cat{C}$, the Lyubashenko modular functor for the Drinfeld center
	 		$Z(\cat{C})$ \begin{pnum}\item is anomaly-free \item and 
	 			extends to an open-closed modular functor sending the open boundary to $\cat{C}$. \end{pnum}
	 	\end{repcorollary}
	 	
	 	In the semisimple situation, the first result can be deduced from $Z^\text{TV}_\cat{C} \simeq Z^\text{RT}_{Z(\cat{C})}$. In the non-semisimple case, they are expected by experts, but we are not aware of a proof,
	 	and establishing these claims directly using the Lyubashenko modular functor seems relatively tedious.

	 		\vspace*{0.2cm}\textsc{Acknowledgments.} 
	 	We thank David Reutter for helpful discussions related to skein theory with projective labels. 
	 	LM gratefully acknowledges support of the Simons Collaboration on Global Categorical Symmetries. Research at Perimeter Institute is supported in part by the Government of Canada through the Department of Innovation, Science and Economic Development and by the Province of Ontario through the Ministry of Colleges and Universities. The Perimeter Institute is in the Haldimand Tract, land promised to the Six Nations. CS is supported by the Deutsche Forschungsgemeinschaft (DFG, German Research Foundation) under SCHW1162/6-1 and under
	 	Germany’s Excellence Strategy - EXC 2121 “Quantum Universe” - 390833306.
	 	LW  gratefully acknowledges support
	 	by the ANR project CPJ n°ANR-22-CPJ1-0001-01 at the Institut de Mathématiques de Bourgogne (IMB).
	 	The IMB receives support from the EIPHI Graduate School (ANR-17-EURE-0002).
	 	YY gratefully acknowledges support from the research group Mathematical Structures in Physics at the Max Planck Institute for Mathematics in the Sciences in Leipzig.

	\section{The non-semisimple string-net construction for the tensor ideal of projective objects\label{secstringnet}}
	In this section, we recall the construction of the string-net vector spaces for surfaces that was inspired by the work of
	 Levin and Wen~\cite{levinwen} and subsequently developed further in \cite{kirillovsn}. 
	 The presentation below is self-contained, but somewhat condensed.
	 On a technical level, it uses insights from a 
	  recent generalization of the string-net construction in \cite{fsy-sn}.

	In the sequel, a \emph{surface} will be an abbreviation for compact oriented two-dimensional manifold with boundary (which can be empty).
For a surface, the following will always be implicitly part of the data:
	The surface comes with a designated subset of its boundary components, the \emph{closed boundary components} which are parametrized, meaning that each of the boundary components is identified with a standard circle through a diffeomorphism which is part of the data. 
	In the remaining boundary components, we allow a finite number of disjoint parametrized intervals, the so-called \emph{open boundary intervals}.
	The boundary components that are neither among the closed boundary components nor have embedded intervals in them are the \emph{free boundary components}.

Next let us establish our conventions regarding graphs:
	A \emph{finite combinatorial graph} $\Gamma$ consists of a finite set $V$ of vertices, a finite set $H$ of half edges, a source map $s : H \to V$ (assigning to a half edge
	 its source vertex) and an involution $i : H \to H$ 
	(gluing half edges together).
	The orbits of the involution are called the \emph{edges} of the graph. Its fixed points are referred to as \emph{legs}.	Edges which are not legs will be called \emph{internal edges}. 
	For any vertex $v \in V$, we call $|v|:=|s^{-1}(v)|$ the \emph{valence of $v$};
	we require that it is always non-zero.  
	For any finite combinatorial graph $\Gamma$, we denote by $|\Gamma|$ the 1-dimensional CW complex obtained by geometric realization. 
	A finite combinatorial graph $\Gamma$ is called a \emph{tree} if $|\Gamma|$ is contractible. 
	A special case of a tree is a \emph{corolla}, a graph with one vertex and a finite number of legs attached to it.

	The role of decoration data for the string-nets that we will define in this section
	will be played by a \emph{pivotal tensor category} in the sense of Etingof-Ostrik; we will recall here the basic notions and refer to \cite{etingofostrik} for details.
	A \emph{finite} (linear) category\label{finitelinearpage} is an abelian linear category
	(over an algebraically closed field $k$
	 that we fix throughout)
	 with finite-dimensional morphism spaces, enough projective objects, finitely many
	isomorphism classes of simple objects, such that every object has finite length; these are exactly those linear categories linearly equivalent to finite-dimensional modules over some finite-dimensional algebra.
	A \emph{finite tensor category} $\cat{C}$
	is a finite category equipped with a left and right rigid monoidal product with simple unit $I$. It is a consequence of this definition that any finite tensor category is \emph{self-injective}, i.e.\ the projective objects are exactly the injective ones.
	A pivotal structure on a finite tensor category is a monoidal isomorphism from the double dual functor to the identity. We denote  the duality by $-^\vee : \cat{C}^\opp \to \cat{C}$ (thanks to pivotality, left and right duality can be identified).
	
There is another condition that will not play a role in the next definition, but only later, namely \emph{sphericality}:
Etingof-Nikshych-Ostrik define in~\cite{eno-d} the
distinguished invertible object $\alpha \in \cat{C}$ of a finite tensor category $\cat{C}$. This object
 describes the quadruple dual of 
  $\cat{C}$ via the \emph{Radford isomorphism}
\begin{align}
-^{\vee \vee\vee\vee}\cong \alpha \otimes-\otimes \alpha^{-1} \ . \label{eqnradford}
\end{align}
One says that $\cat{C}$ is \emph{unimodular} if $\alpha \cong I$. 
As a result, any unimodular pivotal finite tensor category $\cat{C}$ comes with two trivializations
of $-^{\vee \vee\vee\vee}$, namely the one arising from $\alpha \cong I$
(since $I$ is simple, this trivialization does not depend on the choice of the isomorphism $\alpha \cong I$),
and the square of the pivotal structure.
A \emph{spherical finite tensor category} in the sense of Douglas-Schommer-Pries-Snyder~\cite[Definition~3.5.2]{dsps} 
is a unimodular pivotal finite tensor category such that these two trivializations of $-^{\vee\vee\vee\vee}$ agree.
 We should warn the reader that the notion `spherical' that is defined via the quantum trace will generally  only agree 
 with the notion just defined in the semisimple case~\cite[Proposition~3.5.4 \& Example~3.5.5]{dsps}.

	\spaceplease
	\begin{definition}[$\text{Labeled graphs on a surface, following \cite[Definition~3.8]{fsy-sn}}$]\label{defclabaledgraph}
		Let $\Sigma$ be a surface and $\cat{C}$ a pivotal finite tensor category.
		A \emph{$\cat{C}$-labeled graph on $\Sigma$} is the following data:
		\begin{itemize}
			\item A non-empty finite combinatorial graph $\Gamma$ whose edges are each
			labeled with a projective object in $\cat{C}$ and an orientation of the edge. We will frequently suppress the orientation in what follows and only mention it when it is necessary.

			\item An  \emph{admissible embedding} of $|\Gamma|$ into $\Sigma$. Here \emph{admissible} means the following:
			
			\begin{itemize}
				\item  The part of $|\Gamma|$ that is mapped to $\partial \Sigma$ consists precisely of the endpoints of the legs (those ends of the legs which do not lie in a vertex; in particular,  vertices and internal edges can never be mapped to the boundary). We require the intersection of the image of the external legs with $\partial \Sigma$ to be transversal.
				In the sequel, we will often identify a graph with its image under the embedding.
				
				\item Each closed boundary component and each open boundary interval 
				of $\partial \Sigma$ is hit by at least one leg. 
				The free boundary components are not hit by legs.

			\end{itemize}
		\end{itemize}
			We write a $\cat{C}$-labeled graph in $\Sigma$ as pair $(\Gamma,\underline{X})$
			of the graph $\Gamma$ (being identified here with its image under the embedding)
			and a labeling $\underline{X}$ of each edge with a projective object.
			One defines a 
			\emph{category $\cat{C}\text{-}\Graphs(\Sigma)$
				of $\cat{C}$-labeled graph on $\Sigma$} as follows:
			\begin{itemize}
				\item
			Objects are $\cat{C}$-labeled graphs in $\Sigma$, as just defined.
			\item
			Morphisms are generated by
			 the replacements of graphs with corollas within disks (see Figure~\ref{figreplacedisk}):
			For any $(\Gamma,\underline{X})\in \cat{C}\text{-}\Graphs(\Sigma)$
			and any disk $D\subset \Sigma$ embedded in the interior of $\Sigma$ such that $\partial D$ intersects at least one of the legs of $\Gamma \subset \Sigma$, and only transversally without intersecting any vertex of $\Gamma$, one adds a morphism from $(\Gamma,\underline{X})$ to the embedded graph $\Gamma'$ that agrees with $\Gamma$ outside of $D$, but inside of $D$ is replaced with a corolla whose center is the center of $D\subset \Sigma$ and whose legs connect to the finitely many points in $\Gamma \cap \partial D$. 
			The morphisms are generated by these replacements within disks, subject to
			 the following relations:
			 \begin{itemize}
			 \item
			 Any replacement within a \emph{single} disk that is also an endomorphism, i.e.\ having the same domain and codomain, is the identity.
			 	\item  
			 Replacements within disjoint disks commute, which allows us to unambiguously define the replacement within a multidisk, i.e.\ a disjoint union of finitely many disks.
			 
			 \item Suppose that $D\subset D'$ are multidisks. Then the morphism
			 \begin{align}
			 	(\Gamma,\underline{X}) \ra{\text{replacement in $D$}}  (\Gamma',\underline{X'}) \ra{\text{replacement in $D'$}} 
			 	(\Gamma'',\underline{X''}) \ . 
			 	\end{align}
			 	and the morphism
			 	\begin{align}
			 		(\Gamma,\underline{X})  \ra{\text{replacement in $D'$}} 
			 		(\Gamma'',\underline{X''}) 
			 	\end{align}
			 	agree. It is always assumed that the relative position of the graphs and the disks is such that the replacement is allowed according to the above definition.
			 \end{itemize}
		\end{itemize}
	\end{definition}
	\begin{remark}
	It follows from the relations that any two disk replacements with the same domain and the same codomain differing only by rescaling of the disk
	 are identical.
	\end{remark}
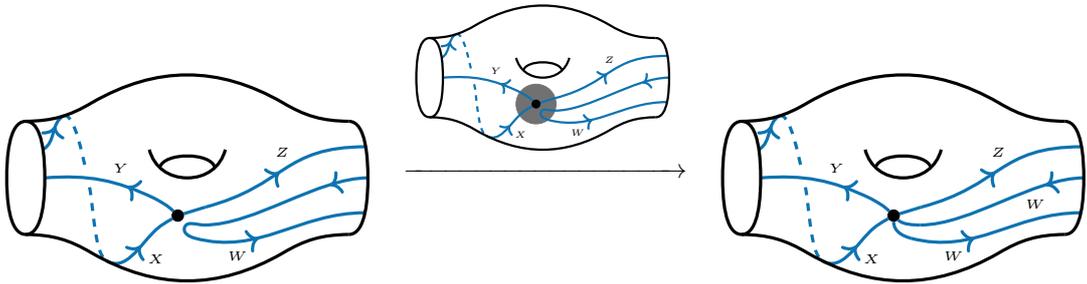
\begin{figure}[h]
	\[
	\begin{tikzpicture}[baseline=-0.25em,scale=0.5]
		\begin{pgfonlayer}{nodelayer}
			\node [style=none] (109) at (1, 0.75) {};
			\node [style=none] (110) at (-1, 0.75) {};
			\node [style=none] (111) at (0, 0) {};
			\node [style=none] (112) at (0.725, 0.275) {};
			\node [style=none] (113) at (-0.725, 0.275) {};
			\node [style=none] (114) at (-3.8, 0.75) {};
			\node [style=none] (115) at (-3.25, 1.65) {};
			\node [style=none] (116) at (-0.85, -0.55) {};
			\node [style=none] (117) at (-2.675, 0.25) {};
			\node [style=none] (118) at (-2.575, -0.25) {};
			\node [style=none] (119) at (-0.875, -1.425) {};
			\node [style=none] (120) at (-2.125, -2.2) {};
			\node [style=none] (121) at (0.475, -0.775) {};
			\node [style=identity] (122) at (-0.25, -1) {};
			\node [style=none] (123) at (4.7, 0.825) {};
			\node [style=none] (124) at (4.675, -0.925) {};
			\node [style=none] (125) at (0.45, -1.25) {};
			\node [style=none] (126) at (0.2, -1.6) {};
			\node [style=none] (127) at (-0.075, -1.325) {};
			\node [style=none] (97) at (-4.25, 1.5) {};
			\node [style=none] (98) at (-3.75, 0) {};
			\node [style=none] (99) at (-4.75, 0) {};
			\node [style=none] (100) at (-4.25, -1.5) {};
			\node [style=none] (101) at (4.25, 1.5) {};
			\node [style=none] (102) at (4.75, 0) {};
			\node [style=none] (104) at (4.25, -1.5) {};
			\node [style=none] (105) at (-2.5, 2) {};
			\node [style=none] (106) at (-2.5, -2) {};
			\node [style=none] (107) at (2.5, 2) {};
			\node [style=none] (108) at (2.5, -2) {};
			\node [style=none] (128) at (-1.525, -0.25) {};
			\node [style=none] (129) at (-1.575, -0.225) {};
			\node [style=none] (130) at (-1.3, -1.9) {};
			\node [style=none] (131) at (-1.225, -1.8) {};
			\node [style=none] (132) at (-3.475, 1.375) {};
			\node [style=none] (133) at (-3.5, 1.275) {};
			\node [style=none] (134) at (1.8, -1.625) {};
			\node [style=none] (135) at (1.875, -1.6) {};
			\node [style=none] (136) at (3.825, -0.175) {};
			\node [style=none] (137) at (3.75, -0.2) {};
			\node [style=none] (138) at (2.325, 0.075) {};
			\node [style=none] (139) at (2.425, 0.15) {};
			\node [style=none] (140) at (-1.75, 0.25) {\tiny$Y$};
			\node [style=none] (141) at (-0.825, -2.15) {\tiny$X$};
			\node [style=none] (142) at (2.5, 0.675) {\tiny$Z$};
			\node [style=none] (143) at (1.325, -2.075) {\tiny$W$};
		\end{pgfonlayer}
		\begin{pgfonlayer}{edgelayer}
			\draw [style=vt] (109.center)
			to [in=0, out=-105] (111.center)
			to [in=-75, out=180] (110.center);
			\draw [style=vt, bend right=75, looseness=0.75] (112.center) to (113.center);
			\draw [style=vt-bl, in=-150, out=30] (114.center) to (115.center);
			\draw [style=vt-bl] (98.center)
			to [bend left=15] (116.center)
			to [in=135, out=-30] (122);
			\draw [style=vt-bl-da, in=105, out=15, looseness=0.50] (115.center) to (117.center);
			\draw [style=vt-bl-da, in=150, out=-75, looseness=0.50] (118.center) to (120.center);
			\draw [style=vt-bl] (120.center)
			to [in=-135, out=-30] (119.center)
			to [in=210, out=45] (122);
			\draw [style=vt-bl] (122) to (121.center);
			\draw [style=vt-bl, in=-180, out=15, looseness=1.25] (121.center) to (123.center);
			\draw [style=vt-bl] (124.center)
			to [in=-15, out=-180] (126.center)
			to [in=-105, out=165] (127.center)
			to [in=180, out=75] (125.center)
			to [in=-180, out=0, looseness=0.75] (102.center);
			\draw [style=vt] (98.center)
			to [in=0, out=-90, looseness=0.75] (100.center)
			to [in=-90, out=180, looseness=0.75] (99.center)
			to [in=180, out=90, looseness=0.75] (97.center)
			to [in=90, out=0, looseness=0.75] cycle;
			\draw [style=vt] (101.center)
			to [in=90, out=0, looseness=0.75] (102.center)
			to [in=0, out=-90, looseness=0.75] (104.center);
			\draw [style=vt] (97.center)
			to [bend right=15] (105.center)
			to [bend left] (107.center)
			to [bend right=15] (101.center);
			\draw [style=vt] (100.center)
			to [bend left=15] (106.center)
			to [bend right] (108.center)
			to [bend left=15] (104.center);
			\draw [style=vt-bl-di] (128.center) to (129.center);
			\draw [style=vt-bl-di] (130.center) to (131.center);
			\draw [style=vt-bl-di] (133.center) to (132.center);
			\draw [style=vt-bl-di] (134.center) to (135.center);
			\draw [style=vt-bl-di] (136.center) to (137.center);
			\draw [style=vt-bl-di] (138.center) to (139.center);
		\end{pgfonlayer}
	\end{tikzpicture}
	\quad\stackrel{\scalebox{0.7}{\begin{tikzpicture}[baseline=-0.25em,scale=0.5]
				\begin{pgfonlayer}{nodelayer}
					\node [style=none] (18) at (-0.25, -0.25) {};
					\node [style=none] (19) at (-0.25, -1.75) {};
					\node [style=none] (20) at (-1, -1) {};
					\node [style=none] (21) at (0.5, -1) {};
					\node [style=none] (109) at (1, 0.75) {};
					\node [style=none] (110) at (-1, 0.75) {};
					\node [style=none] (111) at (0, 0) {};
					\node [style=none] (112) at (0.725, 0.275) {};
					\node [style=none] (113) at (-0.725, 0.275) {};
					\node [style=none] (114) at (-3.8, 0.75) {};
					\node [style=none] (115) at (-3.25, 1.65) {};
					\node [style=none] (116) at (-0.85, -0.55) {};
					\node [style=none] (117) at (-2.675, 0.25) {};
					\node [style=none] (118) at (-2.575, -0.25) {};
					\node [style=none] (119) at (-0.875, -1.425) {};
					\node [style=none] (120) at (-2.125, -2.2) {};
					\node [style=none] (121) at (0.475, -0.775) {};
					\node [style=identity] (122) at (-0.25, -1) {};
					\node [style=none] (123) at (4.7, 0.825) {};
					\node [style=none] (124) at (4.675, -0.925) {};
					\node [style=none] (125) at (0.45, -1.25) {};
					\node [style=none] (126) at (0.2, -1.6) {};
					\node [style=none] (127) at (-0.075, -1.325) {};
					\node [style=none] (97) at (-4.25, 1.5) {};
					\node [style=none] (98) at (-3.75, 0) {};
					\node [style=none] (99) at (-4.75, 0) {};
					\node [style=none] (100) at (-4.25, -1.5) {};
					\node [style=none] (101) at (4.25, 1.5) {};
					\node [style=none] (102) at (4.75, 0) {};
					\node [style=none] (104) at (4.25, -1.5) {};
					\node [style=none] (105) at (-2.5, 2) {};
					\node [style=none] (106) at (-2.5, -2) {};
					\node [style=none] (107) at (2.5, 2) {};
					\node [style=none] (108) at (2.5, -2) {};
					\node [style=none] (128) at (-1.525, -0.25) {};
					\node [style=none] (129) at (-1.575, -0.225) {};
					\node [style=none] (130) at (-1.3, -1.9) {};
					\node [style=none] (131) at (-1.225, -1.8) {};
					\node [style=none] (132) at (-3.475, 1.375) {};
					\node [style=none] (133) at (-3.5, 1.275) {};
					\node [style=none] (134) at (1.8, -1.625) {};
					\node [style=none] (135) at (1.875, -1.6) {};
					\node [style=none] (136) at (3.825, -0.175) {};
					\node [style=none] (137) at (3.75, -0.2) {};
					\node [style=none] (138) at (2.325, 0.075) {};
					\node [style=none] (139) at (2.425, 0.15) {};
					\node [style=none] (140) at (-1.75, 0.25) {\tiny$Y$};
					\node [style=none] (141) at (-0.825, -2.15) {\tiny$X$};
					\node [style=none] (142) at (2.5, 0.675) {\tiny$Z$};
					\node [style=none] (143) at (1.325, -2.075) {\tiny$W$};
				\end{pgfonlayer}
				\begin{pgfonlayer}{edgelayer}
					\draw [style=op-sh] (18.center)
					to [bend right=45] (20.center)
					to [bend right=45] (19.center)
					to [bend right=45] (21.center)
					to [bend right=45] cycle;
					\draw [style=vt] (109.center)
					to [in=0, out=-105] (111.center)
					to [in=-75, out=180] (110.center);
					\draw [style=vt, bend right=75, looseness=0.75] (112.center) to (113.center);
					\draw [style=vt-bl, in=-150, out=30] (114.center) to (115.center);
					\draw [style=vt-bl] (98.center)
					to [bend left=15] (116.center)
					to [in=135, out=-30] (122);
					\draw [style=vt-bl-da, in=105, out=15, looseness=0.50] (115.center) to (117.center);
					\draw [style=vt-bl-da, in=150, out=-75, looseness=0.50] (118.center) to (120.center);
					\draw [style=vt-bl] (120.center)
					to [in=-135, out=-30] (119.center)
					to [in=210, out=45] (122);
					\draw [style=vt-bl] (122) to (121.center);
					\draw [style=vt-bl, in=-180, out=15, looseness=1.25] (121.center) to (123.center);
					\draw [style=vt-bl] (124.center)
					to [in=-15, out=-180] (126.center)
					to [in=-105, out=165] (127.center)
					to [in=180, out=75] (125.center)
					to [in=-180, out=0, looseness=0.75] (102.center);
					\draw [style=vt] (98.center)
					to [in=0, out=-90, looseness=0.75] (100.center)
					to [in=-90, out=180, looseness=0.75] (99.center)
					to [in=180, out=90, looseness=0.75] (97.center)
					to [in=90, out=0, looseness=0.75] cycle;
					\draw [style=vt] (101.center)
					to [in=90, out=0, looseness=0.75] (102.center)
					to [in=0, out=-90, looseness=0.75] (104.center);
					\draw [style=vt] (97.center)
					to [bend right=15] (105.center)
					to [bend left] (107.center)
					to [bend right=15] (101.center);
					\draw [style=vt] (100.center)
					to [bend left=15] (106.center)
					to [bend right] (108.center)
					to [bend left=15] (104.center);
					\draw [style=vt-bl-di] (128.center) to (129.center);
					\draw [style=vt-bl-di] (130.center) to (131.center);
					\draw [style=vt-bl-di] (133.center) to (132.center);
					\draw [style=vt-bl-di] (134.center) to (135.center);
					\draw [style=vt-bl-di] (136.center) to (137.center);
					\draw [style=vt-bl-di] (138.center) to (139.center);
				\end{pgfonlayer}
			\end{tikzpicture}
	}}{\xrightarrow{\hspace*{3.5cm}}}\quad\begin{tikzpicture}[baseline=-0.25em,scale=0.5]
	\begin{pgfonlayer}{nodelayer}
	\node [style=none] (109) at (1, 0.75) {};
	\node [style=none] (110) at (-1, 0.75) {};
	\node [style=none] (111) at (0, 0) {};
	\node [style=none] (112) at (0.725, 0.275) {};
	\node [style=none] (113) at (-0.725, 0.275) {};
	\node [style=none] (114) at (-3.8, 0.75) {};
	\node [style=none] (115) at (-3.25, 1.65) {};
	\node [style=none] (116) at (-0.85, -0.55) {};
	\node [style=none] (117) at (-2.675, 0.25) {};
	\node [style=none] (118) at (-2.575, -0.25) {};
	\node [style=none] (119) at (-0.875, -1.425) {};
	\node [style=none] (120) at (-2.125, -2.2) {};
	\node [style=none] (121) at (0.475, -0.775) {};
	\node [style=identity] (122) at (-0.25, -1) {};
	\node [style=none] (123) at (4.7, 0.825) {};
	\node [style=none] (124) at (4.675, -0.925) {};
	\node [style=none] (125) at (0.45, -1.25) {};
	\node [style=none] (126) at (0.2, -1.6) {};
	\node [style=none] (97) at (-4.25, 1.5) {};
	\node [style=none] (98) at (-3.75, 0) {};
	\node [style=none] (99) at (-4.75, 0) {};
	\node [style=none] (100) at (-4.25, -1.5) {};
	\node [style=none] (101) at (4.25, 1.5) {};
	\node [style=none] (102) at (4.75, 0) {};
	\node [style=none] (104) at (4.25, -1.5) {};
	\node [style=none] (105) at (-2.5, 2) {};
	\node [style=none] (106) at (-2.5, -2) {};
	\node [style=none] (107) at (2.5, 2) {};
	\node [style=none] (108) at (2.5, -2) {};
	\node [style=none] (128) at (-1.525, -0.25) {};
	\node [style=none] (129) at (-1.575, -0.225) {};
	\node [style=none] (130) at (-1.3, -1.9) {};
	\node [style=none] (131) at (-1.225, -1.8) {};
	\node [style=none] (132) at (-3.475, 1.375) {};
	\node [style=none] (133) at (-3.5, 1.275) {};
	\node [style=none] (134) at (1.8, -1.625) {};
	\node [style=none] (135) at (1.875, -1.6) {};
	\node [style=none] (136) at (3.825, -0.175) {};
	\node [style=none] (137) at (3.75, -0.2) {};
	\node [style=none] (138) at (2.325, 0.075) {};
	\node [style=none] (139) at (2.425, 0.15) {};
	\node [style=none] (140) at (-1.75, 0.25) {\tiny$Y$};
	\node [style=none] (141) at (-0.825, -2.15) {\tiny$X$};
	\node [style=none] (142) at (2.5, 0.675) {\tiny$Z$};
	\node [style=none] (143) at (1.325, -2.075) {\tiny$W$};
	\node [style=none] (144) at (3.475, -0.7) {\tiny$W$};
	\end{pgfonlayer}
	\begin{pgfonlayer}{edgelayer}
	\draw [style=vt] (109.center)
	to [in=0, out=-105] (111.center)
	to [in=-75, out=180] (110.center);
	\draw [style=vt, bend right=75, looseness=0.75] (112.center) to (113.center);
	\draw [style=vt-bl, in=-150, out=30] (114.center) to (115.center);
	\draw [style=vt-bl] (98.center)
	to [bend left=15] (116.center)
	to [in=135, out=-30] (122);
	\draw [style=vt-bl-da, in=105, out=15, looseness=0.50] (115.center) to (117.center);
	\draw [style=vt-bl-da, in=150, out=-75, looseness=0.50] (118.center) to (120.center);
	\draw [style=vt-bl] (120.center)
	to [in=-135, out=-30] (119.center)
	to [in=210, out=45] (122);
	\draw [style=vt-bl] (122) to (121.center);
	\draw [style=vt-bl, in=-180, out=15, looseness=1.25] (121.center) to (123.center);
	\draw [style=vt-bl, in=-15, out=-180] (124.center) to (126.center);
	\draw [style=vt-bl, in=-180, out=0, looseness=0.75] (125.center) to (102.center);
	\draw [style=vt] (98.center)
	to [in=0, out=-90, looseness=0.75] (100.center)
	to [in=-90, out=180, looseness=0.75] (99.center)
	to [in=180, out=90, looseness=0.75] (97.center)
	to [in=90, out=0, looseness=0.75] cycle;
	\draw [style=vt] (101.center)
	to [in=90, out=0, looseness=0.75] (102.center)
	to [in=0, out=-90, looseness=0.75] (104.center);
	\draw [style=vt] (97.center)
	to [bend right=15] (105.center)
	to [bend left] (107.center)
	to [bend right=15] (101.center);
	\draw [style=vt] (100.center)
	to [bend left=15] (106.center)
	to [bend right] (108.center)
	to [bend left=15] (104.center);
	\draw [style=vt-bl-di] (128.center) to (129.center);
	\draw [style=vt-bl-di] (130.center) to (131.center);
	\draw [style=vt-bl-di] (133.center) to (132.center);
	\draw [style=vt-bl-di] (134.center) to (135.center);
	\draw [style=vt-bl-di] (136.center) to (137.center);
	\draw [style=vt-bl-di] (138.center) to (139.center);
	\draw [style=vt-bl, in=-45, out=-180, looseness=0.75] (125.center) to (122);
	\draw [style=vt-bl, in=-90, out=165] (126.center) to (122);
	\end{pgfonlayer}
\end{tikzpicture}
	\]
	\caption{Here we illustrate a generating morphism in the category  $\cat{C}\text{-}\Graphs(\Sigma)$, where $\Sigma$ is a genus one	 surface with two boundary circles. Note that each edge label is a projective object in $\cat{C}$.}\label{figreplacedisk}
\end{figure}

	\begin{definition}[$\text{String-net spaces, following \cite[Section~3.7]{fsy-sn}}$]
		Let $\cat{C}$ be a pivotal finite tensor category. 
		For a surface $\Sigma$, a boundary label is a collection of points in $\partial \Sigma$
		such that we have
		 at least one point
		 per closed or open boundary component and none on the free boundary components; 
		 each of the boundary points is
		 labeled with a projective object in $\cat{C}$ and $+$ or $-$ encoding the orientation of edges ending at the point. We agree that $+$ corresponds to the outwards pointing orientation while $-$ corresponds to the inward pointing orientation. Each $(\Gamma,\underline{X})\in \cat{C}\text{-}\Graphs(\Sigma)$ defines such a boundary label by taking the intersection of $\partial \Sigma$ with $\Gamma$ and using the labels and orientations of the external legs of $\Gamma$. We denote this boundary label by $\partial (\Gamma,\underline{X})$.
		For a boundary label $B$ of $\Sigma$, we denote by $\cat{C}\text{-}\Graphs(\Sigma;B)$ the full subcategory of $\cat{C}\text{-}\Graphs(\Sigma)$ spanned by those $(\Gamma,\underline{X})$ with $\partial (\Gamma,\underline{X})=B$. 
		There is an evaluation functor
		\begin{align}\mathbb{E}^{\Sigma,B}_\cat{C} :  \cat{C}\text{-}\Graphs(\Sigma;B) \to\Vect \  \label{eqnevaluation}
		\end{align}to the category of $k$-vector spaces
		defined as follows:
		\begin{itemize}
			\item 
			If $(\Gamma,\underline{X}) \in  \cat{C}\text{-}\Graphs(\Sigma;B)$ is a corolla, then
			\begin{align} 
			\mathbb{E}^{\Sigma,B}_\cat{C}(\Gamma,\underline{X}):=\cat{C}(I,X_1^{\varepsilon_1}\otimes\dots \otimes X_n^{\varepsilon_n})\label{eqndefoncor}
			\end{align} 
	where $\varepsilon_i\in \{+,-\}$ is the sign encoding the orientation of the edges, and
	 $X_i^+:=X_i$, and $X_i^{-}:=X_i^\vee$ 
			if $\underline{X}$ consists of the objects $X_1,\dots,X_n \in \Proj\cat{C}$ in the cyclic order induced from the embedding into the oriented surface $\Sigma$. 
			(The definition~\eqref{eqndefoncor} suggests that not only a cyclic order is chosen, but an actual order.
			This however is done here only to be explicit. The definition can be made without this choice by using the fact that $\cat{C}$, being a pivotal tensor category, is a category-valued cyclic associative algebra~\cite[Section~4]{cyclic}.)

			\item If $(\Gamma,\underline{X}) \in  \cat{C}\text{-}\Graphs(\Sigma,B)$ is  not a corolla, we cut $(\Gamma,\underline{X})$ at all internal edges and obtain a finite disjoint union of corollas to each of which we assign the vector space~\eqref{eqndefoncor}.
			Then $\mathbb{E}^{\Sigma,B}_\cat{C}(\Gamma,\underline{X})$ is obtained by tensoring all these vector spaces together
			(formally, this is an unordered tensor product). 
			
			\item If $\Sigma$ is a disk 
			$D$ and $B=(B_1,\dots,B_\ell)$ a boundary label in which all $B_i$ carry the sign $+$, then the graphical calculus (see \cite[Section~2.15]{fsy-sn}) provides us with a map
			\begin{align}
				\mathbb{E}^{D,B}_\cat{C}(\Gamma,\underline{X}) \to \cat{C}(I,B_1 \otimes \dots \otimes B_\ell) = \mathbb{E}^{D,B}_\cat{C}(T_{\ell-1}, (B_1,\dots,B_\ell) ) \ , 
			\end{align}
			where $T_{\ell-1}$ is the corolla with $\ell$ legs. 
			This affords the definition of~\eqref{eqnevaluation} on morphisms.
		\end{itemize}
		We define the \emph{string-net space for the surface $\Sigma$ with coefficients in
			$\cat{C}$ and boundary label $B$} as the colimit
		\begin{align}
			\snc (\Sigma;B) := \colimsub{(\Gamma,\underline{X}) \in \cat{C}\text{-}\Graphs(\Sigma;B)}     \mathbb{E}^{\Sigma,B}_\cat{C}(\Gamma,\underline{X})   \label{eqncolimitstringnet}
		\end{align}
		Vectors in this vector space are called \emph{string-nets}.
	\end{definition}

\begin{remark} By definition an element of the vector space $\snc(\Sigma)$
	can be represented by a linear combination of \emph{fully labeled oriented graphs} $(\Gamma,\underline{X},\underline{f})$, where $\underline{f}\in\mathbb{E}^{\Sigma,B}_\cat{C}(\Gamma,\underline{X})$ is a choice of labels for the vertices (note that
	for the objects in $\cat{C}\text{-}\Graphs(\Sigma;B)$, no labels for the vertices are specified), by simply applying the structural map 
		\begin{align}
\mathbb{E}^{\Sigma,B}_\cat{C}(\Gamma,\underline{X})\to	\snc (\Sigma;B) := \colimsub{(\Gamma,\underline{X}) \in \cat{C}\text{-}\Graphs(\Sigma;B)}     \mathbb{E}^{\Sigma,B}_\cat{C}(\Gamma,\underline{X})\label{eqndefsncolimit}
	\end{align}
	Two of such fully labeled graphs 
	$(\Gamma,\underline{X},\underline{f})$ and $(\Omega,\underline{Y},\underline{g})$
	represent the same string-net precisely if
	$(\Gamma,\underline{X})$ and $(\Omega,\underline{Y})$ can be connected by a zigzag in 
	$\cat{C}\text{-}\Graphs(\Sigma;B)$ such that the evaluation of $\mathbb{E}^{\Sigma,B}_\cat{C}$ on this zigzag transforms $\underline{f}$ to $\underline{g}$.
	\end{remark}

\begin{remark}
	The colimit definition of string-net spaces is not the one used in~\cite{kirillovsn}, but it is equivalent when $\cat{C}$ is the spherical fusion category (which was the only case considered in \cite{kirillovsn}). 
	This is the content of \cite[Theorem~3.9]{fsy-sn}. 
	Hence, we use~\eqref{eqncolimitstringnet} already as a definition for reasons of conciseness.
	\end{remark}

	\begin{remark}\label{remdsn1}
		There is another motivation for the colimit definition of string-net spaces, see also \cite[Section~3.7]{fsy-sn}: The colimit in~\eqref{eqndefsncolimit}
		can `simply' be replaced by a homotopy colimit, thereby leading to a derived version of string-nets. This lies  beyond the scope of the present paper, see however Remark~\ref{remdsn2} for an additional comment.
		\end{remark}

	We conclude this section with some of the more obvious, but extremely crucial results and the constructions for string-nets that are analogous to those from~\cite{fsy-sn}, which are in turn inspired by \cite{Walker,kirillovsn}. 
	Because of
	 the modifications that we implement in this paper, these results do not  follow directly from~\cite{fsy-sn}.

We begin by calculating the string-net space of the disk, see Figure~\ref{snindisk}:
	
	\begin{lemma}\label{lemmasndisks}
		For any pivotal finite tensor category $\cat{C}$ and any disk $D$ with boundary label $B$ formed by projective labels $X_1, \dots, X_n$ cyclically ordered through the orientation of $D$, there is an isomorphism
		\begin{align}
			\snc (D;B) \cong \cat{C}(I, X_1 \otimes \dots \otimes X_n) \ . 
			\end{align}
		\end{lemma}
		
		\begin{proof}
			The corolla consisting of one vertex in the middle of $D$ radially connected to the boundary labels is a terminal object in $\cat{C}\text{-}\Graphs(\Sigma;B)$, see also~\cite[Example~3.5]{fsy-sn}. This implies the assertion. 
			\end{proof}

\begin{figure}[h]
\[
\begin{tikzpicture}[baseline=-0.25em,scale=0.5]
	\begin{pgfonlayer}{nodelayer}
		\node [style=none] (107) at (3.525, -1.875) {};
		\node [style=none] (110) at (-3.3, 2.15) {};
		\node [style=none] (88) at (2.65, 3.05) {};
		\node [style=none] (90) at (-2.45, -3.15) {};
		\node [style=none] (121) at (3.25, 3.5) {$X_1$};
		\node [style=none] (122) at (4, -2.5) {$X_2$};
		\node [style=none] (123) at (0.75, 1.25) {$X_5$};
		\node [style=none] (124) at (-3.875, 2.7) {$X_4$};
		\node [style=sr] (127) at (1.5, 0) {$f_1$};
		\node [style=rsr] (128) at (-1, -0.5) {$f_2$};
		\node [style=none] (129) at (0.425, 0.75) {};
		\node [style=none] (130) at (0.375, 0.725) {};
		\node [style=none] (152) at (-3, -3.5) {$X_3$};
		\node [style=none] (49) at (-4, 0) {};
		\node [style=none] (50) at (4, 0) {};
		\node [style=none] (51) at (0.25, 4) {};
		\node [style=none] (52) at (0, -4) {};
	\end{pgfonlayer}
	\begin{pgfonlayer}{edgelayer}
		\draw [style=vt-bl, in=120, out=60, looseness=1.25] (128) to (127);
		\draw [style=vt-bl-di, in=-135, out=45] (127) to (88.center);
		\draw [style=vt-bl-di, in=150, out=-90] (127) to (107.center);
		\draw [style=vt-bl-di, in=330, out=90] (128) to (110.center);
		\draw [style=vt-bl-di] (130.center) to (129.center);
		\draw [style=vt-bl-di, bend left=15] (128) to (90.center);
		\draw [style=vt] (50.center)
		to [bend left=45] (52.center)
		to [bend left=45] (49.center)
		to [bend left=45] (51.center)
		to [bend left=45] cycle;
	\end{pgfonlayer}
\end{tikzpicture}
\]
\caption{Here we illustrate a typical string-net in $\snc (D;B) \cong \cat{C}(I, X_1 \otimes X_2 \otimes X_3 \otimes X_4)$. Our convention is that when the disk $D$ is oriented counterclockwise, the labels $X_1,\cdots,X_4\in\Proj\cat{C}$ are placed \emph{clockwise} on the boundary circle.}\label{snindisk}
\end{figure}
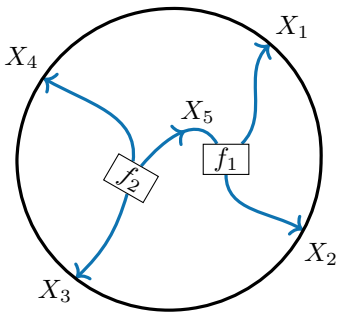

	Isotopies can be decomposed into isotopies supported in disks~\cite[Corollary~1.3]{edwardskirby}. As in \cite[Section~3.1]{fsy-sn},
	 this implies the isotopy invariance of string-nets:
 
\begin{proposition}[Isotopy invariance of string-nets]\label{propisotopy}
	For a pivotal finite tensor category $\cat{C}$, a surface $\Sigma$ and a boundary label $B$, consider
	a labeled graph $(\Gamma,\underline{X}) \in \cat{C}\text{-}\Graphs(\Sigma;B)$ and a labeling $\underline{f}\in\mathbb{E}^{\Sigma,B}_\cat{C}(\Gamma,\underline{X})$ of the vertices of $\Gamma$. If we change $\Gamma$ through an isotopy relative boundary to a different admissible graph $\widetilde \Gamma$
	 in $\Sigma$, then $\widetilde \Gamma$ inherits a labeling for its edges and vertices, again denoted by $\underline{X}$ and $\underline{f}$, respectively.
	 Then $(\Gamma,\underline{X},\underline{f})$ and $(\widetilde \Gamma,\underline{X},\underline{f})$ represent the same vector in $\snc(\Sigma;B)$.
\end{proposition}

  As a preparation for the definition of boundary categories, we recall two key 
   examples of a symmetric monoidal bicategory. For an introduction to symmetric monoidal bicategories, we refer to~\cite[Chapter~2]{schommerpries}. 
  
  \begin{definition}
  	We denote by $\Cat_k$ the bicategory 
  	\begin{itemize}
  		\item whose objects are  $k$-linear categories,
  		\item whose 1-morphisms are  $k$-linear functors,
  		\item and whose 2-morphisms are natural transformations.   
  	\end{itemize}
  	The na\" ive monoidal product $\otimes$ turns $\Cat_k$ into a symmetric monoidal bicategory. Concretely, 
  	the na\" ive monoidal product $\cat{C}\otimes \cat{D}$ of two $k$-linear categories $\cat{C}$ and $\cat{D}$ 
  	has as objects pairs $(c,d)$ with $c\in \cat{C}$ and $d\in \cat{D}$ and morphism spaces 
  	$(\cat{C}\otimes\cat{D})((c, d),(c', d'))= \cat{C} (c,c')\otimes \cat{D}(d,d')$. 
  	The monoidal unit is the $k$-linear category $1_k$ with one object whose endomorphism algebra is $k$.
  \end{definition}

	\begin{definition}\label{defbimod}
	We denote by $\Bimod$ the bicategory   
	\begin{itemize}
		\item whose objects are $k$-linear categories,
		\item whose 1-morphisms are bimodules (also called profunctors), i.e.\ a 1-morphism $\cat{C} \to \cat{D}$ is a $k$-linear functor
		$ \cat{C}\otimes \cat{D}^{\opp} \to \Vect$ ($\Vect$ is the category of $k$-vector spaces;
		composition is defined via coends),
		\item and whose 2-morphisms are natural transformations between bimodules.   
	\end{itemize}
	The na\" ive monoidal product defines the structure of a symmetric monoidal category on $\Bimod$. The monoidal unit is $1_k$.
\end{definition}

	\begin{definition}[$\text{Boundary categories, following \cite[Section~3.7]{fsy-sn}}$]\label{defboundarycategories}
		For a pivotal finite tensor category $\cat{C}$
		and a compact one-dimensional manifold $S$ with boundary, we denote by $\snc(S)$ the category whose objects are collections of finitely many points in the interior of $S$, at least one in each component of $S$, and each decorated with a projective object in $\cat{C}$. 
		For $B,C \in \snc(S)$, the morphism space is defined as
		\begin{align}
			\snc(S)(B,C) := \snc(S\times [0,1] ; B^\vee , C) \ ,
		\end{align}
	where $B^\vee$ is the dual boundary label obtained by replacing each label in $B$ by its dual. Note that we tacitly identify a $\cat{C}$-labeled line with its orientation reversal whose label is given by the dual object.
		This allows us to define composition via stacking cylinders.
		It is immediate from the definition of $\snc(S)$ that for any surface $\Sigma$
		with boundary $\partial \Sigma$ (union of the closed and open boundary components),
		 the spaces $\snc (\Sigma;-)$ assemble into a linear functor
		\begin{align}
			\snc(\Sigma) : \snc(\partial \Sigma)\to\Vect\ ,\label{eqnfunctorforsurface0}
		\end{align}
i.e.\  a 1-morphism in $\Bimod$ from $\snc(\partial \Sigma)$ to the monoidal unit $1_k$ of $\Bimod$. 
	\end{definition}

	\begin{remark}\label{remsncint}
		We 
		have a functor
		$\snc([0,1])\to\Proj\cat{C}$
		that sends $n$ points on the interval labeled with $X_1,\dots,X_n \in \Proj\cat{C}$
		to $X_1\otimes\dots\otimes X_n$. 
		Clearly, this functor is essentially surjective. The functor is also fully faithful,
		and hence an equivalence, because the morphism spaces in $\snc([0,1])$
		are exactly the hom spaces in $\Proj\cat{C}$ by Lemma~\ref{lemmasndisks}.
		\end{remark}

	\begin{remark}\label{remsncsymmon}
		The assignment $S \mapsto \snc(S)$ is symmetric monoidal with respect to the disjoint union of compact one-dimensional manifolds $S$ with boundary and the na\" ive monoidal product of linear categories.
	\end{remark}

For any surface $\Sigma$, the diffeomorphisms of $\Sigma$ preserving the boundary parametrizations form a topological group $\Diff(\Sigma)$ (the free boundary components are not parametrized and hence not preserved point-wise). Its group of path components $\Map(\Sigma):=\pi_0(\Diff(\Sigma))$ is the \emph{mapping class group} of $\Sigma$. 
Clearly, the diffeomorphism group of $\Sigma$ acts on the string-net space of $\Sigma$. 
Thanks to the isotopy invariance statement from Proposition~\ref{propisotopy}, the action descends to the mapping class group:
	
\begin{corollary}\label{cormcg}
For any pivotal finite tensor category $\cat{C}$ and any surface $\Sigma$, the linear functor
$\snc(\Sigma) : \snc(\partial \Sigma)\to\Vect$ carries a representation of the mapping class group $\Map(\Sigma)$ of $\Sigma$ through natural isomorphisms.
\end{corollary}

The definition of boundary categories is compatible with 
the orientation reversal
of manifolds, see also \cite[eq.~(3.35)]{fsy-sn}:

\begin{lemma}\label{Lemma: or rev}
	Let $S$ be an oriented 1-dimensional manifold and $\bar{S}$ the same manifold with reversed 
	orientation. There is a canonical equivalence
	\begin{align}
		\snc(S)^\opp \simeq \snc(\bar{S}) \ \ . 
	\end{align}
\end{lemma}

\begin{proof}
The equivalence sends an object $(S,P_1,\dots, P_n)$ to $(\bar{S},P_1^\vee,\dots, P_n^\vee)$
 and a morphism, i.e.\ an equivalence class of string-nets in $ S \times [0,1] $, to the same string-net composed with the diffeomorphism $S\times [0,1]\to S\times [0,1], \ (s,t) \mapsto (s, 1-t)$.  
\end{proof}

As a direct consequence of Lemma~\ref{Lemma: or rev}, we note that for a cobordism $\Sigma \colon S_1 \to S_2$ the linear functor $\snc(\Sigma)\colon \snc({S_1}) \otimes \snc (\overline{S_2} ) \to \Vect $ can also be interpreted as a morphism $\snc(\Sigma)\colon \snc (S_1 ) \to \snc (S_2 ) $ in $\Bimod$.

	\begin{proposition}[Additivity for string-nets]\label{propadd} Let $\cat{C}$ be a pivotal finite tensor category and let $(\Gamma,\underline{X}, \underline{f})$ be a string-net
	in a surface $\Sigma$ with boundary label $B$ such that a fixed internal edge $e$ is labeled by the direct sum $P\oplus Q$ of two projective objects. Denote by $(\Gamma,\underline{X}_P, \underline{f}_P)$ and $(\Gamma,\underline{X}_{Q	}, \underline{f}_{Q})$ the restricted string-nets, i.e.\
	\begin{itemize}
		\item
		the label $P\oplus Q$ of $e$ is replaced with $P$ or $Q$, respectively, \item and the morphism at the start and end of the edge $e$ is
		restricted to the object $P$ or $Q$, respectively.
	\end{itemize}
	Then we have \begin{align}(\Gamma,\underline{X}, \underline{f})= (\Gamma,\underline{X}_{P}, \underline{f}_{P})+(\Gamma,\underline{X}_{Q}, \underline{f}_{Q})\end{align} as elements of $\snc(\Sigma;B)$.
\end{proposition}

\begin{proof}
	The colimit defining the string-net space $\snc(\Sigma;B)$ from~\eqref{eqncolimitstringnet} can be written as the coequalizer of
	\begin{equation}
	\begin{tikzcd}
	\displaystyle	\bigoplus_{\substack{\text{morphisms} \\ g : (\Omega,\underline{Y})\to (\Omega',\underline{Y'}) \\ \text{in} \ \cat{C}\text{-}\Graphs(\Sigma;B)}} \mathbb{E}^{\Sigma,B}_\cat{C}(\Omega,\underline{Y}) \ar[r, shift left=2,"\id"] \ar[r, swap,shift right=2,"E"]
	& \displaystyle \bigoplus_{(\Omega,\underline{Y})\in \cat{C}\text{-}\Graphs(\Sigma;B)} \mathbb{E}^{\Sigma,B}_\cat{C}(\Omega,\underline{Y}) \ ,   
	\end{tikzcd}
	\end{equation}  
	where \begin{itemize}\item $\id$ projects the morphism used for indexing to its source object and is summand-wise the identity
		(therefore we denote this map by $\id$, even though it is a slight abuse of notation), \item and $E$ (as in `evaluation')
		on the summand  $\mathbb{E}^{\Sigma,B}_\cat{C}(\Omega,\underline{Y})$ indexed by
		$g : (\Omega,\underline{Y})\to (\Omega',\underline{Y'})$
		maps to $\mathbb{E}^{\Sigma,B}_\cat{C}(\Omega',\underline{Y'})$ indexed by $(\Omega',\underline{Y'})$ via
		$\mathbb{E}^{\Sigma,B}_\cat{C}(g):           \mathbb{E}^{\Sigma,B}_\cat{C}(\Omega,\underline{Y})\to
		\mathbb{E}^{\Sigma,B}_\cat{C}(\Omega',\underline{Y'})$. \end{itemize}
	We now choose a disk $D\subset \Sigma$ which completely contains $e$ and whose boundary 
	intersects $\Gamma$ transversely in
	at  least one point.
	The replacement on $D$ gives us morphisms
	\begin{align}
	a : (\Gamma,\underline{X}) &\to ( \Gamma',\underline{X'}) \ , \\
	b : (\Gamma,\underline{X}_P) &\to ( \Gamma',\underline{ X'}_P) \ , \\  
	c : (\Gamma,\underline{X}_Q) &\to ( \Gamma',\underline{ X'}_Q) \ . 
	\end{align}
	With the equality
	$\id_{P\oplus Q}= \iota_P \circ \pi_{P} + \iota_{Q} \circ \pi_{Q} $
	of morphisms in $\cat{C}$ ($\pi$ denotes the projection and $\iota$ the embedding),
	we arrive at
	\begin{align}
	E \left(   a,(\Gamma,\underline{X},\underline{f})     \right)=
	E \left(   b,(\Gamma,\underline{X}_P,\underline{f}_P)     \right)
	+
	E \left(   c,(\Gamma,\underline{X}_Q,\underline{f}_Q)     \right) \ . 
	\end{align} 
	As vectors in $\snc(\Sigma;B)$,
	we have moreover
	\begin{align} (\Gamma,\underline{X}, \underline{f}) &= E \left(   a,(\Gamma,\underline{X},\underline{f})     \right) \ , \\
	(\Gamma,\underline{X}_P,\underline{f}_P) &= E(\Gamma,\underline{X}_P,\underline{f}_P) \ , \\
	(\Gamma,\underline{X}_Q,\underline{f}_Q) &= E(\Gamma,\underline{X}_Q,\underline{f}_Q) 
	\end{align}
	(thanks to the coequalizer),
	and 
	therefore
	\begin{align}
	(\Gamma,\underline{X}, \underline{f}) &= E \left(   a,(\Gamma,\underline{X},\underline{f})     \right)\\&=E \left(   b,(\Gamma,\underline{X}_P,\underline{f}_P)     \right)
	+
	E \left(   c,(\Gamma,\underline{X}_Q,\underline{f}_Q)     \right)\\&=(\Gamma,\underline{X}_P,\underline{f}_P)+(\Gamma,\underline{X}_Q,\underline{f}_Q) \ . 
	\end{align} 
\end{proof}

\begin{remark}[Relation to admissible skein modules in dimension two]\label{relasm}
	The string-net construction used in this paper follows to a large extent
	the general bicategorical 
	string-net construction~\cite{fsy-sn}. It seems 
	likely that one could also use the \emph{admissible skein modules} 
	of Costantino-Geer-Patureau-Mirand~\cite{asm} 
	(the fact that the two-dimensional version of the admissible skein modules is a string-net construction is also mentioned and used there). More precisely, for a pivotal finite tensor category $\cat{C}$,
	the string-net construction $\snc$ used in this paper will 
	agree with the two-dimensional admissible skein module
	construction applied to the tensor ideal $\Proj \cat{C}$. 
	We will not expand further on this comparison as it would not --- to the best of our knowledge --- simplify the proof of our main result (with the exception of the finite-dimensionality statement in Corollary~\ref{corsnfinite} below that is known for admissible skein modules).
	We can, however, use our results to infer results about admissible skein modules, see Remark~\ref{remasm2}. 
	\end{remark}

As just mentioned, one could use the comparison to admissible skein modules to derive from 
\cite[Section~5.4]{asm} that the vector spaces $\snc(\Sigma;B)$ for any surface $\Sigma$ and any boundary label $B$ are always finite-dimensional. 
For completeness, we prove the result here as a consequence of the additivity statement from Proposition~\ref{propadd}:

	\begin{corollary}\label{corsnfinite}	
		Let $\cat{C}$ be a pivotal finite tensor category and $\Sigma$ a 
		surface with boundary label $B$.
		Then the string-net space $\snc(\Sigma;B)$ is finite-dimensional. 
	\end{corollary}

	\begin{proof}
We fix a graph $\Gamma$ embedded in $\Sigma$ with at least one edge and one vertex such that its intersection with the boundary is exactly given by
the marked points prescribed by the boundary value $B$ and such that $\Sigma$ (minus one point in case the boundary is empty) admits a deformation retract onto $\Gamma$. 
To obtain such a graph, we
use a ribbon graph description for $\Sigma$ or, in the case that $\Sigma$ is closed, $\Sigma$ with an additional free boundary component, see e.g.~\cite{costellographs} for the ribbon graph description of the moduli space of surfaces.

 Using the deformation retract and merging internal edges/vertices whenever necessary, we can transform every string-net drawn on $\Sigma$ such that its underlying graph becomes a subgraph of $\Gamma$. 

Since $\cat{C}$ is finite by assumption, we can choose a projective generator $A\in \cat{C}$.
For every projective object $P$ there exist an $m\ge 0$ and maps $\iota \colon P\to A^{\oplus m }$ and $\pi\colon A^{\oplus m }\to P $ such that $\pi \circ \iota=\id_P$. Let $(\Gamma,\underline{X},\underline{f})$ be a string-net with an edge labeled by $P$ representing an equivalence class in the string-net space. Then there is another string-net 
in which the same edge is labeled by $A^{\oplus m }$, which is constructed by `inserting the relation $\pi \circ \iota=\id_P$ on the edge'; the precise description is given in~Figure~\ref{figfinite}.  
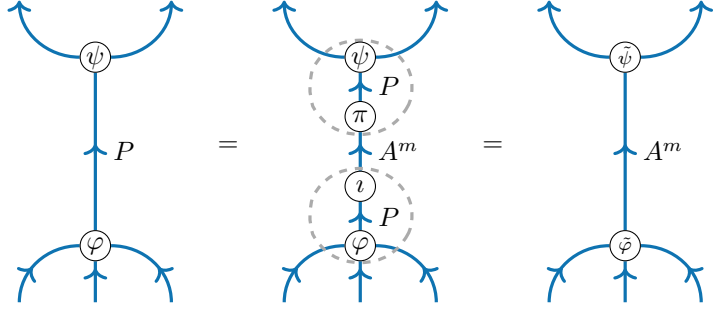
\begin{figure}[h]
\[
\begin{tikzpicture}[baseline=-0.25em,scale=0.5]
	\begin{pgfonlayer}{nodelayer}
		\node [style=nc] (45) at (0, 2.5) {};
		\node [style=nc] (46) at (0, -2.5) {};
		\node [style=none] (47) at (0, -4) {};
		\node [style=none] (48) at (-2, -4) {};
		\node [style=none] (49) at (2, -4) {};
		\node [style=none] (50) at (-2, 4) {};
		\node [style=none] (51) at (2, 4) {};
		\node [style=none] (52) at (0, 2.5) {$\psi$};
		\node [style=none] (53) at (0, -2.5) {$\varphi$};
		\node [style=none] (54) at (0, -0.05) {};
		\node [style=none] (55) at (0, 0.2) {};
		\node [style=none] (56) at (-1.75, -3.2) {};
		\node [style=none] (57) at (-1.65, -3.075) {};
		\node [style=none] (58) at (0, -3.225) {};
		\node [style=none] (59) at (0, -3.1) {};
		\node [style=none] (60) at (1.75, -3.2) {};
		\node [style=none] (61) at (1.65, -3.075) {};
		\node [style=none] (62) at (0.75, 0) {$P$};
	\end{pgfonlayer}
	\begin{pgfonlayer}{edgelayer}
		\draw [style=vt-bl] (46) to (45);
		\draw [style=vt-bl-di, in=-90, out=180] (45) to (50.center);
		\draw [style=vt-bl-di, in=-90, out=0] (45) to (51.center);
		\draw [style=vt-bl, in=-180, out=90] (48.center) to (46);
		\draw [style=vt-bl] (47.center) to (46);
		\draw [style=vt-bl, in=0, out=90] (49.center) to (46);
		\draw [style=vt-bl-di] (54.center) to (55.center);
		\draw [style=vt-bl-di] (58.center) to (59.center);
		\draw [style=vt-bl-di] (56.center) to (57.center);
		\draw [style=vt-bl-di] (60.center) to (61.center);
	\end{pgfonlayer}
\end{tikzpicture}\quad = \quad \begin{tikzpicture}[baseline=-0.25em,scale=0.5]
	\begin{pgfonlayer}{nodelayer}
		\node [style=nc] (45) at (0, 2.5) {};
		\node [style=nc] (46) at (0, -2.5) {};
		\node [style=none] (47) at (0, -4) {};
		\node [style=none] (48) at (-2, -4) {};
		\node [style=none] (49) at (2, -4) {};
		\node [style=none] (50) at (-2, 4) {};
		\node [style=none] (51) at (2, 4) {};
		\node [style=none] (52) at (0, 2.5) {$\psi$};
		\node [style=none] (53) at (0, -2.5) {$\varphi$};
		\node [style=none] (54) at (0, -0.05) {};
		\node [style=none] (55) at (0, 0.2) {};
		\node [style=none] (56) at (-1.75, -3.2) {};
		\node [style=none] (57) at (-1.65, -3.075) {};
		\node [style=none] (58) at (0, -3.225) {};
		\node [style=none] (59) at (0, -3.1) {};
		\node [style=none] (60) at (1.75, -3.2) {};
		\node [style=none] (61) at (1.65, -3.075) {};
		\node [style=none] (62) at (1, 0) {$A^m$};
		\node [style=nc] (63) at (0, 0.925) {};
		\node [style=nc] (64) at (0, -0.925) {};
		\node [style=none] (65) at (0.75, -1.75) {$P$};
		\node [style=none] (66) at (0.75, 1.75) {$P$};
		\node [style=none] (67) at (0, 0.925) {$\pi$};
		\node [style=none] (68) at (0, -0.925) {$\imath$};
		\node [style=none] (69) at (0, 1.625) {};
		\node [style=none] (70) at (0, 1.875) {};
		\node [style=none] (71) at (0, -1.85) {};
		\node [style=none] (72) at (0, -1.6) {};
		\node [style=none] (73) at (0, 2.95) {};
		\node [style=none] (74) at (-1.25, 2.2) {};
		\node [style=none] (75) at (1.25, 2.2) {};
		\node [style=none] (76) at (0, 0.45) {};
		\node [style=none] (77) at (-1.25, 1.2) {};
		\node [style=none] (78) at (1.25, 1.2) {};
		\node [style=none] (79) at (0, -0.45) {};
		\node [style=none] (80) at (-1.25, -1.2) {};
		\node [style=none] (81) at (1.25, -1.2) {};
		\node [style=none] (82) at (0, -2.95) {};
		\node [style=none] (83) at (-1.25, -2.2) {};
		\node [style=none] (84) at (1.25, -2.2) {};
	\end{pgfonlayer}
	\begin{pgfonlayer}{edgelayer}
		\draw [style=vt-bl-di, in=-90, out=180] (45) to (50.center);
		\draw [style=vt-bl-di, in=-90, out=0] (45) to (51.center);
		\draw [style=vt-bl, in=-180, out=90] (48.center) to (46);
		\draw [style=vt-bl] (47.center) to (46);
		\draw [style=vt-bl, in=0, out=90] (49.center) to (46);
		\draw [style=vt-bl-di] (54.center) to (55.center);
		\draw [style=vt-bl-di] (58.center) to (59.center);
		\draw [style=vt-bl-di] (56.center) to (57.center);
		\draw [style=vt-bl-di] (60.center) to (61.center);
		\draw [style=vt-bl] (64) to (63);
		\draw [style=vt-bl] (46) to (64);
		\draw [style=vt-bl] (63) to (45);
		\draw [style=vt-bl-di] (69.center) to (70.center);
		\draw [style=vt-bl-di] (71.center) to (72.center);
		\draw [style=vt-sh-da] (78.center)
		to [bend left] (76.center)
		to [bend left] (77.center)
		to [bend left=15] (74.center)
		to [bend left] (73.center)
		to [bend left] (75.center)
		to [bend left, looseness=0.75] cycle;
		\draw [style=vt-sh-da] (84.center)
		to [bend left] (82.center)
		to [bend left] (83.center)
		to [bend left=15] (80.center)
		to [bend left] (79.center)
		to [bend left] (81.center)
		to [bend left, looseness=0.75] cycle;
	\end{pgfonlayer}
\end{tikzpicture}
 \quad = \quad \begin{tikzpicture}[baseline=-0.25em,scale=0.5]
 	\begin{pgfonlayer}{nodelayer}
 		\node [style=nc] (45) at (0, 2.5) {};
 		\node [style=nc] (46) at (0, -2.5) {};
 		\node [style=none] (47) at (0, -4) {};
 		\node [style=none] (48) at (-2, -4) {};
 		\node [style=none] (49) at (2, -4) {};
 		\node [style=none] (50) at (-2, 4) {};
 		\node [style=none] (51) at (2, 4) {};
 		\node [style=none] (52) at (0, 2.5) {\scriptsize$\tilde{\psi}$};
 		\node [style=none] (53) at (0, -2.5) {\scriptsize$\tilde{\varphi}$};
 		\node [style=none] (54) at (0, -0.05) {};
 		\node [style=none] (55) at (0, 0.2) {};
 		\node [style=none] (56) at (-1.75, -3.2) {};
 		\node [style=none] (57) at (-1.65, -3.075) {};
 		\node [style=none] (58) at (0, -3.225) {};
 		\node [style=none] (59) at (0, -3.1) {};
 		\node [style=none] (60) at (1.75, -3.2) {};
 		\node [style=none] (61) at (1.65, -3.075) {};
 		\node [style=none] (63) at (1, 0) {$A^m$};
 	\end{pgfonlayer}
 	\begin{pgfonlayer}{edgelayer}
 		\draw [style=vt-bl] (46) to (45);
 		\draw [style=vt-bl-di, in=-90, out=180] (45) to (50.center);
 		\draw [style=vt-bl-di, in=-90, out=0] (45) to (51.center);
 		\draw [style=vt-bl, in=-180, out=90] (48.center) to (46);
 		\draw [style=vt-bl] (47.center) to (46);
 		\draw [style=vt-bl, in=0, out=90] (49.center) to (46);
 		\draw [style=vt-bl-di] (54.center) to (55.center);
 		\draw [style=vt-bl-di] (58.center) to (59.center);
 		\draw [style=vt-bl-di] (56.center) to (57.center);
 		\draw [style=vt-bl-di] (60.center) to (61.center);
 	\end{pgfonlayer}
 \end{tikzpicture} 
\]
\caption{On the very left, we see a string-net with an edge labeled by a projective object $P$. Since $\pi \circ \iota = \id_P$, the string-net in the middle represents the same vector in the string-net space. In the last step, we perform an evaluation on the disks bounded by the dashed circles. This produces a 
	 representative for the string-net in which the edge initially labeled by $P$ is now labeled by $A^{\oplus m}$. 
}
\label{figfinite}
\end{figure}

This combined with Proposition~\ref{propadd}  implies
 that every vector
  in $\snc(\Sigma,B)$ can be represented by a linear combination of string-nets
   whose internal edges are all labeled by $A$ (the labels of the legs are fixed by the prescribed boundary value $B$).
		
Since the underlying graph of every string-net appearing in these linear combinations
 is a subgraph of $\Gamma$ (note that the subgraphs of $\Gamma$ form a finite set) and the space of labels for each vertex is isomorphic to a morphism space of $\cat{C}$ and therefore finite-dimensional, the string-net space $\snc(\Sigma,B)$ itself must be the quotient of a finite direct sum of finite tensor products of finite-dimensional vector spaces, hence finite-dimensional.
	\end{proof}

	\section{The open-closed string-net modular functor for a pivotal finite tensor category\label{secopclmf}}
	In this section we will prove  
	 that the non-semisimple string-net construction with projective boundary labels
	 as given in Section~\ref{secstringnet} produces an \emph{open-closed modular functor};
	 an informal reminder on the definition of a modular functor
	  plus precise references will be given below.
	 Large parts of this 
	 statement
	 will follow from the construction or are a straightforward adaption of
	 	\cite[Theorem~3.36]{fsy-sn}, with the exception being that the needed gluing result (`excision') does strictly speaking not follow from the existing gluing results as they appear in \cite{Walker} (see also \cite{cooke} and \cite[Section~2.1-2.3]{skeinfin})
	 	that use skein-theoretic constructions involving \emph{all} objects.
	 	After a free finite cocompletion (as we will discuss it in Section~\ref{freecomcompletion}), this would not lead to the desired results: it would  have the consequence that the category associated to the circle becomes semisimple.
	 
	 Let us now state the needed excision results for string-nets:
	
	\begin{lemma}[Excision for disks that are glued along an interval]\label{cyclicyonedalemma}
		Let $\cat{C}$ be a pivotal finite tensor category.
		The gluing of two disks $D$ and $D'$ along a boundary interval $I$
		 induces an isomorphism
		\begin{align}
			\int^{P\in\Proj\cat{C}} \snc(D;X,P)\otimes\snc(D';P^\vee,X') \ra{\cong} \snc(D\cup_I D'	;X,X') \ . 
		\end{align}
		Here the respective projective boundary labels $(X,P)$ and $(P^\vee,X')$ for $D$ and $D'$ are written such that $P$ and $P^\vee$ are exactly the part of the boundary label lying on $I$.
		It is required that the label $(X,X')$ has at least length one.
	\end{lemma}
	
	\begin{proof}
		Without loss of generality, we can assume that the label $X$ has length one.
		Then
		\begin{align}\snc(D;X,P)\cong \cat{C}(I,X\otimes P)\cong \cat{C}(X^\vee ,P) 
		\end{align}
		by Lemma~\ref{lemmasndisks}.
		We can understand $\cat{C}(X^\vee ,P)$ as the morphism space between $X^\vee$ and $P$ \emph{in $\Proj\cat{C}$}
		because both objects are projective.
		Now the statement follows from the (co-)Yoneda Lemma, see e.g.~\cite[Example~1.4.6]{riehl}.
	\end{proof}

	\begin{theorem}[Excision]\label{thmexcision}
		Let $\cat{C}$ be a pivotal finite tensor category.
		The gluing of a surface $\Sigma$ along a pair of 
		boundaries of shape $S$ (here $S$ is a finite disjoint union of boundary intervals and boundary circles) induces an isomorphism
		\begin{align}
			\int^{P \in \snc(S)} \snc(\Sigma;X,P,P^\vee)\ra{\cong} \snc(\Sigma';X) \quad \text{for}\quad X \in \snc(\partial \Sigma') \ , \label{eqnexcision} 
		\end{align}
		where $\Sigma'$ is the surface obtained by gluing and the boundary label $(X,P,P^\vee)$ is partitioned such that $P$ lives on the first copy $S$, $P^\vee$ on the second copy of $S$
		 and $X$ on the remaining boundary that is not involved in the gluing.
	\end{theorem}
	
	\begin{proof}
		The gluing maps $\snc(\Sigma;X,P,P^\vee)\to\snc(\Sigma';X)$ are clearly dinatural and hence induce the map~\eqref{eqnexcision}.
		
		Any admissible $\Proj\cat{C}$-labeled graph in $\Sigma'$ can be isotoped such that it intersects the gluing boundary only transversally and at least once. Thanks to isotopy invariance
		(Proposition~\ref{propisotopy}), this does not change the vector in the string-net space.
		After cutting the graph at the gluing boundary, we obtain an admissible $\Proj\cat{C}$-labeled graph in $\Sigma$ that after re-gluing recovers our original graph up to isotopy. This proves that~\eqref{eqnexcision} is surjective. 
		
		To prove that~\eqref{eqnexcision} is injective, we need to show that a relation arising through the evaluation on any disk $D\subset \Sigma'$ on the right hand side of~\eqref{eqnexcision} holds already on the left hand side. 
		To this end, denote by $S \subset \Sigma'$ the image of the gluing boundary in $\Sigma'$
		(by slight abuse of notation we use the same symbol as for the gluing boundary). If $D \cap S=\emptyset$, there is clearly nothing to show. 
		If $D \cap S\neq \emptyset$, we can slightly deform $D$ such that $\partial D$ and $S$ only intersect transversally without changing the relations induced by evaluation on $D$ (this is because the evaluation on $D$ only induces relations on string-nets that are transversal to $\partial D$, so we can always slightly adjust the shape of $D$). Next observe that $\partial D \cap S\subset S$ is compact, i.e.\ it is given by a disjoint union of finitely many closed intervals in $S$. But these intervals must each just consist of a point because of transversality. Therefore, $\partial D \cap S$ is a finite set. This implies, as one can see through an induction on the cardinality of $\partial D \cap S$,	 that $S$ cuts $D$ into finitely many disks $D_1,\dots,D_\ell$ 
		that can each be seen as disks in $\Sigma$. The finitely many disks intersect along closed intervals in their boundary and reproduce $D$ when glued along these closed intervals. 
		This entails 
		by Lemma~\ref{cyclicyonedalemma} that $\snc(D;-)$ is the coend of the functors $\snc(D_i;-)$, with one coend for the gluing along each interval. 
		This implies that
		the relations in $\snc(\Sigma';X)$ resulting from the evaluation on $D$ are exactly the following ones:
		\begin{itemize}\item The relations in $\snc(\Sigma;X,P,P^\vee)$ for each $P$ that arise from evaluation on the disks $D_1,\dots,D_\ell$ in $\Sigma$. But these relations hold obviously already in $\int^{P \in \snc(S)}  \snc(\Sigma;X,P,P^\vee)$.
			\item The relations that arise from moving string-nets through the intervals along which one has to glue the $D_1,\dots,D_\ell$ to get $D$, i.e.\ relations implemented through coends 
			\begin{align} \int^{Q\in \snc(I)} \snc(D_i;Q,Y_i) \otimes \snc(D_j;Q^\vee,Y_j)\end{align}
			 for all disks $D_i$ and $D_j$ with boundary labels $Y_i$ and $Y_j$ among the $D_1,\dots,D_m$
			that share a boundary interval $I \subset S$. 
			To see that these relations hold in
			$\int^{P \in \snc(S)}  \snc(\Sigma;X,P,P^\vee)$, we need to see that the map
			\small
			  \begin{align}
			  	\snc(D_i;Q,Y_i) \otimes \snc(D_j;Q^\vee,Y_j) \to  \snc(\Sigma;X,P,P^\vee) \to  \int^{P \in \snc(S)}  \snc(\Sigma;X,P,P^\vee) \label{eqnmapproofexc1}
			\end{align} 
		\normalsize induced by the embedding $D_i \cup D_j \subset \Sigma$ (we choose the labels compatibly) and the projection to the coend
		factors through $\int^{Q\in \snc(I)} \snc(D_i;Q,Y_i) \otimes \snc(D_j;Q^\vee,Y_j)$. This indeed follows from the fact that
		 the inclusion
			 $I\subset S$ induces a map \begin{align}
			 	\int^{Q\in \snc(I)} \snc(D_i;Q,Y_i) \otimes \snc(D_j;Q^\vee,Y_j) \to \int^{P \in \snc(S)}  \snc(\Sigma;X,P,P^\vee)\end{align}
		 	through which~\eqref{eqnmapproofexc1} clearly factors.
		\end{itemize}
		This concludes the proof that~\eqref{eqnexcision} is injective and therefore the proof of the statement.
	\end{proof}

For the notion of a modular functor, many different definitions exist~\cite{Segal,ms89,turaev,tillmann,baki}.
A sufficiently general one can be compactly defined using modular operads
 in the sense of Getzler-Kapranov~\cite{gkmod}.
 We first treat ordinary modular functors and then indicate the modifications needed for the open-closed case:	
A modular functor is a modular algebra over the modular operad of surfaces (or over a suitable extension of the modular surface operad, but this will not be relevant for the class of examples treated in this article). 
The modular algebra takes values in a suitable symmetric monoidal bicategory of linear categories.
This is the definition given in \cite[Section~7.3]{cyclic} or, in more generality, in~\cite[Section~3.1]{brochierwoike}.
There are different options for the symmetric monoidal target category of linear categories, and this is actually an important choice to make; we will expand on this in Section~\ref{freecomcompletion}.

We do not want to reproduce here all the details of the formal definition. Instead, we restrict ourselves to the rough picture, with attention to the case at hand. A modular functor will then consist of the following:

\begin{itemize}
	
	\item An object $\cat{A}$ in the symmetric monoidal bicategory of linear categories that we are considering.
	(If we are looking at $\Bimod$,  this would  just be a linear category.)
	We think of $\cat{A}$ as being the category associated to the circle.
	
	\item For a surface $\Sigma$ with $n \ge 0$ boundary circles, we get a 1-morphism $\mathfrak{F}(\Sigma) : \cat{A}^{\otimes n} \to I$, where $\otimes$ is the monoidal product of linear categories that we are using and $I$ is the monoidal unit of our symmetric monoidal bicategory of linear categories that we choose.
	The mapping class group of $\Sigma$ acts on $\mathfrak{F}(\Sigma)$ through 2-isomorphisms. 
	
	\item There is a compatibility with gluing. 
	To this end, we will need a non-degenerate symmetric pairing
	$\kappa : \cat{A} \otimes \cat{A} \to I$. Non-degeneracy means that there is a coevaluation $\Delta : I \to \cat{A}\otimes\cat{A}$ that together with $\kappa$ satisfies the usual zigzag identities up to isomorphism. Symmetry means that coherent isomorphisms $\kappa(X,Y)\cong \kappa(Y,X)$ are part of the data.
		\end{itemize}
	The surface $\Sigma$ appearing above is
	 traditionally not an open-closed surface, 
	 but we will actually need this case, thereby leading us to
	the notion of an \emph{open-closed modular functor}. In that case, there will be a second
	 linear category $\cat{B}$
	(the category associated to a boundary interval), and if $\Sigma$ has, in addition to its $n$ boundary circles, 
	$m$ 
	open boundaries, we get a 1-morphism $\mathfrak{F}(\Sigma):\cat{A}^{\otimes n} \otimes \cat{B}^{\otimes m} \to I$. 
	Again, the mapping class group of $\Sigma$ acts through 2-isomorphisms on this functor. 
	
	A modular functor a priori only allows us to assign morphisms $\mathcal{A}^{\otimes n} \to I$ to surfaces $\Sigma$ with $n$ `incoming' boundaries. However, since the pairing $\kappa$ 
	 yields an equivalence between $\cat{A}$ and $\cat{A}^\opp$,
	  we can also assign morphisms $\mathcal{A}^{\otimes p}\to \mathcal{A}^{\otimes q}$ to $\Sigma$, which we interpret as the value of the modular functor on the cobordism $\Sigma$ from $p$ `incoming' circles to $q$ `outgoing' circles, see~\cite[Remark 2.2]{mwansular} for more details. This process of turning incoming to outgoing boundaries and vice versa can be extended to the open-closed case.
	
	The result that we want to record in this section is the following: For a pivotal finite tensor category $\cat{C}$, the string-net construction $\snc$ from 
	Section~\ref{secstringnet} will give us an open-closed modular functor with values in $\Bimod$.
	The category associated to a boundary circle is $\snc(\mathbb{S}^1)$; the category associated to a boundary interval is $\snc([0,1])$.
	For a surface with $n$ closed boundary components and $m$ open boundary components, we get by~\eqref{eqnfunctorforsurface0}
	a linear functor $\sn(\Sigma):\snc(\partial\Sigma)\to\Vect$
	on which the mapping class group acts (Corollary~\ref{cormcg}).
	Since $\snc(\partial \Sigma) \simeq \snc(\mathbb{S}^1)^{\otimes n} \otimes \snc([0,1])^{\otimes m}$ (Remark~\ref{remsncsymmon}), this is exactly what we need for an open-closed modular functor in $\Bimod$.
	The compatibility with gluing amounts exactly to the excision result (Theorem~\ref{thmexcision}). 
	The needed non-degenerate symmetric pairings 
	on $\snc([0,1])$
	and $\snc (\mathbb{S}^1)$
	are given by the string-nets on an elbow which is the cobordism $E_S\colon S\sqcup \overline{S} \to \emptyset $ 
	built from $[0,1]\times S$ and the orientation-reversing reflection 
	diffeomorphism:
	\small
	\begin{align}
	\kappa : \snc([0,1]) \otimes \snc([0,1]) \simeq \Proj \cat{C} \otimes \Proj \cat{C} & \to \Vect \\ \quad (X,Y) &\mapsto \snc ( E_{[0,1]}; X,Y) \cong \cat{C}(I,X\otimes Y)\\
\beta : \snc (\mathbb{S}^1) \otimes \snc (\mathbb{S}^1) & \to \Vect \\ 
\left( (X_1,\dots,X_n), (Y_1, \dots,Y_m)\right) &\mapsto \snc ( E_{\mathbb{S}^1}; \underline{X}, \underline{Y}) \cong \snc (\mathbb{S}^1)\left((X_n^\vee, \dots , X_1^\vee) , (Y_1,\dots Y_n) \right) 
\end{align}\normalsize
Here we used Lemma~\ref{Lemma: or rev} for the identifications. 
Put differently, the rigid duality on $\snc([0,1])\simeq \Proj \cat{C}$ induces a non-degenerate symmetric pairing on $\snc(\mathbb{S}^1)$. 
	
We can now summarize:

\spaceplease
\begin{theorem}\label{thmsnmf}
	Let $\cat{C}$ be a pivotal finite tensor category.
	By assigning
	to a surface $\Sigma$ the linear functor $\sn(\Sigma):\snc(\partial\Sigma)\to\Vect$ from~\eqref{eqnfunctorforsurface0}
	one obtains a $\Bimod$-valued
	open-closed modular functor that we denote by $\snc$. 
\end{theorem}

\spaceplease
\section{The string-net description of the central monad\label{seccentralmonad}}
By an important result of Day and Street \cite{daystreet} refined by Brugières and Virelizier~\cite{bruguieresvirelizier}, one can build 
for any finite tensor category $\cat{C}$
a monad,
\begin{align} Z:\cat{C}\to\cat{C} \ , \quad X \mapsto \int^{Y\in\cat{C}}Y^\vee \otimes X \otimes Y\label{eqncentralmonad} \end{align}
 the
so-called
\emph{central monad},
 whose Eilenberg-Moore category of algebras is equivalent to the Drinfeld center
 (for simplicity, we continue to assume pivotality of $\cat{C}$ throughout). 
The central monad is defined in algebraic terms, and the purpose of this section is to give a topological interpretation of the central monad in terms of string-nets. 
We will use this to describe explicitly the category $\snc(\mathbb{S}^1)$
that the string-net construction associates to the circle.
The idea is of course to relate this circle category to the Drinfeld center. 
This was done in the spherical fusion case in~\cite{kirillovsn} without the use of monads. 
Monadic techniques for the reconstruction of the circle category as a Kleisli category of some version of the central monad are used in~\cite{kst} in a framed setting. 
The construction of that article is applied to all objects of a not necessarily semisimple finite tensor category. We have to repeat here the warning from the beginning of Section~\ref{secopclmf} that, after the finite cocompletion that we will have to perform later to be able to compare to the Lyubashenko construction, applying the construction to all objects would produce a modular functor whose category on the circle is inevitably semisimple
(this is not an issue in~\cite{kst} because a different cocompletion is considered).
Therefore, the results of~\cite{kst} do not  apply directly.
 They serve however as a guiding principle for this section. 

In order to discuss all of this in more detail, recall e.g.\ from
\cite[Section~VI.1]{maclane} that a
 monad on a category $\cat{D}$ is an endofunctor $T:\cat{D}\to\cat{D}$ 
equipped with natural transformations $\id_\cat{D}\to T$ and $T^2 \to T$ making $T$ a unital and associative algebra in the monoidal category of endofunctors of $\cat{D}$. For a monad $T$ on $\cat{D}$, one can define the \emph{Kleisli category} $\catf{KL}(T)$ of $T$ whose objects are the objects of $\cat{D}$ with the hom sets given by $\catf{KL}(T)(d,d'):=\cat{D}(d,Td')$ for $d,d'\in\cat{D}$ (the composition comes from the monad structure on $T$). We can also define the \emph{Eilenberg-Moore category $\catf{EM}(T)$ of $T$-algebras}, i.e.\ the category of objects $d\in \cat{D}$ equipped with a morphism $f:Td\to d$ in $\cat{D}$ subject to a condition. A morphism $(d,f)\to (d',f')$ of $T$-algebras is a morphism $g:d\to d'$ in $\cat{D}$ with $f'T(g)=gf$. 

The central monad $Z:\cat{C}\to\cat{C}$ of a pivotal finite tensor category $\cat{C}$ is defined via the coend~\eqref{eqncentralmonad}, with the monad structure induced by the tensor product of dummy variables and the isomorphism $(P\otimes Q)^\vee \cong Q ^\vee \otimes P^\vee$. By \cite{daystreet} the Eilenberg-Moore category of $Z$-algebras is linearly equivalent to the Drinfeld center, i.e.\ the 
braided monoidal category of objects $X\in\cat{C}$ equipped with a half braiding (a natural isomorphism $X \otimes - \cong - \otimes X$ satisfying the usual hexagon relations of braidings). 
Under the equivalence \begin{align} \catf{EM}(Z)\simeq Z(\cat{C})\label{eqnEMequiv} \end{align} the free-forgetful adjunction between $\cat{C}$ and $\catf{EM}(Z)$
(the forgetful functor forgets the algebra structure; its left adjoint freely adds one) translates to the adjunction between the forgetful functor $U:Z(\cat{C})\to\cat{C}$ forgetting the half braiding and its left adjoint $F:\cat{C}\to Z(\cat{C})$
that sends $X\in\cat{C}$ to $\int^{Y\in \cat{C}} Y^\vee \otimes X \otimes Y$ equipped with the \emph{non-crossing half braiding}
that is explained in~Figure~\ref{fignoncrossing}.

\begin{figure}[h]
\[
\begin{tikzpicture}[baseline=-0.25em,scale=0.5]
	\begin{pgfonlayer}{nodelayer}
		\node [style=none] (28) at (-1, -0.75) {};
		\node [style=none] (29) at (1, -0.75) {};
		\node [style=none] (30) at (0, -0.75) {};
		\node [style=none] (31) at (0, -4) {};
		\node [style=none] (32) at (0, 4) {};
		\node [style=none] (33) at (-1, -4) {};
		\node [style=none] (34) at (1, -4) {};
		\node [style=none] (35) at (-0.25, 2) {};
		\node [style=none] (36) at (0.25, 1.75) {};
		\node [style=none] (37) at (2.5, -4) {};
		\node [style=none] (38) at (-1, -4.5) {$Y^\vee$};
		\node [style=none] (39) at (1, -4.5) {$Y$};
		\node [style=none] (40) at (2.5, -4.5) {$W$};
		\node [style=none] (41) at (-2.5, 4) {};
		\node [style=none] (42) at (-2.5, 4.5) {$W$};
		\node [style=none] (0) at (-1.25, -0.75) {};
		\node [style=none] (1) at (1.25, -0.75) {};
		\node [style=none] (2) at (0, 0.75) {};
		\node [style=none] (43) at (0, -4.5) {$X$};
		\node [style=none] (44) at (0, 4.5) {$UF(X)$};
		\node [style=none] (45) at (0, -0.175) {$\imath_Y$};
	\end{pgfonlayer}
	\begin{pgfonlayer}{edgelayer}
		\draw [style=vt-gr] (2.center) to (32.center);
		\draw [style=vt-bl] (31.center) to (30.center);
		\draw [style=vt-bl] (33.center) to (28.center);
		\draw [style=vt-bl] (29.center) to (34.center);
		\draw [style=vt-bl, in=90, out=-30] (36.center) to (37.center);
		\draw [style=vt-bl, in=150, out=-90] (41.center) to (35.center);
		\draw [style=vt] (1.center)
		to [in=0, out=90] (2.center)
		to [in=90, out=-180] (0.center)
		to cycle;
	\end{pgfonlayer}
\end{tikzpicture}=\quad\begin{tikzpicture}[baseline=-0.25em,scale=0.5]
\begin{pgfonlayer}{nodelayer}
	\node [style=none] (28) at (-0.5, -0.75) {};
	\node [style=none] (29) at (0.5, -0.75) {};
	\node [style=none] (30) at (0, -0.75) {};
	\node [style=none] (31) at (0, -4) {};
	\node [style=none] (32) at (0, 4) {};
	\node [style=none] (33) at (-1, -4) {};
	\node [style=none] (34) at (1, -4) {};
	\node [style=none] (37) at (2.5, -4) {};
	\node [style=none] (38) at (-1, -4.5) {$Y^\vee$};
	\node [style=none] (39) at (1, -4.5) {$Y$};
	\node [style=none] (40) at (2.5, -4.5) {$W$};
	\node [style=none] (41) at (-2.5, 4) {};
	\node [style=none] (42) at (-2.5, 4.5) {$W$};
	\node [style=none] (43) at (0, -4.5) {$X$};
	\node [style=none] (44) at (0, 4.5) {$UF(X)$};
	\node [style=none] (45) at (1, -0.75) {};
	\node [style=none] (46) at (-1, -0.75) {};
	\node [style=none] (47) at (-2.5, -0.75) {};
	\node [style=none] (0) at (-1.25, -0.75) {};
	\node [style=none] (1) at (1.25, -0.75) {};
	\node [style=none] (2) at (0, 0.75) {};
	\node [style=none] (48) at (0, -0.175) {$\imath_{Y\otimes W}$};
\end{pgfonlayer}
\begin{pgfonlayer}{edgelayer}
	\draw [style=vt-gr] (2.center) to (32.center);
	\draw [style=vt-bl] (31.center) to (30.center);
	\draw [style=vt-bl, in=270, out=90] (33.center) to (28.center);
	\draw [style=vt-bl, in=90, out=-90] (29.center) to (34.center);
	\draw [style=vt-bl, bend right=90, looseness=1.75] (47.center) to (46.center);
	\draw [style=vt-bl] (47.center) to (41.center);
	\draw [style=vt-bl, in=90, out=-90] (45.center) to (37.center);
	\draw [style=vt] (1.center)
	to [in=0, out=90] (2.center)
	to [in=90, out=-180] (0.center)
	to cycle;
\end{pgfonlayer}
\end{tikzpicture}
\]
\caption{The \emph{non-crossing} half braiding (depicted by a crossing of $UF(X)$ over an arbitrary object $W\in\cat{C}$) is defined by the above formula, where the half-disk shaped coupons stand for the components of the universal dinatural transformation
	that the coend comes equipped with.}\label{fignoncrossing}
\end{figure}
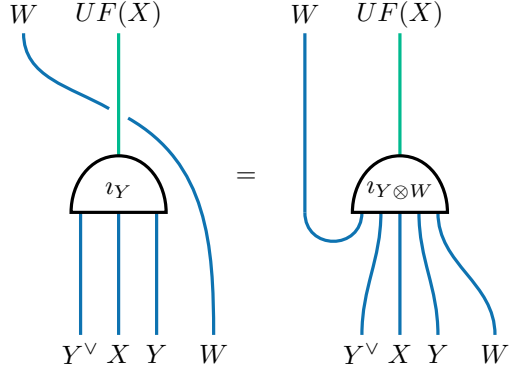

Through the string-net construction, we will not directly find the central monad, but its restriction to $\Proj\cat{C}$. The fact that the central monad correctly restricts is a priori not obvious, but can be shown easily:

\begin{lemma}\label{lemmarestrictmonad}
	For any pivotal finite tensor category $\cat{C}$, the central monad $Z$ restricts to a monad $z:\Proj\cat{C}\to\Proj\cat{C}$ on $\Proj\cat{C}$. 
	\end{lemma}
	
	\begin{proof}Clearly, we can restrict $Z$ to $\Proj \cat{C}$, but we have to verify that it correctly restricts to $\Proj\cat{C}$ in range.
		To this end, recall that both the forgetful functor $U:Z(\cat{C})\to\cat{C}$ and its left adjoint $F:\cat{C}\to Z(\cat{C})$ are exact
		\cite[Corollary~6.9]{shimizuunimodular} and their composition is $UF=Z$. This implies that $F$ preserves projective objects and that $U$ preserves injective objects (which however in $\cat{C}$ and $Z(\cat{C})$ are the same as the projective ones by self-injectivity of finite tensor categories). 
		Therefore $Z=UF$ preserves projective objects. 
		\end{proof}
	
	We can now prove the following generalization of \cite[Theorem~6.3]{kst}:

\begin{proposition}[Central monad from string-nets]\label{propcentralmonad}
	Let $\cat{C}$ be a pivotal finite tensor category.
	Consider projective objects $P$ and $Q$ of $\cat{C}$, each seen as an object in $\snc(\mathbb{S}^1)$. 	
	Then
	\begin{align}
		\snc(\mathbb{S}^1) (P,Q) \cong \cat{C}(P,zQ) \    \label{eqnkleislieqn}
		\end{align}
	for the restriction  $z : \Proj \cat{C}\to\Proj\cat{C}$ of the central monad from Lemma~\ref{lemmarestrictmonad}.
	The composition in  $\snc(\mathbb{S}^1) (P,Q)$ comes from the monad structure on $z$.
	In this way, \eqref{eqnkleislieqn} exhibits $\snc(\mathbb{S}^1)$ as the Kleisli category of $z$:
	\begin{align} \snc(\mathbb{S}^1) \simeq \mathsf{KL}(z) \  .
		\end{align}
	\end{proposition}

\begin{proof}
	Excision from Theorem~\ref{thmexcision}
	 and $\snc([0,1])\simeq \Proj\cat{C}$
	from Remark~\ref{remsncint} tells us
	\begin{align}
	\label{eqnexc2}	\snc(\mathbb{S}^1) (P,Q) = \snc(\mathbb{S}^1 \times [0,1] ; P^\vee , Q) \cong \int^{X \in \Proj\cat{C}} \cat{C}(P, X^\vee\otimes  Q \otimes X) \ . 
		\end{align}
	The coend over $\Proj \cat{C}$ can be expressed as finite colimit via the \emph{Agreement Principle} \cite{mcarthy,keller}, see
	\cite[Theorem~2.9]{dva} for a version adapted for finite (tensor) categories. 
	Since $P$ is projective, $\cat{C}(P,-)$ is exact and preserves finite colimits.
	This implies
	\begin{align}
		\snc(\mathbb{S}^1) (P,Q) &\cong \cat{C}\left(P, \int^{X\in\Proj\cat{C}} X^\vee \otimes Q \otimes X\right)\\&\cong \cat{C}\left(P, \int^{X\in \cat{C}} X^\vee \otimes Q \otimes X\right)\\&=\cat{C}(P,zQ) \ . \label{eqnexc3}	
		\end{align}
	The replacement of $\int^{X\in\Proj\cat{C}}$ with $\int^{X\in \cat{C}}$ is possible in this particular case because of the exactness of $\otimes$ \cite[Proposition~5.1.7]{kl}.

	The functor $z$ is the restriction of the central monad
	(Lemma~\ref{lemmarestrictmonad}), but we still need to verify that the composition of morphisms  in $\snc(\mathbb{S}^1)$ comes from the monad structure on $z$.
	The composition in $\snc(\mathbb{S}^1)$ can also be described via excision, and the result is the following: 
	Under the excision isomorphism~\eqref{eqnexc2},
	 the composition is
	\begin{align}
	&	\int^{X \in \Proj\cat{C}} \cat{C}(P, X^\vee\otimes  Q \otimes X) \otimes \int^{Y \in \Proj\cat{C}} \cat{C}(Q, Y^\vee\otimes  R \otimes Y)\\
	\cong &	\int^{X,Y \in \Proj\cat{C}} \cat{C}(P, X^\vee\otimes  Q \otimes X) \otimes  \cat{C}(Q, Y^\vee\otimes  R \otimes Y)
		 \to \int^{U \in \Proj\cat{C}} \cat{C}(P, U^\vee\otimes  R \otimes U) \ , 
		\end{align}
	where the first isomorphism is 
	the Fubini isomorphism for coends, and the second one is  induced by composition in $\snc([0,1])$ under the coend, i.e.\ it amounts to the composition over $Q$ and by tensoring $X$ and $Y$ together to obtain a new dummy variable $U=X\otimes Y$ (this uses $(X\otimes Y)^\vee \cong Y^\vee \otimes X^\vee$). Under the isomorphism~\eqref{eqnexc3}, the composition in $\snc(\mathbb{S}^1)$ is therefore induced by the monad structure
	\begin{align}
		\int^{X,Y \in\cat{C}} Y^\vee \otimes X ^\vee \otimes P \otimes X \otimes Y \to \int^{U\in\cat{C}} U^\vee \otimes P \otimes U 
		\end{align}
	on $z:\Proj\cat{C}\to\Proj\cat{C}$
	coming from the monoidal product and the fact that $-^\vee$ is opp-monoidal. 
	But this is exactly the restriction of the central monad structure. 
	
	Since all objects in $\snc(\mathbb{S}^1)$ are isomorphic to some projective object in $\cat{C}$ (seen as object in $\snc(\mathbb{S}^1)$ by placing it on the circle), we obtain from~\eqref{eqnkleislieqn}, together with the fact that the composition in $\snc(\mathbb{S}^1)$ is the one coming from the monad structure of $z$,
	 that $\snc(\mathbb{S}^1)$ is the Kleisli category of $z$. 
	\end{proof}

	The Kleisli category of a monad $T$ on a category $\cat{D}$ comes with a functor $R: \catf{KL}(T)\to \cat{D}$ sending $d \in \catf{KL}(T)$ to $Td$ and a morphism $f$ from $d$ to $d'$ in $\catf{KL}(T)$, i.e.\ a morphism $f:d\to Td'$ in $\cat{D}$ to $Td \ra{T(f)} T^2d' \ra{T^2 \to T} Td'$. This functor is right adjoint to the functor $L:\cat{D}\to \catf{KL}(T)$ sending $d\in \cat{D}$ to $d$ and a morphism $g : d \to d'$ in $\cat{D}$ to $d \ra{g} d' \ra{\text{unit}\ \id_\cat{D}\to T} Td'$. 	
	The adjunction $L\dashv R$ induces the monad $T=RL$, and in fact, it is initial among all such adjunctions, see e.g.\
	\cite[Proposition~5.2.12]{riehlcat}.
We hence conclude from Proposition~\ref{propcentralmonad}:	Since $\snc(\mathbb{S}^1)$ is the Kleisli category of $z$, it comes with an adjunction $\ell \dashv r$ between $\ell : \Proj\cat{C}\to \snc(\mathbb{S}^1)$ and $r : \snc(\mathbb{S}^1) \to \Proj\cat{C}$.

	 The Eilenberg-Moore category of $T$, with its free-forgetful adjunction, is terminal among all adjunctions inducing $T$.
	 Fortunately, the Eilenberg-Moore category of the restricted central monad $z$ can be described explicitly:

\begin{lemma}\label{lemmaem}
	For any pivotal finite tensor category $\cat{C}$,
	the Eilenberg-Moore category $\catf{EM}(z)$ of the restricted central monad
	$z:\Proj\cat{C}\to\Proj\cat{C}$
	can be canonically identified with the full subcategory
	$U\text{-}\Proj Z(\cat{C})$ of the Drinfeld center $Z(\cat{C})$ of all $U$-projective objects, i.e.\ all those $X\in Z(\cat{C})$ such that $UX$ is projective, where $U:Z(\cat{C})\to\cat{C}$ is the forgetful functor.
	\end{lemma}

\begin{proof}
	Denote by $\catf{EM}(Z)$ the Eilenberg-Moore category of the central monad $Z:\cat{C}\to \cat{C}$ and by $F\dashv U$ the free-forgetful adjunction associated to it. Now the $U$-projective objects in $\catf{EM}(Z)$ are exactly pairs of a projective object $P\in \cat{C}$ and a map $ZP \to P$ making $P$ a $Z$-module. By the definition of $z$ a map $ZP\to P$ is  just a map $zP \to P$. For this reason, $U\text{-}\Proj 	\catf{EM}(Z) = \catf{EM}(z)$.
	
	But by \cite{daystreet} $\catf{EM}(Z)$ is equivalent to $Z(\cat{C})$ such that $U$ takes the role of the forgetful functor $Z(\cat{C})\to\cat{C}$ and $F$ the role of its left adjoint. This implies the assertion. 
	\end{proof}

We denote by $f: \Proj\cat{C}\to U\text{-}\Proj Z(\cat{C})$ and $u:U\text{-}\Proj Z(\cat{C}) \to \Proj\cat{C}$ the restriction of $F:\cat{C}\to Z(\cat{C})$ and
$U:Z(\cat{C})\to \cat{C}$, respectively. It is a consequence of Lemma~\ref{lemmarestrictmonad} and its proof that these restrictions exist.

	\begin{theorem}\label{thmfunctorH}
	For any pivotal finite tensor category $\cat{C}$, there is a unique embedding
	\begin{align}
		h : \snc(\mathbb{S}^1)\to U\text{-}\Proj Z(\cat{C})
	\end{align} 
	such that
	$h\ell \cong f$ and $uh\cong r$. 
	The essential image of $h$ consists exactly of those $X\in Z(\cat{C})$ that are isomorphic to objects of the form $F(P)$ for $P\in\Proj\cat{C}$, i.e.\
	$h$ induces an equivalence
	\begin{align}
		h : \snc(\mathbb{S}^1) \ra{\simeq} F(\Proj\cat{C})  \label{eqnsmallh}
	\end{align}
	after restriction in range.
\end{theorem}

\begin{proof}
	For existence of uniqueness of the functor $h$, one applies \cite[Proposition~5.2.12]{riehlcat} to Proposition~\ref{propcentralmonad} and
	Lemma~\ref{lemmaem}. From \cite[Lemma~5.2.13]{riehlcat}, one deduces that $h$ is fully faithful  as well as the description of its essential image. 
\end{proof}

	\spaceplease
\section{The finitely cocompleted string-net construction\label{freecomcompletion}}

So far, we have seen that the string-net construction gives rise to an open-closed modular functor with values in the category $\Bimod$ (Theorem~\ref{thmsnmf}).
For the comparison to other constructions of modular functors,
working with bimodules is slightly inconvenient. For this reason, we explain in this section how to pass from a sufficiently finite subcategory of $\Bimod$ to other bicategories of linear categories. 
This will be accomplished through a well-known process called \emph{free finite cocompletion}, see e.g.~\cite{daylack} for a reference.
The material up to Proposition~\ref{Prop: Equivalence Bimod Rexf} is standard, or a minor variation of standard material that we give here for completeness and to fix notation.

\begin{definition}\label{defAhut}
For any linear category $\cat{A}$ with finite-dimensional morphism spaces, we 
 define the \emph{category of right $\cat{A}$-modules} as the  category $\widehat{\cat{A}}= \Cat_k(\cat{A}^{\opp}, \vect)$ of linear functors from $\cat{A}^\opp$ to the category $\vect$ of finite-dimensional $k$-vector spaces
 (in different contexts, these would be called $\vect$-valued presheaves).
The linear category $\widehat{\cat{A}}$ is called the \emph{finite cocompletion} of $\cat{A}$.
\end{definition}

\begin{remark}\label{Rem: free finite}
Through $\cat{A}\ni a \mapsto \cat{A}(-,a)$ we obtain an embedding $\iota_\cat{A} : \cat{A}\to\widehat{\cat{A}}$, the \emph{Yoneda embedding}. Since presheaves on $\cat{A}$ with values in \emph{finite-dimensional} vector spaces are compact objects in the (non-finite) free cocompletion $\Cat_k(\cat{A}^{\opp}, \Vect)$, they are \emph{finite} colimits of  representable objects. On the other hand, finite colimits of representable presheaves on $\cat{A}$ (note that $\cat{A}$
 has finite dimensional hom spaces by assumption) are valued in $\vect$. Therefore, $\widehat{\cat{A}}$ is indeed the finite cocompletion in the conventional sense:
The  category $\widehat{\cat{A}}$
 has the universal property that a linear functor $\cat{A}\to\cat{B}$ with finitely cocomplete target extends uniquely up to canonical isomorphism along $\iota _\cat{A}: \cat{A}\to\widehat{\cat{A}}$ to a right exact  functor $\widehat{\cat{A}}\to\cat{B}$ (a linear functor preserving finite colimits). 
 If $\cat{A}$ has only one object whose endomorphism algebra is $A$, the finite cocompletion is
 the category of finite-dimensional right $A$-modules.
 \end{remark}

 We call a linear category $\cat{A}$ \emph{pre-finite} if its category of right modules is a finite category (the definition of a finite category
 was recalled on page~\pageref{finitelinearpage}).

\begin{lemma}\label{Lem: finite cat}
An object $\cat{A}$ of $\Bimod$ is pre-finite if and only if it is equivalent (in $\Bimod$ this means: Morita equivalent) to a linear category $BA$ with one object and finite-dimensional endomorphism algebra $A$ via finite-dimensional bimodules. 	
\end{lemma}

\begin{proof}
Let us assume that there is a finite-dimensional algebra $A$ and a Morita equivalence 
${}_{A}M_\cat{A}$ from $BA$ to $\mathcal{A}$ (an equivalence in $\Bimod$). It is standard that $M$ induces an equivalence between 
the cocompletions $\Cat_k(BA,\Vect)\simeq \Cat_k(\mathcal{A}, \Vect)$ which sends a right $A$-module 
$N_A$ to the $\cat{A}$-module $N_A \otimes_A ({}_A M_\cat{A})$, see e.g.~\cite[Section~2.1]{skeinfin}. This map preserves the class of finite-dimensional modules because $M$ is finite-dimensional by assumption and hence also induces an 
equivalence between the finite cocompletions. 

Conversely, assume that $\cat{A}$ is pre-finite. We pick a projective generator $P\in \widehat{\cat{A}}$ and consider the $\cat{A}$-$B\End(P)$-bimodule $\widehat{\cat{A}}(\iota_{\cat{A}}(-),P)$, which is a  Morita equivalence via a finite-dimensional bimodule. 
\end{proof}

\begin{definition}
	We define $\Bimod^\catf{f}\subset \Bimod$ as the subcategory
	whose objects are pre-finite 
	linear categories, with bimodules taking values in finite-dimensional vector spaces as 1-morphisms and all 2-morphisms in $\Bimod$ between those 1-morphisms.
\end{definition}    
Note that the composition in $\Bimodf$ is well-defined thanks to Lemma~\ref{Lem: finite cat}. 
\begin{definition}
	We denote by $\Rexf$ the bicategory of finite categories, right exact 
 functors and natural transformations. The monoidal product is the Deligne product and the monoidal unit is the category $\vect$ of finite-dimensional vector spaces.    
\end{definition}
There is a natural symmetric monoidal functor $\Proj(-)\colon \Rexf \to \Bimodf$ sending a finite category to its subcategory of projective objects and a functor $F\colon \Ca\to \cat{D}$ to the bimodule
\begin{align}
\Proj(F)\colon \Proj(\Ca)\otimes \Proj(\cat{D})^{\opp} & \to \vect \\ 
c\otimes d &\longmapsto \cat{D}(d,F(c))  \ \ .
\end{align} 
The assignment
$\widehat{-}$ from Definition~\ref{defAhut}
can be extended to a functor $\Bimodf \to \Rexf$ by sending a 
bimodule $M\colon \cat{A}\otimes \cat{B}^\opp \to \vect$ to the functor 
\begin{align}\label{Eq: Def on bimod}
	\widehat{M} \colon \widehat{\cat{A}}=\Cat_k(\cat{A}^{\opp}, \vect) & \longrightarrow \widehat{\cat{B}}=\Cat_k(\cat{B}^{\opp}, \vect) \\
(	H\colon \cat{A}^{\opp}\to \vect )& \longmapsto \widehat{M}(H)(-)\coloneqq \int^{a\in \cat{A}} M(a,-)\otimes H(a) \ \ .
\end{align}
To show that this functor is well-defined, we need to verify that 
$\int^{a\in \cat{A}} M(a,-)\otimes H(a)$ is a finite-dimensional right $\cat{B}$-module.  Indeed, when composing with the Morita equivalences $B A_{\cat{A}} \simeq \cat{A}$ and $  B A_{\cat{B}}\simeq \cat{B}$ from Lemma~\ref{Lem: finite cat}, the functor 
$H \mapsto \int^{a\in \cat{A}} M(a,-)\otimes H(a)$
corresponds to tensoring with a finite-dimensional $A_{\cat{A}}$-$A_{\cat{B}}$-bimodule. Therefore, it is evident that this preserves  finite-dimensional modules.       

\begin{proposition}\label{Prop: Equivalence Bimod Rexf}
	The pair of functors
	\begin{equation}
		\begin{tikzcd}
			\Bimod^\catf{f}  \ar[rr, "\widehat{-}", bend left=10] &  \simeq  & \ar[ll, "\Proj(-)", bend left=10] \Rexf 
		\end{tikzcd} 
	\end{equation}
	is a pair of inverse symmetric monoidal equivalences of symmetric monoidal bicategories.
\end{proposition}  
	
\begin{proof}
	It is standard that $\widehat{-}$ sends the na\"ive monoidal product of linear categories to the Deligne product, 
	see e.g.~\cite[Example~11 \& Remark~12]{Franco2013}. Therefore, $\widehat{-}$ is a symmetric monoidal functor between symmetric monoidal bicategories.
	
The composition $\widehat{-} \circ \Proj(-)$ is naturally equivalent to the identity via the restricted Yoneda embedding; more precisely, the component of the natural equivalence
	at
$\cat{C}\in\Rexf$ is given by
\begin{align}
\cat{C} \ra{\text{Yoneda embedding}} \widehat{\cat{C}} \ra{\text{restriction}} \widehat{\Proj\cat{C}} \ . \label{eqncompositionyonres}
\end{align}
Note that for this natural transformation between functors between bicategories, the naturality squares do not commute strictly; there is coherence data involved. Since it is given by the `obvious' isomorphisms, we suppress this data here.

We need to prove that~\eqref{eqncompositionyonres} is an equivalence. First we observe that it is fully faithful:
For $X,Y \in \cat{C}$,
 the space of natural transformations from $\cat{C}(-,X)$ to $\cat{C}(-,Y)$, seen as functors $(\Proj\cat{C})^\opp \to \vect$
 can be calculated by
\begin{align}
	\Hom(  \cat{C}(-,X) , \cat{C}(-,Y)    )=\int_{P \in \Proj \cat{C}} \cat{C}(P,X)^* \otimes \cat{C}(P,Y)\cong \cat{C} \left(   \int^{P\in\Proj\cat{C}} \cat{C}(P,X) \otimes P , Y      \right) \ . 
	\end{align}     With the     Agreement Principle~\cite{mcarthy,keller}, see also the proof of Proposition~\ref{propcentralmonad}, we conclude that the map
$\int^{P\in\Proj\cat{C}} \cat{C}(P,X) \otimes P\to X$ is an isomorphism, which leaves us with 
\begin{align}\Hom(  \cat{C}(-,X) , \cat{C}(-,Y)    )\cong \cat{C}(X,Y)\ , \end{align} thereby proving full faithfulness.
 Hence, in order 
  to conclude that~\eqref{eqncompositionyonres} is an equivalence,
	 we only need to show that it is also essentially surjective. 
An arbitrary object $X$ in $\widehat{\Proj(\cat{C})}$ is a finite colimit of representable presheaves, see Remark~\ref{Rem: free finite}. This colimit can be represented by an object in the original category $\cat{C}$ since the colimit can be computed pointwise and since
$\colim_{i} \mathcal{C}(P,C_i) =  \mathcal{C}(P,\colim_{i} C_i) $ for all finite colimits and a projective object $P$.

Now consider the other composition $\Proj(-)  \circ \widehat{-} $. Let $\cat{A}\in \Bimodf$. 
Evaluation defines for us a $(\cat{A},\Proj (\widehat{\cat{A}}))$-bimodule
\begin{align}
	\Proj (\widehat{\cat{A}})\otimes \cat{A}^\opp & \to \vect \\ 
	F\otimes a & \longmapsto F(a)
\end{align} 
 and hence a morphism $\ev \colon \Proj(\widehat{\cat{A}}) \to \cat{A} $ in ${\Bimodf}$. We also have a morphism $\iota:  \cat{A} \to \Proj(\widehat{\cat{A}}) $ in $\Bimod$ given by the bimodule
 \begin{align}
 \iota \colon \cat{A}\otimes \Proj(\widehat{\cat{A}})^\opp  &\to \vect  \\ 
 a\otimes G &\mapsto \widehat{\cat{A}}(G,\cat{A}(-,a))
 \end{align}
 We compute their compositions: For $F,G\in \Proj (\widehat{\cat{A}})$,
\begin{align}
\iota \circ \ev (G,F) = \int^{a\in \cat{A} } \widehat{\cat{A}}(G,\cat{A}(-,a)) \otimes F(a) \cong  \widehat{\cat{A}}\left(G, \int^{a\in \cat{A} } \cat{A}(-,a)\otimes  F(a)\right) \cong \widehat{\cat{A}}(G,F)\ , 
\end{align}
where we use for the first isomorphism
 that $G$ is projective (thereby making
$\widehat{\cat{A}}(G,-)$ exact), and for the last step  the 
Yoneda lemma for functors $\cat{A}^\opp \to \vect$. For the other composition, we find with $a,a'\in \cat{A}$
\begin{align}
	  \ev \circ \iota (a,a') &= \int^{G \in \Proj{\widehat{\cat{A}}} } \widehat{\cat{A}}(G,\cat{A}(-,a)) \otimes G(a') \cong \int^{b \in \cat{A} } \widehat{\cat{A}}(\cat{A}(-,b),\cat{A}(-,a)) \otimes \cat{A}(a',b) \\
	  &\cong \int^{b \in \cat{A} } \cat{A}(b,a) \otimes \cat{A}(a' ,b) \cong  \cat{A}(a' ,a) \ , 
\end{align} 
where we used for the second step
 that it is enough to let the coend run over representable functors, again by the Agreement Principle~\cite{mcarthy,keller}. 
This finishes the proof that $\widehat{-}$ and $\Proj(-)$ are inverse equivalences.
\end{proof}	

For a pivotal finite tensor category $\cat{C}$, we want to 
post-compose the open-closed modular functor $\snc$
from Theorem~\ref{thmsnmf}
with the symmetric monoidal equivalence from Proposition~\ref{Prop: Equivalence Bimod Rexf}, of course after verifying that 
 the $\Bimod$-valued open-closed modular functor $\snc$
factors through $\Bimodf$. 
We then obtain	 the finitely freely cocompleted $\Rexf$-valued
string-net open-closed modular functor $\SNC$ whose value on the open interval and the circle we want to describe.
This will need some preparations:

\begin{lemma}\label{lemmacoeq}
	Let $\cat{C}$ be a pivotal finite tensor category and $X\in Z(\cat{C})$.
	Then
	\begin{equation}\label{eqncoequl}
		\begin{tikzcd}
			\displaystyle FUFU X \ar[r, shift left=2] \ar[r, shift right=2]
			& \displaystyle FUX \ar[r] & X\ .   
		\end{tikzcd}
	\end{equation}
	(with the maps coming from the monad structure on $Z=UF$, and the $Z$-algebra structure of $X$) is a coequalizer in $Z(\cat{C})$.  The coequalizer is split in $Z(\cat{C})$ if $X$ is projective.
\end{lemma}

\begin{proof}
	The Drinfeld center $Z(\cat{C})$ is the Eilenberg-Moore category of algebras over the central monad $Z=UF$ \cite{daystreet}. Now the statement that~\eqref{eqncoequl} is a coequalizer is just a standard Lemma about monads and their algebras, see e.g.\ \cite[Lemma~4.3.3]{borceux} or (the proof of) \cite[Section~VI.7., Theorem~1]{maclane}. 
	
	Let us now assume that $X$ is projective. 
	With $\partial$ being the difference of the two morphisms in~\eqref{eqncoequl}, we get an exact sequence
	\begin{align}
	FUFUX \ra{\partial} FUX \to X \to 0 \ . 
	\end{align}
	Since $X$ is projective, the exact sequence and therefore the coequalizer splits in $Z(\cat{C})$. 
\end{proof}

\begin{lemma}\label{lemmacompletion}
	For any pivotal finite tensor category $\cat{C}$,
	a linear functor 
	$F(\Proj\cat{C}) \to \vect$
	defined on the subcategory $F(\Proj\cat{C})\subset \Proj Z(\cat{C})$ 
	uniquely extends to $\Proj Z(\cat{C})$. In particular, the finite cocompletion of $F(\Proj\cat{C})$ is equivalent to $Z(\cat{C})$. 
\end{lemma}

\begin{proof}
	Lemma~\ref{lemmacoeq} that tells us in particular that any projective object in $Z(\cat{C})$ is a split coequalizer of objects in $F(\Proj\cat{C})$. Hence, any projective object in $Z(\cat{C})$ is an absolute finite colimit of objects in $F(\Proj\cat{C})$, i.e.\ a finite colimit preserved \emph{by all functors}. This implies the existence of a unique extension
	from $F(\Proj\cat{C})$ to
	$\Proj Z(\cat{C})$. 
	Therefore, $F(\Proj\cat{C})$ and $\Proj Z(\cat{C})$ have the same finite cocompletion.
	 Finally, we use that the finite cocompletion 
	  of $\Proj Z(\cat{C})$ is $Z(\cat{C})$ by Proposition~\ref{Prop: Equivalence Bimod Rexf}.
\end{proof}

\begin{theorem}\label{thmlexmf}
	Let $\cat{C}$ be a pivotal finite tensor category.
	After free finite cocompletion,
	the open-closed modular functor $\snc$
	from Theorem~\ref{thmsnmf} yields an open-closed modular functor in $\Rexf$
	that we denote by $\SNC$.
	The $\Rexf$-valued modular functor $\SNC$ 
	associates $\cat{C}$ to $[0,1]$ and $Z(\cat{C})$ to $\mathbb{S}^1$ in the sense that the 
	equivalence $h:\snc(\mathbb{S}^1) \ra{\simeq} F(\Proj\cat{C})$
	from Theorem~\ref{thmfunctorH} induces after finite 
	cocompletion an equivalence
	\begin{align}\label{eqntheequivH}
	H:\SNC(\mathbb{S}^1)\ra{\simeq}
	Z(\cat{C}) \ . 
	\end{align}
\end{theorem}

\begin{proof}
	The $\Bimod$-valued 
	open-closed modular functor $\snc$ from Theorem~\ref{thmsnmf}
	factors through 
	$\Bimod^\catf{f}$. This can be seen as follows:
	\begin{itemize}
		\item 
		
		The categories associated by $\snc$ to one-dimensional manifolds 
		are finite after finite cocompletion. Indeed, 
		for the category $\snc([0,1])\simeq \Proj \cat{C}$ associated to the interval (Remark~\ref{remsncint}), this is trivial. 
		On $\mathbb{S}^1$,
		we have the equivalence
		 $h:\snc(\mathbb{S}^1)\ra{\simeq} F(\Proj\cat{C})$ 
		from Theorem~\ref{thmfunctorH}. After finite free cocompletion, this yields
		an equivalence
		$H:\SNC(\mathbb{S}^1)\ra{\simeq}
		Z(\cat{C})$
		 by Lemma~\ref{lemmacompletion}.
		The category $Z(\cat{C})$ is 
		finite~\cite[Theorem~3.34]{etingofostrik}.
		
		\item The bimodules associated to surfaces take values in finite-dimensional vector spaces. 
		This is a consequence of Corollary~\ref{corsnfinite}.
	\end{itemize}
	Now we apply the equivalence 
	from Proposition~\ref{Prop: Equivalence Bimod Rexf} to get a $\Rexf$-valued open-closed modular functor.
\end{proof}

	\section{The Swiss-Cheese algebra underlying the string-net construction\label{secswisscheese}}
	This section is devoted the following improvement of Theorem~\ref{thmlexmf}:

	\begin{theorem}\label{thmHbraided}
		For any pivotal finite tensor category $\cat{C}$, the equivalence
			\begin{align} H : \SNC(\mathbb{S}^1)\ra{\simeq}
			Z(\cat{C}) \  
		\end{align}
		from Theorem~\ref{thmfunctorH} comes naturally with a braided monoidal structure. 
		Here the braided monoidal structure 
		on $\SNC(\mathbb{S}^1)$ is the one that it has by virtue of being the value of a modular functor on the circle
		while the braided monoidal structure on $Z(\cat{C})$ is the standard one.
		\end{theorem}
	
	A direct proof of this improvement seems relatively hard, and we will choose an indirect approach based on Idrissi's characterization of categorical Swiss-Cheese algebras in~\cite{najib}. 
	Swiss-Cheese algebras will not be needed elsewhere in the article, and the reader willing to accept Theorem~\ref{thmHbraided} can skip the rest of this section.
	
	We do not want to go too far into the operadic details (and luckily, we do not have to). For us, it is enough to briefly discuss the Swiss-Cheese operad
	 that was defined by Voronov~\cite{voronov}. To describe this operad, it is convenient to first recall that the framed Swiss-Cheese operad~\cite[Section~3.4.1]{KM} is the operad 
	 formed by genus zero open-closed surfaces (no cyclic structure for the operad is included for the moment). 
	It has two colors, one for the boundary interval, and one for the boundary circle. Voronov's Swiss-Cheese operad is then the suboperad with the same colors, but the Dehn twist on the cylinder excluded from its operations.  
	So once we look at a Swiss-Cheese algebra in categories (or certain types of linear categories), we obtain a category $\cat{A}$ for the boundary interval and a category $\cat{B}$ for the boundary circle. Idrissi~\cite[Theorem~A]{najib}
	 proved that the operations on $\cat{A}$ and $\cat{B}$ are exactly the following ones:
	
	\begin{itemize}
		
		\item The category $\cat{A}$ is monoidal (this is  the evaluation on the `open part').
		
		\item The category $\cat{B}$ is braided (this the evaluation on the `closed part').
		
		\item The two categories come with a  braided monoidal functor $F:         \cat{B}\to Z(\cat{A})$. 
		
		\end{itemize}
	The relation to our present situation is the following: We know 
	from Theorem~\ref{thmlexmf}
	that string-nets for a pivotal finite tensor category
	$\cat{C}$
	 produce an open-closed 
	modular functor. 
	Clearly, we can restrict an open-closed modular functor to open-closed genus zero surfaces.
	We then obtain the underlying Swiss-Cheese algebra
	\begin{align} (\SNC([0,1])\simeq \cat{C} , \  \SNC(\mathbb{S}^1) ,\  B: \SNC(\mathbb{S}^1)\to Z(\cat{C}) ) 
	\end{align}
	in $\Rexf$. 
	To the boundary interval, it associates $\cat{C}$; in fact, not only as linear category, but as monoidal category because
	the multiplication comes just from a disk with three marked boundary intervals. One can see with the help of Lemma~\ref{lemmasndisks} that the topologically inherited monoidal structure on $\cat{C}$ is then really the original one.
	Extracting the braided category associated to the closed boundary
	 is way harder: We know that we associate $\SNC(\mathbb{S}^1)$ by definition, but with some a priori unknown braided monoidal structure coming from the evaluation on genus zero surfaces. 
	However, thanks to~\cite[Theorem~A]{najib}, we know that we get a braided monoidal functor $B: \SNC(\mathbb{S}^1)\to Z(\cat{C})$.

	In order to get the braided monoidal structure on the functor $H:\SNC(\mathbb{S}^1)\to Z(\cat{C})$
	in Theorem~\ref{thmHbraided},
	it suffices
	to prove $H\cong B$ as functors:

	\begin{proposition}
	For a pivotal finite tensor category $\cat{C}$, denote by
	\begin{align}
		(\cat{C} , \SNC(\mathbb{S}^1) ,  B: \SNC(\mathbb{S}^1)\to Z(\cat{C}) )
	\end{align}
	the Swiss-Cheese algebra underlying the open-closed modular functor given by the string-net construction applied to $\Proj\cat{C}$.
	Then the underlying functor of $B$ is naturally isomorphic to the functor $H : \SNC(\mathbb{S}^1)\to Z(\cat{C})$ 
	from Theorem~\ref{thmfunctorH}.
	\end{proposition}
	
	\begin{proof}
		It follows from the main result of \cite{najib} that $UB$ is obtained 
		by evaluation of $\SNC$ on the annulus with one incoming boundary circle and one outgoing boundary interval. From \eqref{eqnkleislieqn}, we can deduce that $UB$ sends a projective object placed on the circle to the object 
		$UFP\cong \int^{X\in\cat{C}} X^\vee \otimes P \otimes X$ in $\cat{C}$. 
		The lift $B$ of $UB$ to $Z(\cat{C})$ is obtained by extracting from the Swiss-Cheese algebra a half braiding for the object $UBP \cong \int^{X\in\cat{C}} X^\vee \otimes P \otimes X$. The way that this half braiding arises topologically is described in \cite{najib}. When specified to the case at hand, we obtain it as follows:
		Consider an annulus $A=\mathbb{S}^1 \times [0,1]$ with \begin{itemize}
			\item an incoming boundary interval on $\mathbb{S}^1 \times \{0\}$
		 labeled by $X\in \Proj\cat{C}$,
		 \item an incoming boundary interval on $\mathbb{S}^1 \times \{1\}$
		 labeled by $Y \in \Proj\cat{C}$,
		 \item and an outgoing boundary interval on $\mathbb{S}^1 \times \{1\}$ labeled by $V \in \Proj\cat{C}$.
		 \end{itemize}

\begin{figure}[h]
\[
\begin{tikzpicture}[baseline=-0.25em,scale=0.5]
	\begin{pgfonlayer}{nodelayer}
		\node [style=none] (36) at (0, -1) {};
		\node [style=none] (38) at (0, -4) {};
		\node [style=none] (28) at (0, -1) {};
		\node [style=none] (22) at (0, -1) {};
		\node [style=none] (14) at (0, 1) {};
		\node [style=none] (17) at (0, 4) {};
		\node [style=bdot] (41) at (-2.25, 3.325) {};
		\node [style=bdot] (42) at (2.25, 3.325) {};
		\node [style=bdot] (43) at (0, 1) {};
		\node [style=none] (44) at (0, 0.25) {$X$};
		\node [style=none] (45) at (-2.25, 4) {$Y$};
		\node [style=none] (46) at (2.25, 4) {$V^\vee$};
		\node [style=nc] (49) at (0, 2.5) {};
		\node [style=none] (50) at (0, 2.5) {$\varphi$};
		\node [style=none] (24) at (0, -4) {};
		\node [style=none] (0) at (0, 1) {};
		\node [style=none] (1) at (-1, 0) {};
		\node [style=none] (2) at (1, 0) {};
		\node [style=none] (3) at (0, -1) {};
		\node [style=none] (8) at (0, 4) {};
		\node [style=none] (39) at (-4, 0) {};
		\node [style=none] (40) at (4, 0) {};
	\end{pgfonlayer}
	\begin{pgfonlayer}{edgelayer}
		\draw [style=vt-sh] (2.center) to (40.center);
		\draw [style=vt-bl-di, bend left=15] (49) to (41);
		\draw [style=vt-bl-di, bend right=15] (49) to (42);
		\draw [style=vt-bl-di] (49) to (43);
		\draw [style=vt] (1.center)
		to [bend right=45] (3.center)
		to [bend right=45] (2.center)
		to [bend right=45] (0.center)
		to [bend right=45] cycle;
		\draw [style=vt] (24.center)
		to [bend right=45] (40.center)
		to [bend right=45] (8.center)
		to [bend right=45] (39.center)
		to [bend right=45] cycle;
	\end{pgfonlayer}
\end{tikzpicture}
\quad\mapsto\quad\begin{tikzpicture}[baseline=-0.25em,scale=0.5]
	\begin{pgfonlayer}{nodelayer}
		\node [style=none] (36) at (0, -1) {};
		\node [style=none] (38) at (0, -4) {};
		\node [style=none] (28) at (0, -1) {};
		\node [style=none] (22) at (0, -1) {};
		\node [style=none] (14) at (0, 1) {};
		\node [style=none] (17) at (0, 4) {};
		\node [style=bdot] (41) at (-2.25, 3.325) {};
		\node [style=bdot] (42) at (2.25, 3.325) {};
		\node [style=bdot] (43) at (0, 1) {};
		\node [style=none] (44) at (0, 0.25) {$X$};
		\node [style=none] (45) at (2.25, 4) {$Y$};
		\node [style=none] (46) at (-2.25, 4) {$V^\vee$};
		\node [style=nc] (49) at (-1.75, 1.75) {};
		\node [style=none] (50) at (0, -2.75) {};
		\node [style=none] (51) at (2.75, 0) {};
		\node [style=l70-none] (52) at (-1.75, 1.75) {$\varphi$};
		\node [style=none] (24) at (0, -4) {};
		\node [style=none] (0) at (0, 1) {};
		\node [style=none] (1) at (-1, 0) {};
		\node [style=none] (2) at (1, 0) {};
		\node [style=none] (3) at (0, -1) {};
		\node [style=none] (8) at (0, 4) {};
		\node [style=none] (39) at (-4, 0) {};
		\node [style=none] (40) at (4, 0) {};
	\end{pgfonlayer}
	\begin{pgfonlayer}{edgelayer}
		\draw [style=vt-sh] (2.center) to (40.center);
		\draw [style=vt-bl-di, bend right=15] (49) to (41);
		\draw [style=vt-bl-di, bend left] (49) to (43);
		\draw [style=vt-bl, in=180, out=-135, looseness=1.25] (49) to (50.center);
		\draw [style=vt-bl, in=-90, out=0] (50.center) to (51.center);
		\draw [style=vt-bl-di, in=-105, out=90] (51.center) to (42);
		\draw [style=vt] (1.center)
		to [bend right=45] (3.center)
		to [bend right=45] (2.center)
		to [bend right=45] (0.center)
		to [bend right=45] cycle;
		\draw [style=vt] (24.center)
		to [bend right=45] (40.center)
		to [bend right=45] (8.center)
		to [bend right=45] (39.center)
		to [bend right=45] cycle;
	\end{pgfonlayer}
\end{tikzpicture}
\]
\caption{Here we suppress the intervals and label the points living on the outgoing boundary intervals by the dual of $V\in\Proj(\cat{C})$. The cut
	line is kept still while the diffeomorphism swaps the relative positions of marked points living on $\mathbb{S}^1\times \{1\}$. To facilitate the visualization, a string-net is included on both sides.}
	\label{fig: half braiding}
\end{figure}
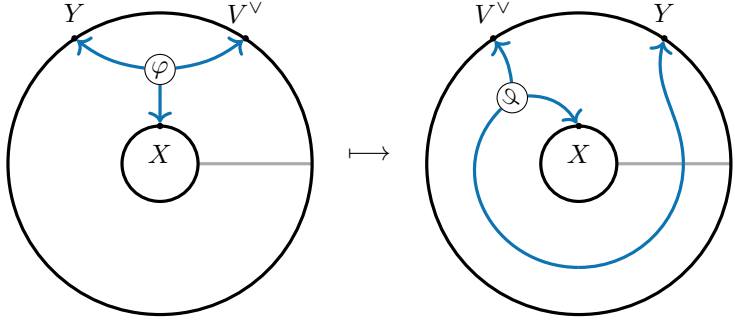
		 The string-net construction $\snc$ assigns to the annulus
		 $A$ with these labels the vector space
		 \begin{align} \int^{P \in \Proj\cat{C}}\cat{C}(I,  P^\vee \otimes X \otimes P \otimes Y \otimes V^\vee)&\cong  \int^{P \in \Proj\cat{C}}\cat{C}(V, P^\vee \otimes X \otimes P \otimes Y )
		 	\\&\cong \cat{C}\left( V, \int^{P \in \Proj \cat{C}}    P^\vee \otimes X \otimes P \otimes Y     \right)
		 	 \\&\cong \cat{C}\left( V, \int^{P \in \cat{C}}    P^\vee \otimes X \otimes P \otimes Y     \right)  \ . \label{eqnidentifications}
		 	\end{align}
		 	This follows from excision (Theorem~\ref{thmexcision}); the simplifications made here use again the arguments in the proof of Theorem~\ref{propcentralmonad}. 
		 	Through the diffeomorphism moving around $Y$ counterclockwise
		 	on $\mathbb{S}^1\times \{1\}$ past $X$, we obtain an isomorphism
		 	\begin{align}
		 	\tau :	\cat{C}\left( V, \int^{P \in \cat{C}}    P^\vee \otimes X \otimes P \otimes Y     \right) \ra{\cong} \cat{C}\left( V, Y \otimes \int^{Q \in \cat{C}}    Q^\vee \otimes X \otimes Q     \right) \ . \label{eqnmoveiso}
		 		\end{align}
		 		The resulting isomorphism \begin{align}\label{eqnhalfbraiding} \beta : \int^{P \in \cat{C}}    P^\vee \otimes X \otimes P \otimes Y  \ra{\cong} Y \otimes \int^{Q \in \cat{C}}    Q^\vee \otimes X \otimes Q \end{align} is the sought after half braiding on $\int^{P \in \cat{C}}    P^\vee \otimes X \otimes P$. We refer to Figure~\ref{fig: half braiding} for a sketch of the action of this isomorphism on string-nets.
		 		It remains to read off \eqref{eqnmoveiso} from the string-net construction and to translate it into an isomorphism of the form~\eqref{eqnhalfbraiding}: First note that the coend in~\eqref{eqnmoveiso} can be pulled out of the hom argument (see the calculations in~\eqref{eqnidentifications}). Therefore, we can see $\cat{C}\left( V,     P^\vee \otimes X \otimes P \otimes Y     \right)$ as the integrand of the coend on the left hand side of~\eqref{eqnmoveiso}, with the dummy variable being $P$. From the definition of the string-net construction and  excision (Theorem~\ref{thmexcision}), it follows that $\tau$ is induced by the map sending a morphism $\varphi : V \to P^\vee \otimes X \otimes P \otimes Y  $ to $b_Y \otimes \varphi : V \to Y \otimes Y^\vee \otimes P^\vee \otimes X \otimes P \otimes Y$; here $b_Y:I\to Y\otimes Y^\vee$ is the coevaluation of $Y$, and  the new dummy variable is $Q=P\otimes Y$. When translated to an isomorphism of the form~\eqref{eqnmoveiso}, this yields the non-crossing half braiding of  $\int^{P \in \cat{C}}    P^\vee \otimes X \otimes P$, see~Figure~\ref{fignoncrossing}. 
		 		Since the functor $F:\cat{C}\to Z(\cat{C})$ sends $X\in\cat{C}$ to $\int^{P \in \cat{C}}    P^\vee \otimes X \otimes P$ plus the non-crossing half braiding, it follows from this computation and the description of the functor $h$ in Theorem~\ref{thmfunctorH} that the composition
		 		\begin{align}
		 			\snc(\mathbb{S}^1) \subset \SNC(\mathbb{S}^1) \ra{B} Z(\cat{C})
		 			\end{align} agrees with
		 			\begin{align}
		 			\label{eqnsmallhext}	\snc(\mathbb{S}^1) \ra{h} F(\Proj\cat{C}) \subset Z(\cat{C}) \ . 
		 				\end{align}
		 			By the definition of the finite cocompletion there is, up to natural isomorphism, exactly one right exact functor $\SNC(\mathbb{S}^1)\to Z(\cat{C})$ extending~\eqref{eqnsmallhext}, namely $H$  (by its construction in Theorem~\ref{thmfunctorH}). This implies $B\cong H$. 
		\end{proof}

	\begin{remark}
		The strategy of the proof of 
		Theorem~\ref{thmHbraided} is informed by the examination of the algebraic structure underlying
		 bulk-boundary systems 
	in~\cite{fsv}. This does not actually simplify the proof, but it provides some context that might be helpful for the physically inclined reader.
		In~\cite[Section~3]{fsv} several algebraic properties of bulk-boundary conditions are postulated based on physical considerations, without explicit reference to the Swiss-Cheese operad.
		Nonetheless, the result is closely related to the structure of a categorical Swiss-Cheese algebra $(\cat{A},\cat{B}, F:\cat{B}\to Z(\cat{A}))$
		 from~\cite[Theorem~A]{najib}. The category $\cat{A}$ would then be called the \emph{category of boundary Wilson lines};
		 $\cat{B}$ is the \emph{category of bulk Wilson lines}. There is however one crucial difference: In \cite[Section~3]{fsv}, a class of boundary conditions is singled out by imposing
		  the requirement that the functor $F:\cat{B}\to Z(\cat{A})$ is not only braided monoidal, but also an equivalence, thereby providing a so-called \emph{Witt trivialization of $\cat{B}$}. 
		 It is important to note that this is not a property enforced by the Swiss-Cheese operad. 
		  The results of this section can be rephrased by saying that  $(\cat{C},\SNC(\mathbb{S}^1))$ is a mathematical incarnation
		 of a bulk-boundary system that actually has this additional property
		  postulated in~\cite{fsv}. 
		\end{remark}

	\spaceplease
	\section{The comparison theorem}
	
	In this section, we finally prove the main result of this article, namely the equivalence, in the sense of  \cite[Section~3.2]{brochierwoike}, between the string-net modular functor for a pivotal finite tensor category $\cat{C}$ and the Lyubashenko modular functor for the Drinfeld center $Z(\cat{C})$:

		\begin{theorem}\label{thmmain}
		For any pivotal finite tensor category $\cat{C}$, the string-net modular functor $\SNC$ associated 
		to $\cat{C}$ is equivalent to the Lyubashenko modular functor associated to the Drinfeld center $Z(\cat{C})$. 
	\end{theorem}

As recalled in the introduction,
the Lyubashenko construction~\cite{lyubacmp,lyu,lyulex} is a modular functor construction that takes as an input a modular category, i.e.\ a finite ribbon category with non-degenerate braiding.
If the pivotal finite tensor category $\cat{C}$ is spherical in the sense of~\cite{dsps}, then $Z(\cat{C})$ is modular by \cite{shimizuribbon}, so that the Lyubashenko construction can be applied to $Z(\cat{C})$. 
If $\cat{C}$ is not spherical, one can still construct a modular functor from $Z(\cat{C})$ by \cite[Section~8.4]{brochierwoike}, even if $Z(\cat{C})$ is not modular.
We understand the Lyubashenko modular functor in Theorem~\ref{thmmain} in this generalized sense.
	
	A direct comparison at all genera  of the two types of modular functors 
	appearing in Theorem~\ref{thmmain}
	seems difficult, and we will therefore follow an indirect approach that uses the results of \cite{cyclic,brochierwoike}:
	By \cite[Theorem~7.17]{cyclic} the value of a modular functor on the circle
	 inherits through the evaluation of the modular functor on genus zero surfaces a ribbon Grothendieck-Verdier structure, a notion defined by Boyarchenko-Drinfeld~\cite{bd}. 
	A \emph{ribbon Grothendieck-Verdier category} in $\Rexf$ is \begin{itemize}
		\item a finite category $\cat{A}$ with right exact monoidal product $\otimes : \cat{A}\boxtimes \cat{A}\to \cat{A}$, braiding $c_{X,Y} : X \otimes Y \ra{\cong} Y \otimes X$ and balancing $\theta_X : X \ra{\cong}X$, i.e.\ a natural automorphism of the identity with $\theta_I=\id_I$ and $\theta_{X \otimes Y}=c_{Y,X}c_{X,Y}(\theta_X \otimes \theta_Y)$,
		
		\item an object $K$, called the \emph{dualizing object}, such that $\cat{A}(X\otimes - ,K)$ is representable for all $X\in \cat{A}$, i.e.\ $\cat{A}(X\otimes - ,K)\cong \cat{A}(-,DX)$ for some $DX\in \cat{A}$, such that $D:\cat{A}\to\cat{A}^\opp, X \mapsto DX$ is an equivalence and $D\theta_X = \theta_{DX}$.
		\end{itemize}
		The dualizing object $K$ can be recovered from $D$ via $K=DI$.

		By Theorem~\ref{thmlexmf} $\SNC$
		is a modular functor. By what was just explained, the evaluation on the circle (and genus zero surfaces)
		produces a ribbon Grothendieck-Verdier category.
		This ribbon Grothendieck-Verdier structure
		 will be explicitly calculated through our next result.

		 In order to state the result,
		 	recall from~\eqref{eqnradford}  the distinguished invertible object $\alpha \in \cat{C}$ of~\cite{eno-d}
		 	and the Radford isomorphism
		 $-^{\vee \vee\vee\vee}\cong \alpha \otimes-\otimes \alpha^{-1}$.
		 If $\cat{C}$ is pivotal, the square of the pivotal structure endows $\alpha$ with a half braiding, thereby promoting it to an object in the Drinfeld center $Z(\cat{C})$.
		 By~\cite{mwcenter}
		  this object then becomes the dualizing object for a ribbon Grothendieck-Verdier structure on $Z(\cat{C})$ that agrees with the unit of $Z(\cat{C})$ if and only if $\cat{C}$ is spherical.\label{mentionmwcenter}

		 \begin{theorem}\label{thmcycfrE2}
		 	For any pivotal finite tensor category, the functor $H: \SNC(\mathbb{S}^1)\to Z(\cat{C})$ 
		 	is an equivalence of ribbon Grothendieck-Verdier categories, where
		 	\begin{itemize}
		 		\item $\SNC(\mathbb{S}^1)$ carries the ribbon Grothendieck-Verdier structure that it inherits by virtue of being the value of a modular functor on the circle,

		 		\item and $Z(\cat{C})$ carries the ribbon Grothendieck-Verdier structure $D=-^\vee \otimes \alpha^{-1}$
		 		with the distinguished invertible object $\alpha$
		 		of $\cat{C}$ (equipped with a half braiding through the Radford isomorphism and the pivotal structure) as dualizing object in $Z(\cat{C})$. 
		 		
		 	\end{itemize} 
		 \end{theorem}

		 \begin{proof}
		 	By Theorem~\ref{thmHbraided} the functor $H$ is already a braided monoidal equivalence.
		 	
		 	Next we prove that $H$ also preserves the balancing. 
		 	To this end, we place a projective object $P\in\Proj\cat{C}$ on the circle and denote the corresponding object in $\SNC(\mathbb{S}^1)$ by $\widetilde P$. The component of the balancing in $\SNC(\mathbb{S}^1)$ at $\widetilde P$ is denoted by $\theta_{\widetilde P}^{\SNC}$. 
		 	By construction $H$ sends $\widetilde P$ to $FP \in Z(\cat{C})$. We denote the balancing on $Z(\cat{C})$ by $\theta$;
		 	this is the `standard' balancing that $Z(\cat{C})$ is endowed with because it is braided and pivotal.
		 	It now suffices to prove
		 	\begin{align}\label{eqnvergleichbal}
		 	H \theta_{\widetilde P}^{\SNC} = \theta_{F(P)}  
		 	\end{align}
		 	because this implies $H\theta^{\SNC}_X = \theta_{HX}$ for all $X\in \SNC(\mathbb{S}^1)$ (this is because any 
		 	 object in $\SNC(\mathbb{S}^1)$ and $Z(\cat{C})$ receives an epimorphism from an object of the form $\widetilde P$ and $F(P)$, respectively, see Lemma~\ref{lemmacoeq}).
		 	
		 	For the proof of~\eqref{eqnvergleichbal}, we compute both sides separately:
		 	 The balancing $\theta_{F(P)}:F(P)\to F(P)$
		 	  at $F(P)$, as a consequence of the adjunction  $F\dashv U$,
		 	  is completely determined by the map \begin{align}
		 	  	P\ra{\iota_I} UF(P)\xrightarrow{U\theta_{F(P)}} UF(P) \end{align} in $\cat{C}$
	 	  	(here $P\cong I^\vee \otimes P \otimes I \ra{\iota_I} UF(P)$ is the structure map of the coend)
		 	  given by:
		 	\begin{align}\label{twisteqn1}
\begin{tikzpicture}[baseline=-0.25em,scale=0.5]
	\begin{pgfonlayer}{nodelayer}
		\node [style=none] (0) at (0, -4) {};
		\node [style=none] (1) at (0, -1) {};
		\node [style=none] (5) at (0, -0.5) {$\imath_I$};
		\node [style=none] (6) at (0, 4) {};
		\node [style=none] (7) at (-1, 1.25) {};
		\node [style=none] (8) at (0, 2.25) {};
		\node [style=none] (9) at (-0.75, 2) {};
		\node [style=none] (10) at (-0.25, 1.5) {};
		\node [style=none] (11) at (-1, 2.5) {};
		\node [style=none] (13) at (-2, 1.25) {};
		\node [style=none] (14) at (-2, 2.5) {};
		\node [style=none] (15) at (0, 4.5) {$UF(P)$};
		\node [style=none] (16) at (0, -4.5) {$P$};
		\node [style=none] (2) at (0, 0) {};
		\node [style=none] (3) at (-0.75, -1) {};
		\node [style=none] (4) at (0.75, -1) {};
	\end{pgfonlayer}
	\begin{pgfonlayer}{edgelayer}
		\draw [style=vt-bl-di] (0.center) to (1.center);
		\draw [style=vt-gr, in=-90, out=90] (7.center) to (8.center);
		\draw [style=vt-gr, in=135, out=-90] (11.center) to (9.center);
		\draw [style=vt-gr, in=-45, out=90, looseness=0.75] (2.center) to (10.center);
		\draw [style=vt-gr] (13.center) to (14.center);
		\draw [style=vt-gr, bend right=90, looseness=1.50] (13.center) to (7.center);
		\draw [style=vt-gr, bend left=90, looseness=1.50] (14.center) to (11.center);
		\draw [style=vt-gr-di] (8.center) to (6.center);
		\draw [style=vt] (2.center)
		to [in=90, out=-180] (3.center)
		to (4.center)
		to [in=0, out=90] cycle;
	\end{pgfonlayer}
\end{tikzpicture}\enspace=\quad\begin{tikzpicture}[baseline=-0.25em,scale=0.5]
\begin{pgfonlayer}{nodelayer}
	\node [style=none] (0) at (0, -4) {};
	\node [style=none] (6) at (0, 4) {};
	\node [style=none] (7) at (-1.25, -0.25) {};
	\node [style=none] (8) at (0, 1.25) {};
	\node [style=none] (9) at (-1, 0.75) {};
	\node [style=none] (10) at (-0.25, 0.25) {};
	\node [style=none] (11) at (-1.25, 1.25) {};
	\node [style=none] (13) at (-2.75, -1.25) {};
	\node [style=none] (14) at (-2.75, 1.25) {};
	\node [style=none] (15) at (0, 4.5) {$UF(P)$};
	\node [style=none] (16) at (0, -4.5) {$P$};
	\node [style=none] (2) at (0, -0.25) {};
	\node [style=none] (17) at (-1.25, -1.25) {};
	\node [style=none] (18) at (-1.25, -0.75) {$\imath_I$};
	\node [style=none] (19) at (-1.25, -0.25) {};
	\node [style=none] (20) at (-2, -1.25) {};
	\node [style=none] (21) at (-0.5, -1.25) {};
\end{pgfonlayer}
\begin{pgfonlayer}{edgelayer}
	\draw [style=vt-gr, in=-90, out=90] (7.center) to (8.center);
	\draw [style=vt-bl, in=135, out=-90] (11.center) to (9.center);
	\draw [style=vt-bl, in=-45, out=90, looseness=0.75] (2.center) to (10.center);
	\draw [style=vt-bl] (14.center) to (13.center);
	\draw [style=vt-gr-di] (8.center) to (6.center);
	\draw [style=vt-bl] (2.center) to (0.center);
	\draw [style=vt-bl-di, bend right=90, looseness=1.50] (13.center) to (17.center);
	\draw [style=vt-bl, bend left=90, looseness=1.50] (14.center) to (11.center);
	\draw [style=vt] (19.center)
	to [in=90, out=-180] (20.center)
	to (21.center)
	to [in=0, out=90] cycle;
\end{pgfonlayer}
\end{tikzpicture}
=\quad\begin{tikzpicture}[baseline=-0.25em,scale=0.5]
	\begin{pgfonlayer}{nodelayer}
		\node [style=none] (6) at (0, 4) {};
		\node [style=none] (15) at (0, 4.5) {$UF(P)$};
		\node [style=none] (16) at (0, -4.5) {$P$};
		\node [style=none] (22) at (0, 0.5) {};
		\node [style=none] (26) at (0, -4) {};
		\node [style=none] (28) at (-2.25, 0.5) {};
		\node [style=none] (29) at (-3.25, 0.5) {};
		\node [style=none] (30) at (-1.75, -1.25) {};
		\node [style=none] (23) at (-1.25, 0.5) {};
		\node [style=none] (24) at (1.25, 0.5) {};
		\node [style=none] (25) at (0, 2) {};
		\node [style=none] (31) at (-1.75, -1.75) {$P$};
		\node [style=none] (32) at (0, 1.25) {$\imath_P$};
	\end{pgfonlayer}
	\begin{pgfonlayer}{edgelayer}
		\draw [style=vt-gr-di] (25.center) to (6.center);
		\draw [style=vt-bl-di, in=-90, out=90] (26.center) to (24.center);
		\draw [style=vt-bl, bend left=90, looseness=2.00] (23.center) to (28.center);
		\draw [style=vt-bl, bend right=90, looseness=2.00] (28.center) to (29.center);
		\draw [style=vt-bl, in=180, out=-90] (29.center) to (30.center);
		\draw [style=vt-bl-di, in=-90, out=0] (30.center) to (22.center);
		\draw [style=vt] (24.center)
		to [in=0, out=90] (25.center)
		to [in=90, out=-180] (23.center)
		to cycle;
	\end{pgfonlayer}
\end{tikzpicture}
\quad=\quad\begin{tikzpicture}[baseline=-0.25em,scale=0.5]
	\begin{pgfonlayer}{nodelayer}
		\node [style=none] (6) at (0, 4) {};
		\node [style=none] (15) at (0, 4.5) {$UF(P)$};
		\node [style=none] (16) at (0, -4.5) {$P$};
		\node [style=none] (22) at (0, 0.5) {};
		\node [style=none] (26) at (0, -4) {};
		\node [style=none] (23) at (-1.25, 0.5) {};
		\node [style=none] (24) at (1.25, 0.5) {};
		\node [style=none] (25) at (0, 2) {};
		\node [style=none] (31) at (-0.625, -1.2) {$P$};
		\node [style=none] (32) at (0, 1.25) {$\imath_P$};
	\end{pgfonlayer}
	\begin{pgfonlayer}{edgelayer}
		\draw [style=vt-gr-di] (25.center) to (6.center);
		\draw [style=vt-bl-di, in=-90, out=90] (26.center) to (24.center);
		\draw [style=vt] (24.center)
		to [in=0, out=90] (25.center)
		to [in=90, out=-180] (23.center)
		to cycle;
		\draw [style=vt-bl-di, bend right=90, looseness=3.00] (23.center) to (22.center);
	\end{pgfonlayer}
\end{tikzpicture}
\end{align}
This uses the naturality of the half braiding and the definition of the non-crossing half braiding (that $F(P)$ comes equipped with by definition, see~Figure~\ref{fignoncrossing}).
		 On the string-net side, 
		 we can read off the balancing topologically: The automorphism $\theta_{\widetilde P}^{\SNC}$ is the image of the identity at $\widetilde P$ under the Dehn twist about the waist of the cylinder.
		 Graphically, 
		 the balancing $\theta_{\widetilde P}^{\SNC}$
		 is given by the automorphism 
		 
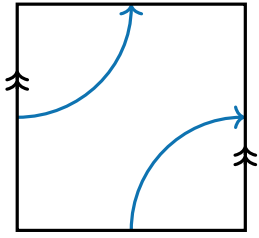
\begin{align}\label{twisteqn2}
\begin{tikzpicture}[baseline=-0.25em,scale=0.5]
	\begin{pgfonlayer}{nodelayer}
		\node [style=none] (4) at (-3, 1) {};
		\node [style=none] (5) at (-3, 0.75) {};
		\node [style=none] (6) at (3, -1.25) {};
		\node [style=none] (7) at (3, -1) {};
		\node [style=none] (8) at (3, -0.75) {};
		\node [style=none] (9) at (3, -1) {};
		\node [style=none] (10) at (-3, 1.25) {};
		\node [style=none] (11) at (-3, 1) {};
		\node [style=none] (12) at (-3, 0) {};
		\node [style=none] (13) at (3, 0) {};
		\node [style=none] (14) at (0, 3) {};
		\node [style=none] (15) at (0, -3) {};
		\node [style=none] (0) at (-3, 3) {};
		\node [style=none] (1) at (3, 3) {};
		\node [style=none] (2) at (-3, -3) {};
		\node [style=none] (3) at (3, -3) {};
	\end{pgfonlayer}
	\begin{pgfonlayer}{edgelayer}
		\draw [style=vt-di] (5.center) to (4.center);
		\draw [style=vt-di] (6.center) to (7.center);
		\draw [style=vt-di] (9.center) to (8.center);
		\draw [style=vt-di] (11.center) to (10.center);
		\draw [style=vt-bl-di, bend right=45] (12.center) to (14.center);
		\draw [style=vt-bl-di, bend left=45] (15.center) to (13.center);
		\draw [style=vt] (3.center)
		to (1.center)
		to (0.center)
		to (2.center)
		to cycle;
	\end{pgfonlayer}
\end{tikzpicture}
\end{align}
		 	in the category $\SNC(\mathbb{S}^1)$.
		 	Here the opposite vertical sides of the rectangle are identified.
		 	Under the isomorphism
		 	 from~\eqref{eqnexc3}, this automorphism
		 	in the category $\SNC(\mathbb{S}^1)$ agrees with the right hand side of~\eqref{twisteqn1}, i.e.\ with the `standard' balancing for the Drinfeld center computed before. 
		 	This establishes that $H:\SNC(\mathbb{S}^1) \to Z(\cat{C})$ is a balanced braided monoidal equivalence.
		 	
		 	It now follows that $H$ is also an equivalence of ribbon Grothendieck-Verdier categories because for the two categories in question the ribbon Grothendieck-Verdier structure is unique relative to the underlying 
		 	balanced braided structure as follows from \cite[Theorem~2.11]{mwcenter} because $Z(\cat{C})$, and therefore also $\SNC(\mathbb{S}^1)$, has a non-degenerate braiding in the sense that they Müger center is trivial~\cite[Proposition~4.4]{eno-d}.
		 \end{proof}

		 \begin{remark}
		 	The Grothendieck-Verdier duality mentioned in the last paragraph of the previous proof is unique, but actually we can say even more: The space of choices is $Bk^\times$, the classifying space of the group of units of the field $k$. This follows from~\cite[Corollary~4.5]{mwcenter}. 
		 	\end{remark}

		 \begin{proof}[\slshape Proof of Theorem~\ref{thmmain}]
		 	By \cite[Theorem~6.6]{brochierwoike}
		 	two $\Rexf$-valued modular functors are equivalent if and only if
		 	the 
		 	underlying ribbon Grothendieck-Verdier categories are equivalent in the sense of \cite[Section~2.4]{cyclic}.
		 	The ribbon Grothendieck-Verdier category extracted from $\SNC$ is described in Theorem~\ref{thmcycfrE2}: 
		 	It agrees 
		 	 with $Z(\cat{C})$
		 equipped	with the ribbon Grothendieck-Verdier structure from
		 	\cite{mwcenter}. 
		 	But this is by construction (see \cite[Section~8.4]{brochierwoike}) the ribbon Grothendieck-Verdier structure that the Lyubashenko construction for $Z(\cat{C})$ produces.
		 	This finishes the proof.
		 \end{proof}

	It is very instructive to compute the Grothendieck-Verdier duality on $Z(\cat{C})$ that corresponds to the topological Grothendieck-Verdier duality on $\SNC(\mathbb{S}^1)$ directly. For this, we need:
		
		\begin{lemma}\label{lemmadistinv}
			Let $\cat{C}$ be a pivotal
			finite tensor category. The left adjoint $F:\cat{C}\to Z(\cat{C})$ to the forgetful functor $U:Z(\cat{C})\to\cat{C}$ comes with a canonical isomorphism
			\begin{align}
			F(-^\vee) \cong (F(\alpha \otimes -))^\vee \ , 
			\end{align}
			where $\alpha$ is the distinguished invertible object of $\cat{C}$. 
		\end{lemma}

		\begin{proof}
		For the right adjoint $G:\cat{C}\to Z(\cat{C})$ of the forgetful functor $U:Z(\cat{C})\to \cat{C}$, the relation $F(-^\vee)\cong (G(-))^\vee$ holds~\cite[Lemma 3.5]{bruguieresvirelizier}. 
			Together with \cite[Lemma~4.7]{shimizuunimodular}, this implies
			$F(-^\vee) \cong (F(\alpha \otimes -))^\vee$.
		\end{proof}

	\begin{remark}[Direct comparison of the Grothendieck-Verdier dualities on
		$\SNC(\mathbb{S}^1)$ and $Z(\cat{C})$, independently	 of Theorem~\ref{thmcycfrE2}]\label{remduality}	
		We can directly prove using Lemma~\ref{lemmadistinv} that
		$H: \SNC(\mathbb{S}^1)\to Z(\cat{C})$ is compatible with the canonical Grothendieck-Verdier dualities that we have on both sides:
		The Grothendieck-Verdier duality on $\SNC(\mathbb{S}^1)$ is the topological one 
		 induced by orientation reversal which sends an object $P\in \Proj \cat{C}$ on the circle to $P^\vee$~\cite[Theorem~7.17]{cyclic}.  
		The corresponding duality functor $D$ on $Z(\cat{C})$ satisfies therefore $DFP=F(P^\vee)$. 
		We need to prove that this implies \begin{align}\label{eqnDalpha}
		DX\cong X^\vee \otimes \alpha^{-1}\quad \text{for all}\ X \in Z(\cat{C}) \ . 
		\end{align}
		Indeed, with Lemma~\ref{lemmadistinv} we find
		\begin{align}
		DFP=F(P^\vee)
		\cong (F(\alpha \otimes P))^\vee  \ . 
		\end{align}
		With the half braiding of $\alpha$ induced by the pivotal structure and the Radford isomorphism~\cite[Lemma~2.1]{mwcenter}, we can simplify further:
		\begin{align}
		DFP  \cong (\alpha \otimes FP)^\vee \cong  (FP)^\vee \otimes \alpha^{-1}\ . 
		\end{align}
		This proves~\eqref{eqnDalpha} on $F(\Proj \cat{C})$ and therefore on all objects because $F(\Proj \cat{C})$ generates $Z(\cat{C})$ under finite colimits (Lemma~\ref{lemmacompletion}) and $D$ is an equivalence. 
		\end{remark}

	\begin{remark}[Towards a derived generalization of the main result]\label{remdsn2}
	One could ask whether there is a relation between a derived version of the string-nets for $\cat{C}$, see Remark~\ref{remdsn1}, and the differential graded modular functor for $Z(\cat{C})$, as a special case of the construction in~\cite{dmf}.	This would be a derived generalization of Theorem~\ref{thmmain}. As plausible as such a generalization might seem, it would be somewhat involved because
	 basically none of the technical tools that we used in the linear setting, i.e.\ in this paper, are currently
	available in the derived setting.
		\end{remark}

	\spaceplease
	\section{Applications and examples	}
	If we spell out Theorem~\ref{thmmain}, we arrive at the following:

		\begin{corollary}\label{corconformalblocks}
	Let $\cat{C}$ be a pivotal finite tensor category and $\Sigma$ a surface with $n$ boundary components that are  labeled with $X_1 , \dots , X_n \in \cat{C}$.
	Then we can identify
	\begin{align}
	\SNC(\Sigma;X_1,\dots,X_n) \cong Z(\cat{C})  \left(  FX_1 \otimes \dots \otimes FX_n \otimes
	\left(\int_{X \in Z(\cat{C})} X \otimes X^\vee\right) ^{\otimes g} 
	, \alpha^{\otimes (g-1)}   \right) ^*    \ , 
	\end{align}
	where $F:\cat{C}\to Z(\cat{C})$ is the left adjoint to the forgetful functor from $Z(\cat{C})$ to $\cat{C}$, and $\alpha$ is the distinguished invertible object of $\cat{C}$, seen as object in the Drinfeld center via the pivotal structure and the Radford isomorphism.
	This isomorphism intertwines  the $\Map(\Sigma)$-action if 	$\SNC(\Sigma;X_1,\dots,X_n)$ is equipped with the geometric action and the right hand side with the (generalized) Lyubashenko action.
	\end{corollary}
	
	\begin{remark}
		The object $\alpha^{\otimes (g-1)}$ can be replaced by the monoidal unit $I \in Z(\cat{C})$ if $\cat{C}$ is spherical. In fact, 
		$\alpha \cong I$ in $Z(\cat{C})$ if and only if $\cat{C}$ is spherical (see the explanations on page~\pageref{mentionmwcenter}).
		For $g=1$, the contribution $\alpha^{\otimes (g-1)} \cong I$ disappears automatically. 
		This is not a surprise because the space of conformal blocks for the torus (as vector space, without the mapping class group action) never sees more than the linear structure of the category~\cite[Corollary~6.6]{mwansular}.
	\end{remark}

	\begin{proof}[\slshape Proof of Corollary~\ref{corconformalblocks}]
		By Theorem~\ref{thmmain} 
		$\SNC(\Sigma;X_1,\dots,X_n)$
		is isomorphic to the space of conformal blocks for $Z(\cat{C})$ evaluated at $\Sigma$ with boundary labels $FX_1 , \dots, FX_n$. 
		The latter space of conformal blocks is calculated for a general ribbon Grothendieck-Verdier category in \cite[Theorem~7.9]{cyclic} in the $\Lexf$-valued case; the dual $\Rexf$-valued version is given in \cite[Corollary~8.1]{brochierwoike}. So far, this tells us
		\begin{align}
			\SNC(\Sigma;X_1,\dots,X_n) \cong Z(\cat{C})  \left(  FX_1 \otimes \dots \otimes FX_n \otimes
			\left(\int_{X \in Z(\cat{C})} X \otimes DX \right) ^{\otimes g} 
			, \alpha^{-1}   \right) ^*    \ . 
		\end{align} With $DX=X^\vee \otimes \alpha^{-1}$ (see Theorem~\ref{thmcycfrE2}), we obtain the desired result. 
	\end{proof}

	If we combine Theorem~\ref{thmmain} with \cite[Corollary~3.1]{mwcenter}, we obtain:
	\begin{corollary}\label{corollarysnsphere}
		For any pivotal finite tensor category $\cat{C}$,
		\begin{align}
			\SNC(\mathbb{S}^2) \cong \left\{ \begin{array}{cl} k \ , & \text{if}\ \cat{C} \ \text{is spherical,} \\
				0 \ , & \text{else.} \end{array}\right.
		\end{align}
	\end{corollary}

	\begin{example}\label{exsphH}
		Let $H$ be a finite-dimensional spherical Hopf algebra and $\Sigma$ a closed surface of genus $g$.
		Then
		\begin{align}
			\SN_{H\catf{-mod}}(\Sigma) \cong \Hom_{D(H)} (      D(H)_\text{adj}^{\otimes g} , k          )^* \ , 
		\end{align}
		where $D(H)$ is the Drinfeld double and $D(H)_\text{adj}$ is $D(H)$ equipped with the adjoint action.
		The object $D(H)_\text{adj}$ is the canonical end of $D(H)\catf{-mod}$, and since this category is unimodular~\cite[Proposition~4.5]{eno-d}, the object becomes self-dual~\cite[Theorem~4.10]{shimizuunimodular}. 
		This implies \begin{align} \SN_{H\catf{-mod}}(\Sigma) \cong \Hom_{D(H)} (     k, D(H)_\text{adj}^{\otimes g}        )^* \ . 
		\end{align} 
		The vector space $\Hom_{D(H)}(k,D(H)_\text{adj})\cong \Hom_{D(H)}(D(H)_\text{adj},k)$ is the center $Z(D(H))$, or equivalently the space of class functions of $D(H)$.
		 Since $\Hom_{D(H)}(k,D(H)_\text{adj})^{\otimes g} \subset \Hom_{D(H)}(k,D(H)_\text{adj}^{\otimes g})$, we can see $Z(D(H))^{\otimes g}$ as a subspace of $\SN_{H\catf{-mod}}(\Sigma)^*$. In combination with
		\begin{align} \dim \Hom_{D(H)} (     k, D(H)_\text{adj}^{\otimes g}        ) \le \dim D(H)^{\otimes g} = n^{2g}\end{align} with $n:=\dim H$, this gives us
		\begin{align}
			\dim Z(D(H)) ^ g \le \dim \SN_{H\catf{-mod}}(\Sigma) \le n^{2 g} \ . 
		\end{align}
	\end{example}

	\begin{example}\label{exgroupG}
		For a finite group $G$, denote by ${\vectwok}_G^d$ the finite tensor category of finite-dimensional $G$-graded vector spaces, with the pivotal structure given by a group morphism $d:G \to k^\times$
		(one could twist the associator with a 3-cocycle on $G$, but we refrain from doing that here). We suppress the field from the notation.
		The description of the pivotal structure via a group morphism $d:G\to k^\times$
		is standard and can be found e.g.\ in \cite[Example~1.7.3]{TV}. For the considerations made here, this description will not be relevant. It suffices to know that $\vectwok_G^d$ is spherical if and only if $d^2=1$. 
		The modular functor for $Z( \vectwok_G^d    )$ is discussed 
		in \cite[Example~2.2, 2.13 \& 3.4]{mwcenter}. Thanks to Theorem~\ref{thmmain}, we know that it is equivalent to the string-net modular functor $\SN _{\vectwok_G^d}$, a fact that we can now exploit: If $G$ is abelian, it implies that $\SN _{\vectwok_G^d}(\Sigma)$ for a closed surface of genus $g$ is $|G|^{2g}$-dimensional if $d^{\chi(\Sigma)}=1$, with $\chi(\Sigma)=2-2g$ being the Euler characteristic of $\Sigma$. Otherwise, it is the zero vector space. 
		Hence, this modular functor is always non-trivial on the torus; this is clear because the space of conformal blocks of the torus must be the $k$-linear span of the $|G|^{2}$ many simple object of $Z( \vectwok_G^d    )$
		\cite[Corollary~6.6]{mwansular}.
		For every other closed surface of genus $g \neq 1$, we can choose appropriate input data to make the space of conformal blocks at genus $g$ zero while still maintaining an overall non-trivial modular functor:
		If $g=0$, we choose $G$ such that we can arrange $d^2 \neq 1$, thereby making $Z( \vectwok_G^d    )$ non-spherical
		(see Corollary~\ref{corollarysnsphere}).
		If $g\ge 2$, we choose $G=\mathbb{Z}_n = \mathbb{Z} / n\mathbb{Z}$ with $n=-\chi(\Sigma)+1 >0$, $k=\mathbb{C}$ and
		\begin{align}  d(\ell) = \exp \left( 
			\frac{2\pi \text{i} \ell}{n}\right) \quad \text{for}\quad \ell \in \mathbb{Z}_n\ . 
		\end{align}
		Then $d^{\chi(\Sigma)}(1)\neq 1$ and hence $\SN _{\vectwok_G^d}(\Sigma)=0$. 
	\end{example}

	One insight from Example~\ref{exgroupG} is the following:
	
	\begin{corollary}\label{corzero}
		Let $\Sigma$ be a closed surface of genus $g\neq 1$.
		Then there exists a pivotal finite tensor category $\cat{C}$ such that $\SNC(\Sigma)=0$, even though the modular functor $\SNC$ is still overall non-trivial, i.e.\ non-zero on some other closed surface.
	\end{corollary}

	Generally, modular functors are defined over an extension of the modular surface operad.
	For the string-nets, this extension is obviously not needed because the mapping class group acts directly; in other words, the string-net modular functor is \emph{anomaly-free}.
	Thanks to Theorem~\ref{thmmain},
	the same is true for the Lyubashenko modular functor for a Drinfeld center $Z(\cat{C})$.
	Moreover, the modular functor for $Z(\cat{C})$ extends in fact to an open-closed modular functor simply because this is true for the string-net modular functor (Theorem~\ref{thmlexmf}). Let us summarize this:
	
	\begin{corollary}\label{coranomalyopenclosed}
		For a pivotal finite tensor category $\cat{C}$, the Lyubashenko modular functor for the Drinfeld center
		$Z(\cat{C})$ \begin{pnum}\item is anomaly-free \item and 
		extends to an open-closed modular functor sending the open boundary to $\cat{C}$.\label{coranomalyopenclosedii} \end{pnum}
		\end{corollary}

	If $\cat{C}$ is a non-spherical fusion category (in particular, $\cat{C}$ is then semisimple), then the string-net construction discussed in this paper agrees with the one given by Runkel in~\cite{Non-s-string-nets}. When we apply Theorem~\ref{thmmain} to this special case, we find:

	\begin{corollary}\label{corsnr} The string-nets for a non-spherical fusion category $\cat{C}$
		defined by Runkel extend to an open-closed modular functor that sends 
			 the circle to $Z(\cat{C})$, but with a non-rigid Grothendieck-Verdier duality in which $\alpha$ is the dualizing object.
		\end{corollary}
	
	\begin{remark}
		Corollary~\ref{corconformalblocks} tells us that for the disk and one variable boundary label the non-spherical string-net construction gives us the functor $Z(\cat{C})(-,\alpha^{-1})^*$, so that $\alpha^{-1}$ acts as a background charge in the sense of \cite[Remark~7.2]{Non-s-string-nets}.
		\end{remark}

	\begin{corollary}[$\text{originally~\cite{balsamkirillov,balsam2,balsam,tuvi}}$]\label{cortvrt}
		For a spherical fusion category $\cat{C}$ with $\dim (\cat{C})\neq 0$,
		\begin{align}Z^\text{TV}_\cat{C} \simeq Z^\text{RT}_{Z(\cat{C})}\end{align}
		as once-extended three-dimensional topological field theories.
		\end{corollary}
	
	\begin{proof}[\slshape New proof  using string-nets, assuming~\cite{BDSPV15}]
		By \cite{BDSPV15} it suffices to prove that the balanced braided categories obtained by evaluation
		of both field theories on the circle are equivalent.
		These categories are necessarily modular fusion categories, again by \cite{BDSPV15}, so we may equivalently ask them to be equivalent as ribbon Grothendieck-Verdier categories~\cite[Corollary~4.4]{mwcenter}. 
		For $Z^\text{RT}_{Z(\cat{C})}$, this modular fusion category is $Z(\cat{C})$;
		for $Z^\text{TV}_\cat{C}$ we obtain, thanks to \cite{bartlett}, $\SNC(\mathbb{S}^1)$. Strictly speaking, \cite{bartlett} does not give us this comparison as balanced braided categories as explained in the introduction, but it does follow if we assume the results of \cite{BDSPV15} and use additionally \cite{bartlettgoosen}.
		Now the claim follows from Theorem~\ref{thmcycfrE2}.
		\end{proof}
	
	\begin{remark}[Transfer of knowledge to and from admissible skein modules in dimension two]\label{remasm2}
		The relation to admissible skein modules that was explained in Remark~\ref{relasm} can be used for a substantial transfer of results between \cite{asm} and our paper, simply because we can relate the \emph{constructions}, but have almost no intersection as far as \emph{results} about the constructions are concerned. 
		We do not want to expand on this comparison too much here and just state the main points.
		From the results in \cite{asm}, we learn the following facts about the string-net spaces:
		\begin{itemize}
			\item As explained in Remark~\ref{relasm}, the finite-dimensionality of string-net spaces could have been deduced from \cite[Section~5.4]{asm}.

			\item From \cite[Theorem~3.1]{asm}, we learn that the string-net space $\SNC(\mathbb{S}^2)$ is (dual to)
			the space of two-sided modified traces on $\Proj \cat{C}$. This space is $k$ if $\cat{C}$ is spherical, zero otherwise. If we combine these two insights, we recover  the algebraic result \cite[Corollary~6.11]{shibatashimizu} that a two-sided modified trace on $\Proj \cat{C}$ exists if and only if $\cat{C}$ is spherical.
			\end{itemize}
		Conversely, we learn the following facts about admissible skein modules through our results on $\SNC$:
		\begin{itemize}
			\item From Theorem~\ref{thmexcision}, it follows that admissible skein modules in dimension two satisfy excision.
			
			\item From Theorem~\ref{thmmain}, we learn that the admissible skein modules in dimension two for $\Proj \cat{C}$ extend, after a finite free cocompletion,
			 to an open-closed modular functor, namely the Lyubashenko modular functor for $Z(\cat{C})$ that has an open-closed extension thanks to Corollary~\ref{coranomalyopenclosed}~\ref{coranomalyopenclosedii}.

			\end{itemize}
		\end{remark}
	
	\small	
\newcommand{\etalchar}[1]{$^{#1}$}

\vspace*{0.3cm} \noindent  \textsc{Perimeter Institute,  N2L 2Y5 Waterloo, Canada} \\ \textsc{Fachbereich Mathematik, Universität Hamburg,  Bereich Algebra und Zahlentheorie, Bundesstraße 55, D-20146 Hamburg, Germany} \\
\textsc{Université Bourgogne Europe, CNRS, IMB UMR 5584, F-21000 Dijon, France} \\
\textsc{Technische Universität München,
	Department of Mathematics,
	Boltzmannstraße 3,
	D-85748 Garching bei München, 
	Germany}

\end{document}